\documentclass[letterpaper,11pt,twoside,keywordsasfootnote,addressatend,noinfoline]{article}

\usepackage{imsart}
\usepackage{fullpage}
\usepackage[english]{babel}
\usepackage{amssymb}
\usepackage{amsmath}
\usepackage{amsthm}
\usepackage{epsfig}
\usepackage{subfigure}
\usepackage{mathrsfs}
\usepackage{color}

\newtheorem{theorem}{Theorem}
\newtheorem{lemma}[theorem]{Lemma}

\newtheorem{proposition}[theorem]{Proposition}

\renewenvironment{proof}{\noindent{\bf Proof.}}{\hspace*{2mm}~$\square$}
\newenvironment{proofof}[1]{\noindent{\bf Proof of #1.}}{\hspace*{2mm}~$\square$}

\newcommand{\A}{\mathscr{A}}
\newcommand{\B}{\mathscr{B}}
\newcommand{\C}{\mathscr{C}}
\newcommand{\Z}{\mathbb{Z}}
\newcommand{\R}{\mathbb{R}}
\newcommand{\lat}{\mathscr{L}}

\newcommand{\norm}[1]{|\!|#1|\!|}
\newcommand{\ind}{\mathbf{1}}
\newcommand{\ep}{\epsilon}
\newcommand{\n}{\hspace*{-5pt}}
\newcommand{\bul}[1]{#1\,-\,$\bullet$}

\newcommand{\arrow}[1]{\overset{\hbox{\tiny #1}}{\longrightarrow}}

\newcommand{\cp}{\textsc{cp}}
\newcommand{\irp}{\textsc{irp}}
\newcommand{\cpa}{\textsc{cpa}}
\newcommand{\irpa}{\textsc{irpa}}

\DeclareMathOperator{\poisson}{Poisson}
\DeclareMathOperator{\card}{card}

\begin{document}

\begin{frontmatter}
\title{The contact process with social clusters \\ and asymptomatic states}
\runtitle{The contact process with social clusters and asymptomatic states}
\author{Karima Belatreche, Lamia Belhadji and Nicolas Lanchier}
\runauthor{Karima Belatreche, Lamia Belhadji and Nicolas Lanchier}
\address{}
\maketitle

\begin{abstract} \
 This paper is concerned with a natural variant of the contact process that places a social cluster of individuals, rather than a single individual, at each lattice point, and distinguishes between symptomatic and asymptomatic individuals. In particular, the epidemic model depends on four infection parameters: the external~(inter-cluster) and internal~(intra-cluster) infection rates for both symptomatic and asymptomatic individuals. Newly infected individuals begin as asymptomatic before transitioning to the symptomatic state at a spontaneous mutation rate, and infected individuals recover at a baseline rate of one. We study the phase structure of the deterministic nonspatial mean-field model when each cluster reduces to a pair of individuals, and prove that the stochastic spatial model exhibits the same phase transitions with respect to the four infection rates for all cluster sizes. We also prove the existence of a phase transition driven by the cluster size itself for the spatial model.
\end{abstract}



\end{frontmatter}


\section{Introduction}
\label{sec:intro}
 Like the original model of Kermack and McKendrick~\cite{kermack_mckendrick_1927}, most mathematical models in epidemiology consist of systems of coupled ordinary differential equations.
 Diseases, however, spread by contact, showing the importance of models based on interacting particle systems that account for the presence of an underlying network of interactions, or social network in the context of human populations.
 The first example of such a model is Harris' contact process~\cite{harris_1974} in which the infection spreads by contact through the edges of the~$d$-dimensional integer lattice.
 Since Harris' original paper, a number of variants of the contact process modeling the spread of a disease have been introduced, and we refer to~\cite[Chapter~2]{lanchier_2024} for an overview.
 In this spirit, the objective of this paper is to study an extension of the contact process that accounts for the presence of social clusters and distinguishes between symptomatic and asymptomatic individuals. \\
\indent
 In the contact process, each site~$x \in \Z^d$ represents an individual who is either healthy or infected.
 Infected individuals recover at rate one, and interact with each of their~$2d$ neighbors at rate~$\lambda$, which results in healthy neighbors becoming infected.
 More precisely, the state at time~$t$ is a configuration~$\xi_t : \Z^d \to \{0, 1 \}$, where~0 = healthy and~1 = infected, and the process evolves according to the continuous-time Markov chain whose transitions at site~$x$ are given by
 $$ \begin{array}{rclclcrclcl}
      0 & \n \to \n & 1 & \hbox{at rate} & \lambda \,\sum_{y \sim x} \xi_t (y) & \quad \hbox{and} \quad & 1 & \n \to \n & 0 & \hbox{at rate} & 1, \end{array} $$
 where~$y \sim x$ means that site~$y$ is one of the~$2d$ nearest neighbors of site~$x$.
 This process will be referred to as~$\cp (\lambda)$ later.
 The main result about the contact process on the lattice is the existence of a unique phase transition at a nondegenerate critical value:
 there exists~$\lambda_c > 0$ that depends on the spatial dimension such that the process dies out when~$\lambda < \lambda_c$ but survives when~$\lambda > \lambda_c$.
 After proving that a proper rescaling of the process dominates oriented site percolation, Bezuidenhout and Grimmett~\cite{bezuidenhout_grimmett_1990} also deduced several important results: extinction at the critical value, complete convergence theorem, and shape theorem. \\
\indent
 Schinazi~\cite{schinazi_2002} introduced a variant of the contact process in which each lattice point represents a social cluster of individuals rather than a single individual.
 The process is characterized by an external infection rate~$\lambda$, the rate at which infected individuals infect nearby healthy clusters, as well as an internal infection rate~$\phi$, the rate at which infected individuals infect healthy individuals in their own social cluster.
 Belhadji and Lanchier~\cite{belhadji_lanchier_2005} studied a variant of his model, called the individual recovery process, in which individuals, rather than social clusters, recover independently at rate one.
 In this case, the state at time~$t$  is a configuration
 $$ \xi_t : \Z^d \to \{0, 1, 2, \ldots, N \} \quad \hbox{where} \quad N = \hbox{cluster size}, $$
 that keeps track of the number of infected individuals in each cluster, and the possible transitions of the process at cluster~$x$ are given by
 $$ \begin{array}{rclclrclcll}
      0 & \n \to \n & 1 & \hbox{at rate} & \lambda \,\sum_{y \sim x} \xi_t (y), & \quad
      i & \n \to \n & i + 1 & \hbox{at rate} & i \phi & \hbox{for} \ i < N, \vspace*{4pt} \\ &&&&&
      i & \n \to \n & i - 1 & \hbox{at rate} & i      & \hbox{for} \ i \leq N. \end{array} $$
 This process will be denoted by~$\irp (\lambda, \phi, N)$ later.
 Belhadji, Lanchier and Mercer~\cite{belhadji_lanchier_mercer_2024} introduced more recently another variant of the contact process that distinguishes infected individuals depending on whether they are symptomatic or asymptomatic.
 The process is characterized by the infection rate~$\lambda_1$ of asymptomatic individuals, the infection rate~$\lambda_2$ of symptomatic individuals, and the rate~$\gamma$ at which asymptomatic individuals become symptomatic.
 In this case, the state at time~$t$  is a configuration~$\xi_t : \Z^d \to \{0, 1, 2 \}$, where~0 = healthy,~1 = asymptomatic, and~2 = symptomatic, and the transition rates of the process at site~$x$ are given by
 $$ \begin{array}{rclclrclcl}
      0 & \n \to \n & 1 & \hbox{at rate} & \sum_{y \sim x} (\lambda_1 \ind \{\xi_t (y) = 1 \} + \lambda_2 \ind \{\xi_t (y) = 2 \}), & \quad
      1 & \n \to \n & 2 & \hbox{at rate} & \gamma, \vspace*{4pt} \\
      1 & \n \to \n & 0 & \hbox{at rate} & 1, & \quad
      2 & \n \to \n & 0 & \hbox{at rate} & 1. \end{array} $$
 This process will be denoted by~$\cpa (\lambda_1, \lambda_2, \gamma)$ later.
 We refer the reader to~\cite[Section~2.4.4]{lanchier_2024} for a review of the individual recovery process, and to~\cite[Section~2.1.5]{lanchier_2024} for a review of the contact process with an asymptomatic state. \\
\indent
 The epidemic model we study in this paper combines~\cite{belhadji_lanchier_2005, belhadji_lanchier_mercer_2024} by placing a social cluster at each lattice point, and by distinguishing between symptomatic and asymptomatic individuals.
 In particular, the process now depends on four infection rates: an external infection rate and an internal infection rate from both symptomatic and asymptomatic individuals.
 The process is also characterized by the cluster size~$N$ like model~\cite{belhadji_lanchier_2005} and the mutation rate~$\gamma$ from asymptomatic to symptomatic like model~\cite{belhadji_lanchier_mercer_2024}.
 More precisely, the state at time~$t$ is now
 $$ \xi_t : \Z^d \to \{(i, j) \in \Z_+ \times \Z_+ : i + j \leq N \}, $$
 where~$\xi_t (x) = (i, j)$ means that the social cluster~$x$ has~$i$ infected asymptomatic and~$j$ infected symptomatic individuals, and the transition rates at cluster~$x$ are given by
 $$ \begin{array}{rclcl}
      (0, 0) & \n \to \n & (1, 0) & \hbox{at rate} & \sum_{y \sim x} \sum_{i, j} \,(i \lambda_1 + j \lambda_2) \,\ind \{\xi_t (y) = (i, j) \}, \vspace*{4pt} \\
      (i, j) & \n \to \n & (i + 1, j) & \hbox{at rate} & i \phi_1 + j \phi_2 \quad \hbox{for} \ i + j < N, \vspace*{4pt} \\
      (i, j) & \n \to \n & (i - 1, j + 1) & \hbox{at rate} & i \gamma \quad \hbox{for} \ i + j \leq N, j < N, \vspace*{4pt} \\
      (i, j) & \n \to \n & (i - 1, j) & \hbox{at rate} & i \quad \hbox{for} \ i + j \leq N, \vspace*{4pt} \\
      (i, j) & \n \to \n & (i, j - 1) & \hbox{at rate} & j \quad \hbox{for} \ i + j \leq N. \end{array} $$
 The first transition indicates that healthy cluster get infected at the external infection rate~$\lambda_1$ from each nearby asymptomatic individual, and at the external infection rate~$\lambda_2$ from each nearby symptomatic individual.
 Once a cluster is infected, additional infections originate from the cluster, and the second transition indicates that healthy individuals now get infected at the internal infection rate~$\phi_1$ from each asymptomatic individual in the cluster, and at the internal infection rate~$\phi_2$ from each symptomatic individual in the cluster.
 Regardless of the origin of the infection, newly infected individuals are always asymptomatic, and the third transition indicates that these individuals become symptomatic at rate~$\gamma$.
 Finally, the last two transitions indicate that infected individuals recover at rate one, regardless of whether they are asymptomatic or symptomatic.
 From now on, we assume that the process starts from a translation-invariant distribution with a positive density of~1s and a positive density of~2s, and we say that there is
 $$ \begin{array}{rcrcl}
      \hbox{survival} & \hbox{when} & \liminf_{t \to \infty} P (\xi_t (x) \neq 0) > 0 & \hbox{for all} & x \in \Z^d, \vspace*{4pt} \\
    \hbox{extinction} & \hbox{when} &    \lim_{t \to \infty} P (\xi_t (x) \neq 0) = 0 & \hbox{for all} & x \in \Z^d. \end{array} $$
 Before investigating the survival phase and extinction phase of the stochastic spatial model, we will study its deterministic nonspatial counterpart~(mean-field model).
 This model consists of a system of coupled differential equations for the densities~$u_{ij}$ of clusters with~$i$ asymptomatic individuals and~$j$ symptomatic individuals,
 obtained under the assumption that the population is homogeneously mixing.
 Defining survival in the context of the mean-field model as the existence of a fixed point with a positive density of infected sites, at least for social clusters of size two, we will prove the following necessary and sufficient condition for survival.
\begin{proposition}
\label{prop:MF}
 When~$N = 2$, the mean-field model survives if and only if
 $$ (\lambda_1 + \gamma \lambda_2)((1 + \phi_1)(1 + \phi_2) + (1 + \phi_1) + \gamma (1 + \phi_2)) > \gamma \phi_1 + \phi_2 + (1 + \gamma)(2 + \gamma). $$
\end{proposition}
\noindent
 The proposition implies that, at least when~$N = 2$, the system is monotone with respect to the infection rates in the sense that, if there is survival for a given set of parameters, increasing the infection rates still results in survival, and by contraposition, if there is extinction for a given set of parameters, decreasing the infection rates still results in extinction.
 The same result is expected for the spatial model, but using standard coupling methods to compare processes with different infection rates only shows that the process is monotone with respect to the infection rates in the parameter region where~$\lambda_1 \leq \lambda_2$ and~$\phi_1 \leq \phi_2$.
 However, our analysis shows that the phase transitions following from Proposition~\ref{prop:MF} for the mean-field model exist as well for the spatial model, and we now look at the effect of each parameter to explain these results. \vspace*{5pt} \\
{\bf External infection rates}.
 It follows from the proposition that the mean-field model survives when the external infection rates are larger than one, or when this holds for the asymptomatic external infection rate and~$\gamma$ is small or for the symptomatic external infection rate and~$\gamma$ is large.
 The same results hold for the spatial model regardless of the size of the social clusters after replacing the threshold one with the critical value~$\lambda_c$ of the contact process.
\begin{theorem}
\label{th:lambdalarge}
 For all~$\phi_1, \phi_2 \geq 0$, there is survival when
\begin{itemize}
\item[a.]
 $\lambda_1, \lambda_2 > \lambda_+$ for some~$\lambda_+ \leq \lambda_c$, \vspace*{4pt}
\item[b.]
 $\lambda_1 > \lambda_c$ and~$\gamma < \gamma_-$ for some~$\gamma_- > 0$, \vspace*{4pt}
\item[c.]
 $\lambda_2 > \lambda_c$ and~$\gamma > \gamma_+$ for some~$\gamma_+ < \infty$.
\end{itemize}
\end{theorem}
\noindent
 Indeed, when both external infection rates are larger than~$\lambda_c$, the process dominates a supercritical contact process.
 Similarly, when~$\gamma$ is small, most of the infected individuals are asymptomatic so the process essentially dominates a contact process with parameter~$\lambda_1$ and survives when~$\lambda_1 > \lambda_c$, whereas when~$\gamma$ is large, most of the infected individuals are symptomatic so the process essentially dominates a contact process with parameter~$\lambda_2$ and survives when~$\lambda_2 > \lambda_c$.
 It also follows from the proposition that the mean-field model dies out when the external infection rates are small, or when this holds for the asymptomatic external infection rate and~$\gamma$ is small or for the symptomatic external infection rate and~$\gamma$ is large.
 The spatial model exhibits a similar behavior regardless of the size of the social clusters, as stated in the following theorem.
\begin{theorem}
\label{th:lambdasmall}
 For all~$\phi_1, \phi_2 \geq 0$, there is extinction when
\begin{itemize}
\item[a.]
 $\lambda_1, \lambda_2 < \lambda_-$ for some~$\lambda_- > 0$, \vspace*{4pt}
\item[b.]
 $\lambda_1 < \lambda_-$ and~$\gamma < \gamma_-$ for some~$\lambda_- > 0$ and~$\gamma_- > 0$, \vspace*{4pt}
\item[c.]
 $\lambda_2 < \lambda_-$ and~$\gamma > \gamma_+$ for some~$\lambda_- > 0$ and~$\gamma_+ < \infty$.
\end{itemize}
\end{theorem}
\noindent
 The intuition behind these results is that, even if the cluster size and the internal infection rates are large, infected clusters fully recover after a finite random time, so if the external infections are rare enough, the infection dies out.
 The last two parts of the theorem follow from a similar argument using again that, when~$\gamma$ is small, most of the infected individuals are asymptomatic, whereas when~$\gamma$ is large, most of the infected individuals are symptomatic. \vspace*{5pt} \\
\noindent {\bf Internal infection rates}.
 In view of the previous results, the next natural step is to study the existence of a phase transition with respect to the internal infection rates when the external infection rates are less than one.
 The proposition implies that the mean-field model survives when the external infection rates are positive and the internal infection rates are large enough, or when this holds for the asymptomatic infection rates and~$\gamma$ is small or for the symptomatic infection rates and~$\gamma$ is large.
 At least when~$N > 1$, the same results hold for the spatial model.
\begin{theorem}
\label{th:philarge}
 For all~$N > 1$, there is survival when
\begin{itemize}
\item[a.]
 $\lambda_1, \lambda_2 > 0$ and~$\phi_1, \phi_2 > \phi_+$ for some~$\phi_+ < \infty$, \vspace*{4pt}
\item[b.]
 $\lambda_1 > 0$ and~$\phi_1 > \phi_+$ and~$\gamma < \gamma_-$ for some~$\phi_+ < \infty$ and~$\gamma_- > 0$, \vspace*{4pt}
\item[c.]
 $\lambda_2 > 0$ and~$\phi_2 > \phi_+$ and~$\gamma > \gamma_+$ for some~$\phi_+ < \infty$ and~$\gamma_+ < \infty$.
\end{itemize}
\end{theorem}
\noindent
 The idea in this case is that, the time infected clusters remain infected can be made arbitrarily large by choosing the internal infection rates large, so even if the external infections are rare, the infection can slowly expand.
 Looking now at the extinction phase, the mean-field model dies out when the external infection rates are smaller than one and the internal infection rates are small enough, or when this holds for the asymptomatic infection rates and~$\gamma$ is small or for the symptomatic infection rates and~$\gamma$ is large.
 The same results again hold for the spatial model regardless of the size of the clusters, as stated in the following theorem.
\begin{theorem}
\label{th:phismall}
 There is extinction when
\begin{itemize}
\item[a.]
 $\lambda_1, \lambda_2 < \lambda_c$ and~$\phi_1, \phi_2 < \phi_-$ for some~$\phi_- > 0$, \vspace*{4pt}
\item[b.]
 $\lambda_1 < \lambda_c$ and~$\phi_1 < \phi_-$ and~$\gamma < \gamma_-$ for some~$\phi_- > 0$ and~$\gamma_- > 0$, \vspace*{4pt}
\item[c.]
 $\lambda_2 < \lambda_c$ and~$\phi_2 < \phi_-$ and~$\gamma > \gamma_+$ for some~$\phi_- > 0$ and~$\gamma_+ < \infty$.
\end{itemize}
\end{theorem}
\noindent
 Intuitively, this result follows from the fact that, when the internal infection rates are small, most clusters have at most one infected individuals so the process is essentially dominated by a subcritical contact process, which leads to extinction. \vspace*{5pt} \\
\noindent {\bf Cluster size}.
 Our analysis of the mean-field model reduces to the~$N = 2$ case, but our analysis of the spatial model shows again a phase transition with respect to the cluster size.
 First, we prove survival when the external infection rates are positive and the internal infection rates larger than one, or when this holds for the asymptomatic infection rates and~$\gamma$ is small or for the symptomatic infection rates and~$\gamma$ is large.
\begin{theorem}
\label{th:Nlarge}
 There is survival when
\begin{itemize}
\item[a.]
 $\lambda_1, \lambda_2 > 0$ and~$\phi_1, \phi_2 > 1$ and~$N > N_+$ for some~$N_+ < \infty$, \vspace*{4pt}
\item[b.]
 $\lambda_1 > 0$ and~$\phi_1 > 1$ and~$N > N_+$ and~$\gamma < \gamma_-$ for some~$N_+ < \infty$ and~$\gamma_- > 0$, \vspace*{4pt}
\item[c.]
 $\lambda_2 > 0$ and~$\phi_2 > 1$ and~$N > N_+$ and~$\gamma > \gamma_+$ for some~$N_+ < \infty$ and~$\gamma_+ < \infty$.
\end{itemize}
\end{theorem}
\noindent
 The intuition behind the result is that, when the internal infection rates are larger than one, and so larger than the recover rate, the time infected clusters remain infected can be made arbitrarily large by choosing the cluster size large, so even if the external infections are rare, the infection can slowly expand.
 In contrast, when the external infection rates are less than~$\lambda_c$, or when this holds for the asymptomatic infection rates and~$\gamma$ is small or for the symptomatic infection rates and~$\gamma$ is large, there is extinction when the cluster size is small enough.
\begin{theorem}
\label{th:Nsmall}
 There is extinction when
\begin{itemize}
\item[a.]
 $\lambda_1, \lambda_2 < \lambda_c$ and~$N < N_-$ for some~$N_- > 1$, \vspace*{4pt}
\item[b.]
 $\lambda_1 < \lambda_c$ and~$N < N_-$ and~$\gamma < \gamma_-$ for some~$N_- > 1$ and~$\gamma_- > 0$, \vspace*{4pt}
\item[c.]
 $\lambda_2 < \lambda_c$ and~$N < N_-$ and~$\gamma > \gamma_+$ for some~$N_- > 1$ and $\gamma_+ < \infty$.
\end{itemize}
\end{theorem}
\noindent
 The idea here is that, in the~$N = 1$ case where the clusters reduce to a single individual, the process reduces to a subcritical contact process, and so dies out. \vspace*{5pt} \\
\noindent {\bf Mutation rate}.
 Intuitively, when the asymptomatic infection rates are high but the symptomatic infection rates low, there should be survival for~$\gamma$ small but extinction for~$\gamma$ large, whereas when the asymptomatic infection rates are low but the symptomatic infection rates high, there should be extinction for~$\gamma$ small but survival for~$\gamma$ large.
 In particular, neither the mean-field model nor the spatial model should be monotone with respect to~$\gamma$.
 In fact, from the proposition, we can deduce the existence of infection rates for which there are two phase transitions with respect to the mutation rate:
 extinction for small values of~$\gamma$, then survival for intermediate values of~$\gamma$, and then again extinction for large values of~$\gamma$~(see Figure~\ref{fig:MF2} below).
 This confirms the lack of monotonicity for the mean-field model.
 We believe that an analogous result holds for the spatial model. However, to keep the present paper concise, we leave the study of this case to future work. \vspace*{5pt} \\
 The rest of the paper is devoted to the proofs of Proposition~\ref{prop:MF} and Theorems~\ref{th:lambdalarge}--\ref{th:Nsmall}.
 We first study the mean-field model, proving the proposition as well as important implications on the phase structure~(section~\ref{sec:MF}).
 Then, we focus on the interacting particle system, explaining first how to construct the process graphically and couple/compare the process with the individual recovery process and the contact process with an asymptomatic state~(section~\ref{sec:GR}).
 These couplings are used to prove some of the theorems~(section~\ref{sec:irp-cpa}), while the remaining theorems require the use of various block constructions and perturbation techniques.
 Also, before proving these more intricate theorems, we collect a few percolation results~(section~\ref{sec:perco}) used later in the block constructions, and show some general perturbation results for~$\gamma$ small and~$\gamma$ large~(section~\ref{sec:perturb}).
 Finally, we prove the main survival results by coupling the process properly rescaled in space and time with a certain type of percolation process~(sections~\ref{sec:lambdalarge}--\ref{sec:Nlarge}), and the main extinction results by coupling the process with another type of percolation process~(sections~\ref{sec:lambdasmall}--\ref{sec:phismall}).
 Depending on the parameter region under consideration, these block constructions also involve coupling arguments, duality results, random-walk estimates, martingale techniques, and large deviations theory.
 

\section{Mean-field model}
\label{sec:MF}
 In this section, we study the mean-field model of the individual recovery process with an asymptomatic state and clusters of size two, prove Proposition~\ref{prop:MF}, and deduce several other results, including the analogs of Theorems~\ref{th:lambdalarge}--\ref{th:Nsmall}.
 Assuming that the population is homogeneously mixing, and letting~$u_{ij}$ be the density of clusters with~$i$ asymptomatic individuals and~$j$ symptomatic individuals, the process is described by the following system of coupled differential equations:
\begin{equation}
\label{eq:MF1}
\begin{array}{rcl}
  u_{00}' & \n = \n & - \,(\lambda_1 u_{10} + \lambda_2 u_{01} + 2 \lambda_1 u_{20} + (\lambda_1 + \lambda_2) u_{11} + 2 \lambda_2 u_{02}) u_{00}
            \vspace*{4pt} \\ && \hspace*{20pt} + \,u_{10} + u_{01}, \vspace*{4pt} \\
  u_{10}' & \n = \n & + \,(\lambda_1 u_{10} + \lambda_2 u_{01} + 2 \lambda_1 u_{20} + (\lambda_1 + \lambda_2) u_{11} + 2 \lambda_2 u_{02}) u_{00}
            \vspace*{4pt} \\ && \hspace*{20pt} - \,(1 + \phi_1 + \gamma) u_{10} + 2 u_{20} + u_{11}, \vspace*{4pt} \\
  u_{01}' & \n = \n & + \,\gamma u_{10} - (1 + \phi_2) u_{01} + u_{11} + 2 u_{02}, \vspace*{4pt} \\
  u_{20}' & \n = \n & + \,\phi_1 u_{10} - 2 (1 + \gamma) u_{20}, \vspace*{4pt} \\
  u_{11}' & \n = \n & + \,\phi_2 u_{01} + 2 \gamma u_{20} - (2 + \gamma) u_{11}, \vspace*{4pt} \\
  u_{02}' & \n = \n & + \,\gamma u_{11} - 2 u_{02}. \end{array}
\end{equation}
 The expression of the derivatives can be understood from Figure~\ref{fig:states}, which shows how the states are connected and the~(spontaneous) transition rates.
\begin{figure}[t!]
\centering
\scalebox{0.40}{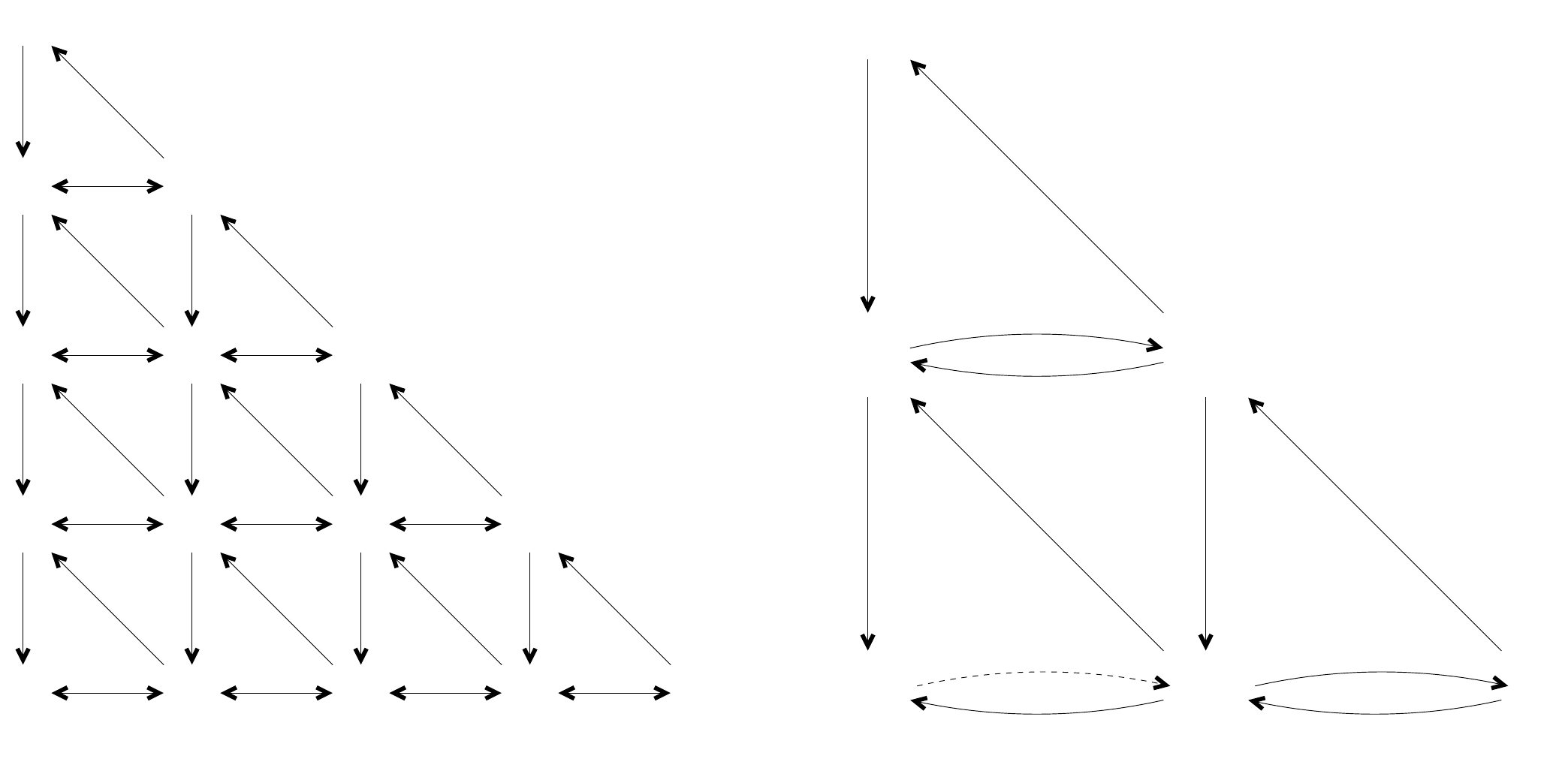}
\caption{\upshape{
 Left: Picture of all the possible transitions when~$N = 4$.
 Arrows~$\rightarrow$ correspond to an infection,~$\leftarrow$ to the recovery of an asymptomatic individual,~$\downarrow$ to the recovery of a symptomatic individual, and~$\nwarrow$ to a mutation from asymptomatic to symptomatic.
 Right: Transition rates when~$N = 2$ except for the transition~$00 \to 10$ which is density-dependent.
 In both pictures,~$ij$ is a short for state~$(i, j)$.}}
\label{fig:states}
\end{figure}
 The solution curves are contained in the standard 5-dimensional simplex, which cannot be visualized in a picture.
 Figures~\ref{fig:MF1} and~\ref{fig:MF2} below show instead a projection in the~$(u_1, u_2)$ plane of the solution curves where
 $$ \textstyle
     u_1 = \frac{1}{2} u_{10} + \frac{1}{2} u_{11} + u_{20} \quad \hbox{and} \quad
     u_2 = \frac{1}{2} u_{01} + \frac{1}{2} u_{11} + u_{02} $$
 represent the fraction of individuals that are asymptomatic/symptomatic. \\ \\
\begin{proofof}{Proposition~\ref{prop:MF}}
 Setting the right-hand side of the system~\eqref{eq:MF1} equal to zero, it follows from the fourth and last equations that
\begin{equation}
\label{eq:MF2}
  u_{20} = \frac{\phi_1}{2 (1 + \gamma)} \,u_{10} \quad \hbox{and} \quad u_{02} = \frac{\gamma}{2} \,u_{11}.
\end{equation}
 Substituting in the third and fifth equations of~\eqref{eq:MF1}, we get
 $$ \begin{array}{rcl}
                   u_{01}' = \gamma u_{10} - (1 + \phi_2) u_{01} + (1 + \gamma) u_{11} & \n = \n & 0, \vspace*{4pt} \\
      (1 + \gamma) u_{11}' = \gamma \phi_1 u_{10} + \phi_2 (1 + \gamma) u_{01} - (1 + \gamma)(2 + \gamma) u_{11} & \n = \n & 0, \end{array} $$
 from which we deduce the following:
\begin{equation}
\label{eq:MF3}
\begin{array}{l}
  \phi_1 u_{01}' - (1 + \gamma) u_{11}' = 0 \ \ \Longrightarrow \vspace*{4pt} \\ \hspace*{50pt}
 (\phi_1 + \phi_2 (1 + \phi_1 + \gamma)) u_{01} = (1 + \gamma)(2 + \phi_1 + \gamma) u_{11}, \vspace*{8pt} \\
 (2 + \gamma) u_{01}' + (1 + \gamma) u_{11}' = 0 \ \ \Longrightarrow \vspace*{4pt} \\ \hspace*{50pt}
 (2 + \phi_2 + \gamma) u_{01} = \gamma (2 + \phi_1 + \gamma) u_{10}. \end{array}
\end{equation}
 Letting~$\bar u_{01} = u_{01} / (2 \gamma (1 + \gamma)(2 + \phi_1 + \gamma))$, and using~\eqref{eq:MF3} then~\eqref{eq:MF2}, we obtain
\begin{equation}
\label{eq:MF4}
\begin{array}{rcl}
  u_{11} & \n = \n & \displaystyle \frac{\phi_1 + \phi_2 (1 + \phi_1 + \gamma)}{(1 + \gamma)(2 + \phi_1 + \gamma)} \,u_{01} =
                     \displaystyle  2 \gamma (\phi_1 + \phi_2 (1 + \phi_1 + \gamma)) \bar u_{01}, \vspace*{8pt} \\
  u_{10} & \n = \n & \displaystyle \frac{2 + \phi_2 + \gamma}{\gamma (2 + \phi_1 + \gamma)} \,u_{01} =
                     \displaystyle  2 (1 + \gamma)(2 + \phi_2 + \gamma) \bar u_{01}, \vspace*{8pt} \\
  u_{20} & \n = \n & \displaystyle \frac{\phi_1}{2 (1 + \gamma)} \,u_{01} =
                     \displaystyle  \phi_1 (2 + \phi_2 + \gamma) \bar u_{01}, \vspace*{8pt} \\
  u_{02} & \n = \n & \displaystyle \frac{\gamma}{2} \,u_{01} =
                     \displaystyle  \gamma^2 (\phi_1 + \phi_2 (1 + \phi_1 + \gamma)) \bar u_{01}. \end{array}
\end{equation}
 From the first equation in~\eqref{eq:MF1}, we get
 $$ \begin{array}{l}
     (\lambda_1 u_{10} + \lambda_2 u_{01} + 2 \lambda_1 u_{20} + (\lambda_1 + \lambda_2) u_{11} + 2 \lambda_2 u_{02}) \vspace*{4pt} \\ \hspace*{40pt}
      \times \,(1 - u_{10} - u_{01} - u_{20} - u_{11} - u_{02}) = u_{10} + u_{01}, \end{array} $$
 which, according to~\eqref{eq:MF4}, can be rewritten in the form
 $$ A u_{01} (1 - B u_{01}) = C u_{01} \quad \hbox{with} \quad A, B, C > 0. $$
 This implies that~$u_{01} > 0$ if and only if
 $$ u_{01} = \frac{1}{B} \bigg(1 - \frac{C}{A} \bigg) > 0 \quad \Longleftrightarrow \quad A > C. $$
 Using~\eqref{eq:MF4} to compute~$A u_{01}$ and simplifying, we get
 $$ \begin{array}{rcl}
      A u_{01} & \n = \n & \lambda_1 u_{10} + \lambda_2 u_{01} + 2 \lambda_1 u_{20} + (\lambda_1 + \lambda_2) u_{11} + 2 \lambda_2 u_{02} \vspace*{4pt} \\
               & \n = \n & 2 \lambda_1 (1 + \gamma)(2 + \phi_2 + \gamma) \bar u_{01} +
                           2 \gamma \lambda_2 (1 + \gamma)(2 + \phi_1 + \gamma) \bar u_{01} \vspace*{4pt} \\ && \hspace*{25pt} 
                       + \ 2 \lambda_1 \phi_1 (2 + \phi_2 + \gamma) \bar u_{01} +
                           2 \gamma (\lambda_1 + \lambda_2)(\phi_1 + \phi_2 (1 + \phi_1 + \gamma)) \bar u_{01} \vspace*{4pt} \\ && \hspace*{25pt}
                       + \ 2 \gamma^2 \lambda_2 (\phi_1 + \phi_2 (1 + \phi_1 + \gamma)) \bar u_{01} \vspace*{4pt} \\
              & \n = \n & 2 (1 + \gamma)(\lambda_1 + \gamma \lambda_2)((1 + \phi_1)(1 + \phi_2) + (1 + \phi_1) + \gamma (1 + \phi_2)) \bar u_{01}. \end{array} $$
 Similarly, using~\eqref{eq:MF4} to compute~$C u_{01}$, we get
 $$ \begin{array}{rcl}
      C u_{01} & \n = \n & u_{10} + u_{01} = 2 (1 + \gamma)(2 + \phi_2 + \gamma) \bar u_{01} + 2 \gamma (1 + \gamma)(2 + \phi_1 + \gamma) \bar u_{01} \vspace*{4pt} \\
               & \n = \n & 2 (1 + \gamma)(\gamma \phi_1 + \phi_2 + (1 + \gamma)(2 + \gamma)) \bar u_{01}. \end{array} $$
 In conclusion, at least when the clusters contain two individuals, there is an interior fixed point with a positive density of infected individuals~($A - C > 0$) if and only if
\begin{equation}
\label{eq:MF5}
\begin{array}{l}
 (\lambda_1 + \gamma \lambda_2)((1 + \phi_1)(1 + \phi_2) + (1 + \phi_1) + \gamma (1 + \phi_2)) \vspace*{4pt} \\ \hspace*{134pt}
                                                         - \,\gamma \phi_1 - \phi_2 - (1 + \gamma)(2 + \gamma) > 0. \end{array}
\end{equation}
 which proves the proposition.
\end{proofof} \\ \\
 Taking~$\gamma = 0$ to exclude symptomatic individuals, or~$\phi_1 = \phi_2 = 0$ to force the number of infected individuals per cluster to be at most one,~\eqref{eq:MF5} simplifies to
 $$ \lambda_1 (1 + \phi_1) > 1 \ \ \hbox{when} \ \ \gamma = 0 \quad \hbox{and} \quad \lambda_1 + \gamma \lambda_2 > 1 + \gamma \ \ \hbox{when} \ \ \phi_1 = \phi_2 = 0. $$
 The first inequality is the particular case for clusters of size two of the condition for survival of the individual recovery process derived in~\cite{belhadji_lanchier_2005}, while the second inequality is the condition for survival of the contact process with an asymptomatic state derived in~\cite{belhadji_lanchier_mercer_2024}. \vspace*{5pt} \\
{\bf External infection rates}.
 Because the left-hand side of~\eqref{eq:MF5} is increasing in~$\lambda_1, \lambda_2$, if there is survival, increasing~$\lambda_1$ and~$\lambda_2$ still results in survival.
 More generally, the form of the expression in~\eqref{eq:MF5} shows that the survival phase in the~$(\lambda_1, \lambda_2)$ plane consists of the region above the straight line connecting~$(\lambda_1^*, 0)$ and~$(0, \lambda_2^*)$, where
 $$ \begin{array}{rcl}
    \lambda_1^* & \n = \n & \displaystyle
    \frac{\gamma \phi_1 + \phi_2 + (1 + \gamma)(2 + \gamma)}{(1 + \phi_1)(1 + \phi_2) + (1 + \phi_1) + \gamma (1 + \phi_2)} \leq 1 + \gamma, \vspace*{8pt} \\
    \lambda_2^* & \n = \n & \displaystyle \frac{\lambda_1^*}{\gamma} =
    \frac{1}{\gamma} \frac{\gamma \phi_1 + \phi_2 + (1 + \gamma)(2 + \gamma)}{(1 + \phi_1)(1 + \phi_2) + (1 + \phi_1) + \gamma (1 + \phi_2)} \leq 1 + \frac{1}{\gamma}, \end{array} $$
 which is below and parallel to the straight line with equation~$\lambda_1 + \gamma \lambda_2 = 1 + \gamma$.
 The inequalities above follow from a direct calculation.
 This implies that
\begin{equation}
\label{eq:MF6}
\begin{array}{rcl} \lambda_1 + \gamma \lambda_2 > 1 + \gamma & \Longrightarrow & \hbox{survival} \ \ \forall \,\phi_1, \phi_2 \geq 0. \end{array}
\end{equation}
 Using also the continuity in the parameters, this shows that
\begin{equation}
\label{eq:MF7}
\begin{array}{rcl}
\lambda_1, \lambda_2 > 1 & \Longrightarrow & \hbox{survival} \ \ \forall \,\gamma \geq 0, \vspace*{4pt} \\
           \lambda_1 > 1 & \Longrightarrow & \hbox{survival} \ \ \forall \,\gamma > 0 \ \hbox{small}, \vspace*{4pt} \\
           \lambda_2 > 1 & \Longrightarrow & \hbox{survival} \ \ \forall \,\gamma < \infty \ \hbox{large}, \end{array}
\end{equation}
 which is the analog of Theorem~\ref{th:lambdalarge} with the critical value~1 instead of~$\lambda_c$.
 To better understand the phase structure of the mean-field model, we define
 $$ \Lambda_{11} = \lambda_1 (1 + \phi_1), \quad \Lambda_{12} = \lambda_1 (1 + \phi_2), \quad \Lambda_{21} = \lambda_2 (1 + \phi_1), \quad \Lambda_{22} = \lambda_2 (1 + \phi_2). $$
 Assuming that~$\Lambda_{11}, \Lambda_{22} > 1$, the AM-GM inequality implies that
 $$ \Lambda_{12} + \Lambda_{21} \geq 2 \sqrt{\Lambda_{12} \Lambda_{21}} = 2 \sqrt{\Lambda_{11} \Lambda_{22}} > 2, $$
 from which it follows that
 $$ \begin{array}{l}
      (\lambda_1 + \gamma \lambda_2)((1 + \phi_1)(1 + \phi_2) + (1 + \phi_1) + \gamma (1 + \phi_2)) \vspace*{4pt} \\ \hspace*{20pt} = \
       \lambda_1 (1 + \phi_1)(2 + \phi_2) + \gamma \lambda_2 (1 + \phi_2)(1 + \gamma + \phi_1) + \gamma (\lambda_1 (1 + \phi_2) + \lambda_2 (1 + \phi_1)) \vspace*{4pt} \\ \hspace*{20pt} = \
      (2 + \phi_2) \Lambda_{11} + \gamma (1 + \gamma + \phi_1) \Lambda_{22} + \gamma (\Lambda_{12} + \Lambda_{21}) \vspace*{4pt} \\ \hspace*{20pt} > \
      (2 + \phi_2) + \gamma (1 + \gamma + \phi_1) + 2 \gamma = \gamma \phi_1 + \phi_2 + (1 + \gamma)(2 + \gamma). \end{array} $$
 This shows~\eqref{eq:MF5}, and so survival.
 Note also that~\eqref{eq:MF5} reduces to~$\Lambda_{11} > 1$ when~$\gamma = 0$, and to~$\Lambda_{22} > 1$ when~$\gamma = \infty$, therefore, using again continuity, we get
\begin{equation}
\label{eq:MF8}
\begin{array}{rcl}
\lambda_1 > 1 / (1 + \phi_1) \ \hbox{and} \ \lambda_2 > 1 / (1 + \phi_2) & \Longrightarrow & \hbox{survival} \ \ \forall \,\gamma \geq 0, \vspace*{4pt} \\
                                            \lambda_1 > 1 / (1 + \phi_1) & \Longrightarrow & \hbox{survival} \ \ \forall \,\gamma > 0 \ \hbox{small}, \vspace*{4pt} \\
                                            \lambda_2 > 1 / (1 + \phi_2) & \Longrightarrow & \hbox{survival} \ \ \forall \,\gamma < \infty \ \hbox{large}, \end{array}
\end{equation}
 which improves~\eqref{eq:MF7}.
 See the top-left picture and the bottom-right picture of Figure~\ref{fig:MF1} for an illustration.
 In contrast, when all four~$\Lambda_{ij}$ are smaller than one,
 $$ \begin{array}{l}
    (\lambda_1 + \gamma \lambda_2)((1 + \phi_1)(1 + \phi_2) + (1 + \phi_1) + \gamma (1 + \phi_2)) \vspace*{4pt} \\ \hspace*{20pt} = \
     \Lambda_{11} (1 + \phi_2) + \Lambda_{11} + \gamma \Lambda_{12} + \gamma \Lambda_{22} (1 + \phi_1) + \gamma \lambda_{2, 1} + \gamma^2 \Lambda_{22} \vspace*{4pt} \\ \hspace*{20pt} < \
     (1 + \phi_2) + 1 + \gamma + \gamma (1 + \phi_1) + \gamma + \gamma^2 = \gamma \phi_1 + \phi_2 + (1 + \gamma)(2 + \gamma). \end{array} $$
 This shows that~\eqref{eq:MF5} is not satisfied, and so extinction.
\begin{figure}[t!]
\centering
\scalebox{0.45}{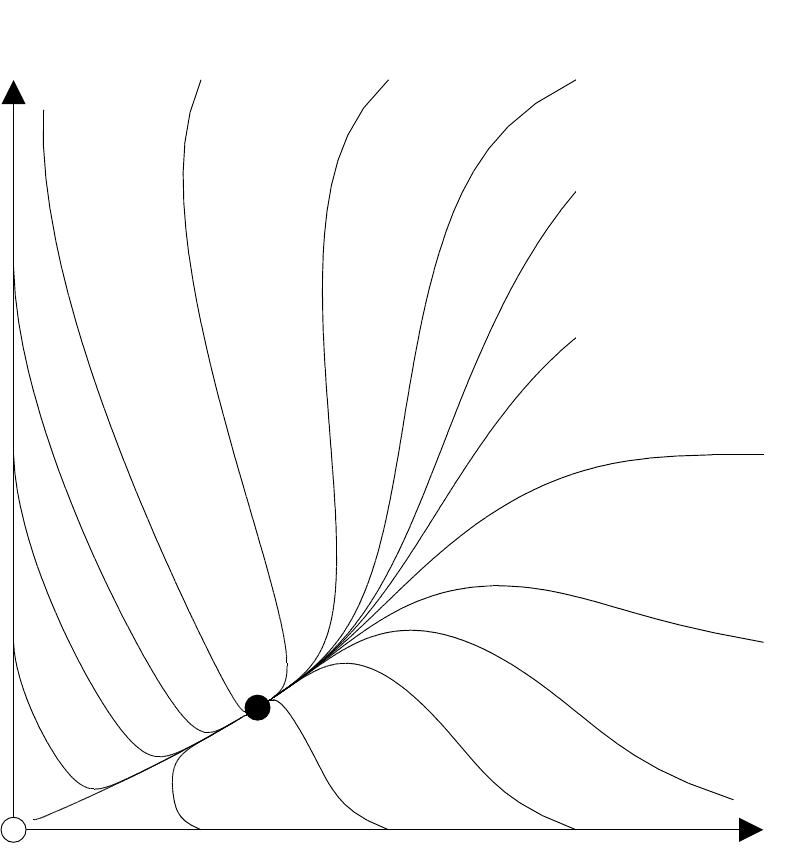} \hspace*{40pt} \scalebox{0.45}{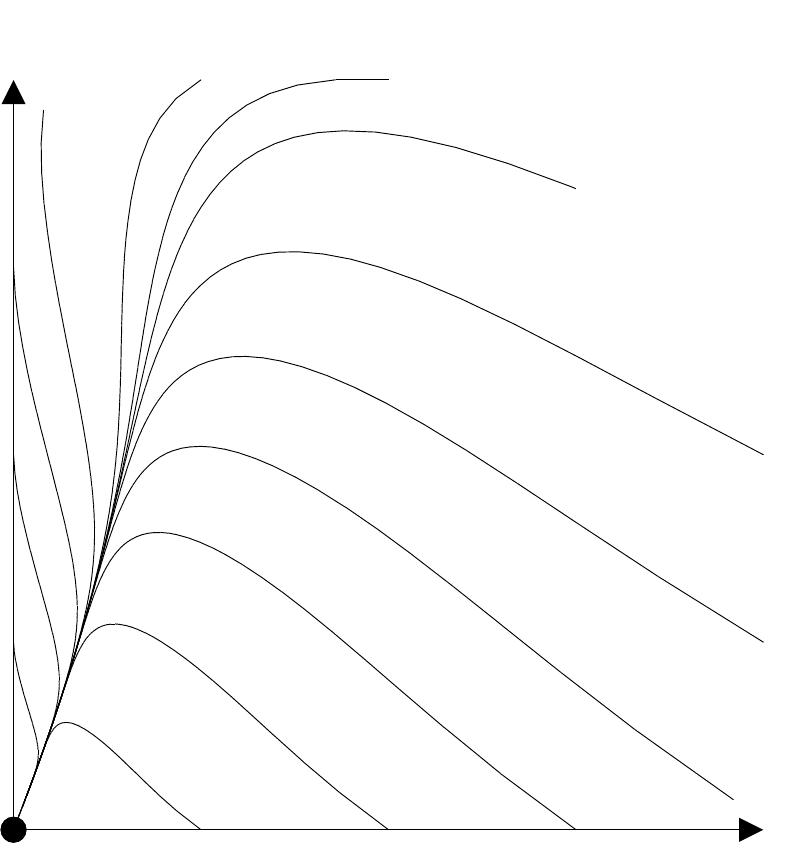} \vspace*{10pt} \\
\scalebox{0.45}{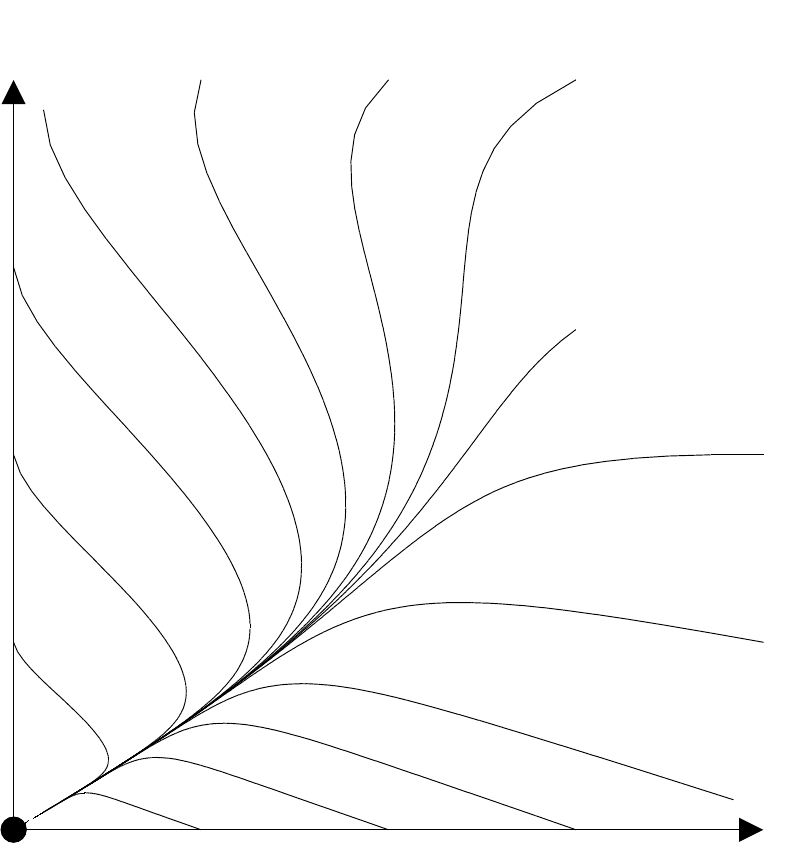} \hspace*{40pt} \scalebox{0.45}{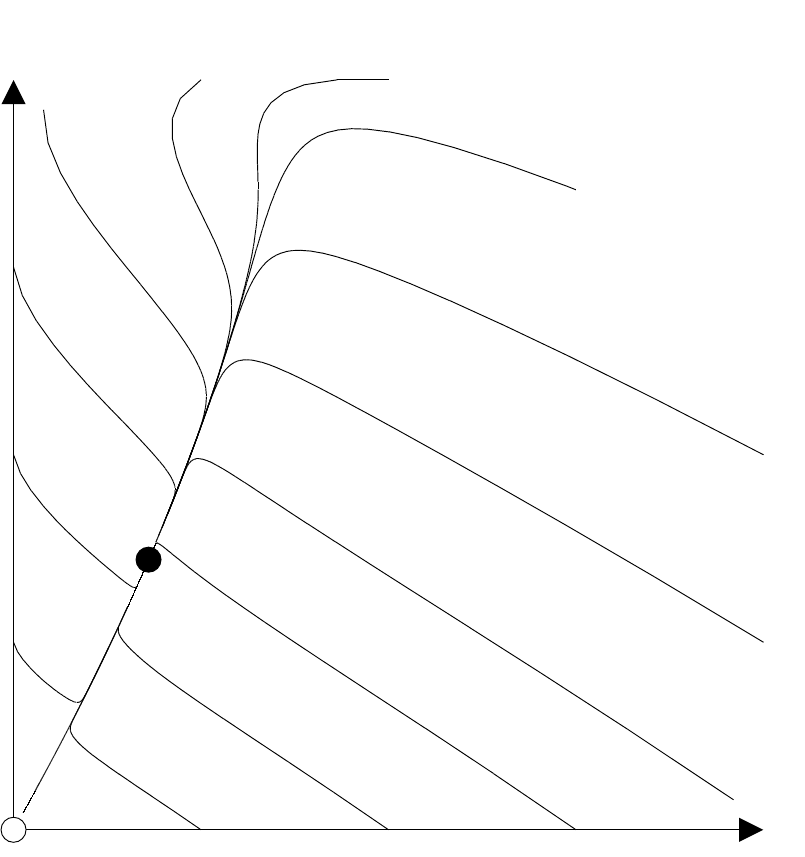}
\caption{\upshape{
 Projection in the~$(u_1, u_2)$ plane of the solution curves of~\eqref{eq:MF1}.
 The two pictures at the top are examples where~$\Lambda_{11} > 1$ and~$\Lambda_{22} < 1$, leading to survival when~$\gamma$ is small~(left) but not when~$\gamma$ is large~(right).
 The two pictures at the bottom show that reversing the internal and external infection rates of asymptomatic and symptomatic individuals, we now have survival when~$\gamma$ is large~(right) but not when~$\gamma$ is small~(left).}}
\label{fig:MF1}
\end{figure}
 Using again continuity and that~\eqref{eq:MF5} reduces to the inequality~$\Lambda_{11} > 1$ when~$\gamma = 0$, and~$\Lambda_{22} > 1$ when~$\gamma = \infty$, we conclude that
\begin{equation}
\label{eq:MF9}
\begin{array}{rcl}
\lambda_1, \lambda_2 < 1 / (1 + \phi_1 \vee \phi_2) & \Longrightarrow & \hbox{extinction} \ \ \forall \,\gamma \geq 0, \vspace*{4pt} \\
                       \lambda_1 < 1 / (1 + \phi_1) & \Longrightarrow & \hbox{extinction} \ \ \forall \,\gamma > 0 \ \hbox{small}, \vspace*{4pt} \\
                       \lambda_2 < 1 / (1 + \phi_2) & \Longrightarrow & \hbox{extinction} \ \ \forall \,\gamma < \infty \ \hbox{large}, \end{array}
\end{equation}
 which is the analog of Theorem~\ref{th:lambdasmall} with a specific~$\lambda_- = \lambda_- (\phi_1, \phi_2) > 0$.
 See the top-right picture and the bottom-left picture of Figure~\ref{fig:MF1} for an illustration. \vspace*{5pt} \\
{\bf Internal infection rates}.
 The system is also monotone with respect to the internal infection rates~$\phi_1, \phi_2$ in terms of survival versus extinction.
 However, the proof is more complicated because the left-hand side of~\eqref{eq:MF5}, which can be written as
 $$ \begin{array}{rcl}
      F (\phi_1, \phi_2) & \n = \n & (\lambda_1 + \gamma \lambda_2) \phi_1 \phi_2 + (2 (\lambda_1 + \gamma \lambda_2) - \gamma) \phi_1 \vspace*{4pt} \\ && \hspace*{18pt} + \,
                                     ((1 + \gamma)(\lambda_1 + \gamma \lambda_2) - 1) \phi_2 + (2 + \gamma)((\lambda_1 + \gamma \lambda_2) - (1 + \gamma)), \end{array} $$
 is not monotone in~$\phi_1, \phi_2$.
 Taking the partial derivatives, we get
 $$ \partial_{\phi_1} F = (\lambda_1 + \gamma \lambda_2)(2 + \phi_2) - \gamma \quad \hbox{and} \quad \partial_{\phi_2} F = (\lambda_1 + \gamma \lambda_2)(1 + \gamma + \phi_1) - 1, $$
 so the function~$F$ is separately convex in the first quadrant.
 In addition,
 $$ \begin{array}{l}
      \partial_{\phi_1} F (0, \phi_2) = (\lambda_1 + \gamma \lambda_2)(2 + \phi_2) - \gamma < 0 \vspace*{4pt} \\ \hspace*{20pt}
      \Longrightarrow \ \ (\lambda_1 + \gamma \lambda_2) \phi_2 < \gamma - 2 (\lambda_1 + \gamma \lambda_2) \vspace*{4pt} \\ \hspace*{20pt}
      \Longrightarrow \ \ F (0, \phi_2) \leq (1 + \gamma)(\lambda_1 + \gamma \lambda_2) \phi_2 + (2 + \gamma)((\lambda_1 + \gamma \lambda_2) - (1 + \gamma)) \vspace*{4pt} \\ \hspace*{90pt}
                                           < (1 + \gamma)(\gamma - 2 (\lambda_1 + \gamma \lambda_2)) + (2 + \gamma)((\lambda_1 + \gamma \lambda_2) - (1 + \gamma)) \vspace*{4pt} \\ \hspace*{90pt}
                                           = - \gamma (\lambda_1 + \gamma \lambda_2) - 2 (1 + \gamma) \leq 0. \end{array} $$
 Similarly, for the other variable,
 $$ \begin{array}{l}
      \partial_{\phi_2} F (\phi_1, 0) = (\lambda_1 + \gamma \lambda_2)(1 + \gamma + \phi_1) - 1 < 0 \vspace*{4pt} \\ \hspace*{20pt}
      \Longrightarrow \ \ (\lambda_1 + \gamma \lambda_2) \phi_1 < 1 - (1 + \gamma)(\lambda_1 + \gamma \lambda_2) \vspace*{4pt} \\ \hspace*{20pt}
      \Longrightarrow \ \ F (\phi_1, 0) \leq 2 (\lambda_1 + \gamma \lambda_2) \phi_1 + (2 + \gamma)((\lambda_1 + \gamma \lambda_2) - (1 + \gamma)) \vspace*{4pt} \\ \hspace*{90pt}
                                           < 2 (1 - (1 + \gamma)(\lambda_1 + \gamma \lambda_2)) + (2 + \gamma)((\lambda_1 + \gamma \lambda_2) - (1 + \gamma)) \vspace*{4pt} \\ \hspace*{90pt}
                                           = - \gamma (\lambda_1 + \gamma \lambda_2) - \gamma (3 + \gamma) \leq 0. \end{array} $$
 This shows that, for all~$\phi_1, \phi_2 \geq 0$,
 $$ F (0, \phi_2) \geq 0 \ \Longrightarrow \ \partial_{\phi_1} F (0, \phi_2) \geq 0 \quad \hbox{and} \quad
    F (\phi_1, 0) \geq 0 \ \Longrightarrow \ \partial_{\phi_2} F (\phi_1, 0) \geq 0. $$
 Using also convexity, we deduce that the sign of~$F (\phi_1, \phi_2)$, as opposed to the function itself, can only increase when~$\phi_1$ and~$\phi_2$ increase.
 This shows monotonicity with respect to the internal infection rates in the sense that, if there is survival, increasing~$\phi_1$ and~$\phi_2$ still results in survival, so the survival phase in the~$(\phi_1, \phi_2)$ plane consists of the region above a decreasing curve.
 In addition, taking~$\phi_2 = 0$ and solving for~$\phi_1$, and taking~$\phi_1 = 0$ and solving for~$\phi_2$, we get
 $$ \begin{array}{rclcl}
    \phi_1 & \n = \n & \displaystyle \frac{(2 + \gamma)((1 + \gamma) - (\lambda_1 + \gamma \lambda_2))}{2 (\lambda_1 + \gamma \lambda_2) - \gamma}
           & \hbox{when} & \phi_2 = 0, \vspace*{8pt} \\
    \phi_2 & \n = \n & \displaystyle \frac{(2 + \gamma)((1 + \gamma) - (\lambda_1 + \gamma \lambda_2))}{(1 + \gamma)(\lambda_1 + \gamma \lambda_2) - 1}
           & \hbox{when} & \phi_1 = 0. \end{array} $$
 The previous two expressions, together with~\eqref{eq:MF6}, imply that there is a unique phase transition from extinction to survival in the direction of
 $$ \begin{array}{rcl}
    \hbox{$\phi_1$ for all~$\phi_2 \geq 0$} & \Longleftrightarrow & \gamma / 2 < \lambda_1 + \gamma \lambda_2 < 1 + \gamma, \vspace*{4pt} \\
    \hbox{$\phi_2$ for all~$\phi_1 \geq 0$} & \Longleftrightarrow & 1 / (1 + \gamma) < \lambda_1 + \gamma \lambda_2 < 1 + \gamma, \end{array} $$
 but if~$\lambda_1 + \gamma \lambda_2$ is smaller, there is always extinction when~$\phi_1$ or~$\phi_2$ is small~(horizontal and vertical asymptotes in the first quadrant).
 Moreover, the implications in~\eqref{eq:MF8} can be written as
 $$ \begin{array}{rcl}
    \lambda_1, \lambda_2 > 0 \ \hbox{and} \ \phi_1 > 1 / \lambda_1 - 1 \ \hbox{and} \ \phi_2 > 1 / \lambda_2 - 1 & \Longrightarrow & \hbox{survival} \ \ \forall \,\gamma \geq 0, \vspace*{4pt} \\
                                                         \lambda_1 > 0 \ \hbox{and} \ \phi_1 > 1 / \lambda_1 - 1 & \Longrightarrow & \hbox{survival} \ \ \forall \,\gamma > 0 \ \hbox{small}, \vspace*{4pt} \\
                                                         \lambda_2 > 0 \ \hbox{and} \ \phi_2 > 1 / \lambda_2 - 1 & \Longrightarrow & \hbox{survival} \ \ \forall \,\gamma < \infty \ \hbox{large}, \end{array} $$
 which is the analog of Theorem~\ref{th:philarge} with a specific~$\phi_+ = \phi_+ (\lambda_1, \lambda_2) < \infty$.
 See the top-left picture and the bottom-right picture of Figure~\ref{fig:MF1}.
 Similarly,~\eqref{eq:MF9} can be written as
 $$ \begin{array}{rcl}
    \lambda_1, \lambda_2 < 1 \ \hbox{and} \ \phi_1, \phi_2 < 1 / (\lambda_1 \vee \lambda_2) - 1 & \Longrightarrow & \hbox{extinction} \ \ \forall \,\gamma \geq 0, \vspace*{4pt} \\
                                        \lambda_1 < 1 \ \hbox{and} \ \phi_1 < 1 / \lambda_1 - 1 & \Longrightarrow & \hbox{extinction} \ \ \forall \,\gamma > 0 \ \hbox{small}, \vspace*{4pt} \\
                                        \lambda_2 < 1 \ \hbox{and} \ \phi_2 < 1 / \lambda_2 - 1 & \Longrightarrow & \hbox{extinction} \ \ \forall \,\gamma < \infty \ \hbox{large}, \end{array} $$
 which is the analog of Theorem~\ref{th:phismall} with a specific~$\phi_- = \phi_- (\lambda_1, \lambda_2) > 0$.
 See the top-right picture and the bottom-left picture of Figure~\ref{fig:MF1} for an illustration. \vspace*{5pt} \\
{\bf Cluster size}.
 Finding the condition for survival when the cluster size is a general integer is much more complicated.
 However, the condition for survival in the mean-field model of the individual recovery process~\cite{belhadji_lanchier_2005}, which corresponds to the case~$\gamma = 0$ in our model, is given by
 $$ \lambda_1 (1 + \phi_1 + \phi_1^2 + \cdots + \phi_1^{N - 1}) > 1. $$
 Based on this result and the results above, we conjecture that we still have
 $$ \begin{array}{rcl}
                                \Lambda_{11}, \Lambda_{22} > 1 & \Longrightarrow & \hbox{survival} \ \ \forall \,\gamma \geq 0, \vspace*{4pt} \\
    \Lambda_{11}, \Lambda_{12}, \Lambda_{21}, \Lambda_{22} < 1 & \Longrightarrow & \hbox{extinction} \ \ \forall \,\gamma \geq 0, \end{array} $$
 and that whether we have survival or extinction depends only on~$\Lambda_{11}$ in the limit as~$\gamma \to 0$ and only on~$\Lambda_{22}$ in the limit as~$\gamma \to \infty$ but with
 $$ \Lambda_{ij} = \lambda_i (1 + \phi_j + \phi_j^2 + \cdots + \phi_j^{N - 1}) \quad \hbox{for} \ i, j = 1, 2, $$
 instead of the expression~$\Lambda_{ij} = \lambda_i (1 + \phi_j)$.
 If true, and because
 $$ \begin{array}{rcl}
    \lambda_i > 0 \ \hbox{and} \ \phi_j > 1 & \Longrightarrow & \Lambda_{ij} = \lambda_i (1 + \phi_j + \phi_j^2 + \cdots + \phi_j^{N - 1}) \to \infty \ \hbox{as} \ N \to \infty, \end{array} $$
 we deduce that
 $$ \begin{array}{rcl}
    \lambda_1, \lambda_2 > 0 \ \hbox{and} \ \phi_1, \phi_2 > 1 & \Longrightarrow & \hbox{survival} \ \ \forall \,N < \infty \ \hbox{large and} \ \gamma \geq 0, \vspace*{4pt} \\
                       \lambda_1 > 0 \ \hbox{and} \ \phi_1 > 1 & \Longrightarrow & \hbox{survival} \ \ \forall \,N < \infty \ \hbox{large and} \ \gamma > 0 \ \hbox{small}, \vspace*{4pt} \\
                       \lambda_2 > 0 \ \hbox{and} \ \phi_2 > 1 & \Longrightarrow & \hbox{survival} \ \ \forall \,N < \infty \ \hbox{large and} \ \gamma < \infty \ \hbox{large}, \end{array} $$
 which is the analog of Theorem~\ref{th:Nlarge}.
 In contrast, when~$N = 1$, the process reduces to the contact process with an asymptomatic state, therefore we have extinction when~$\lambda_1 + \gamma \lambda_2 < 1 + \gamma$ according to the mean-field analysis in~\cite{belhadji_lanchier_mercer_2024}.
 It follows that
 $$ \begin{array}{rcl}
    \lambda_1, \lambda_2 < 1 \ \hbox{and} \ N = 1 & \Longrightarrow & \hbox{extinction} \ \ \forall \,\gamma \geq 0, \vspace*{4pt} \\
               \lambda_1 < 1 \ \hbox{and} \ N = 1 & \Longrightarrow & \hbox{extinction} \ \ \forall \,\gamma > 0 \ \hbox{small}, \vspace*{4pt} \\
               \lambda_2 < 1 \ \hbox{and} \ N = 1 & \Longrightarrow & \hbox{extinction} \ \ \forall \,\gamma < \infty \ \hbox{large}, \end{array} $$
 which is the analog of Theorem~\ref{th:Nsmall} for the mean-field model. \vspace*{5pt} \\
{\bf Mutation rate}.
 In contrast with the infection rates, whether we have survival or extinction is not monotone with respect to the mutation rate.
\begin{figure}[t!]
\centering
\scalebox{0.60}{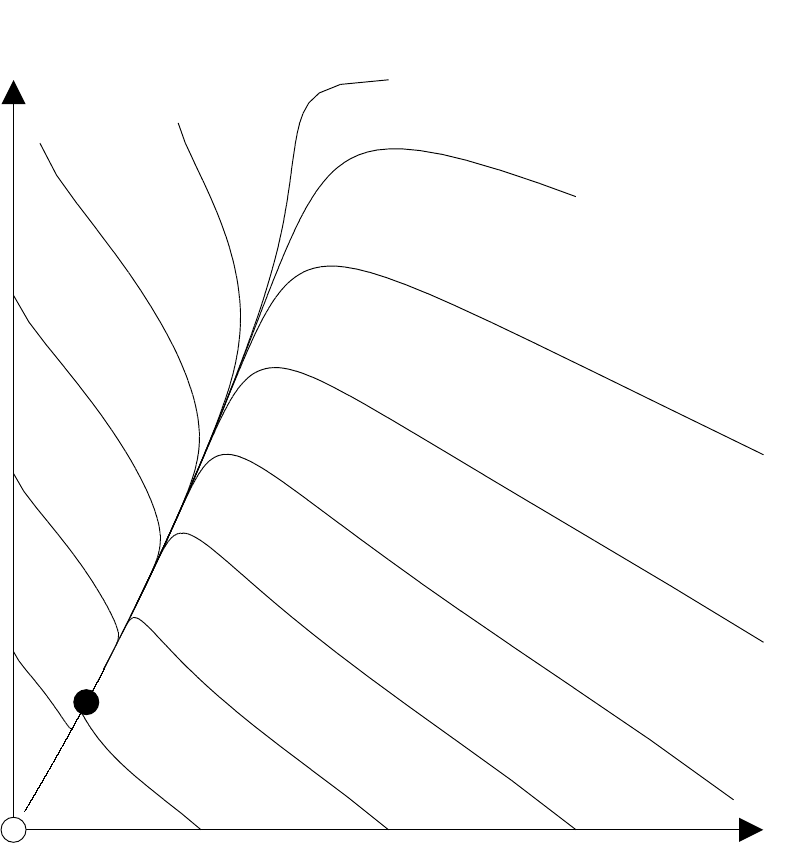} \hspace*{20pt} \scalebox{0.34}{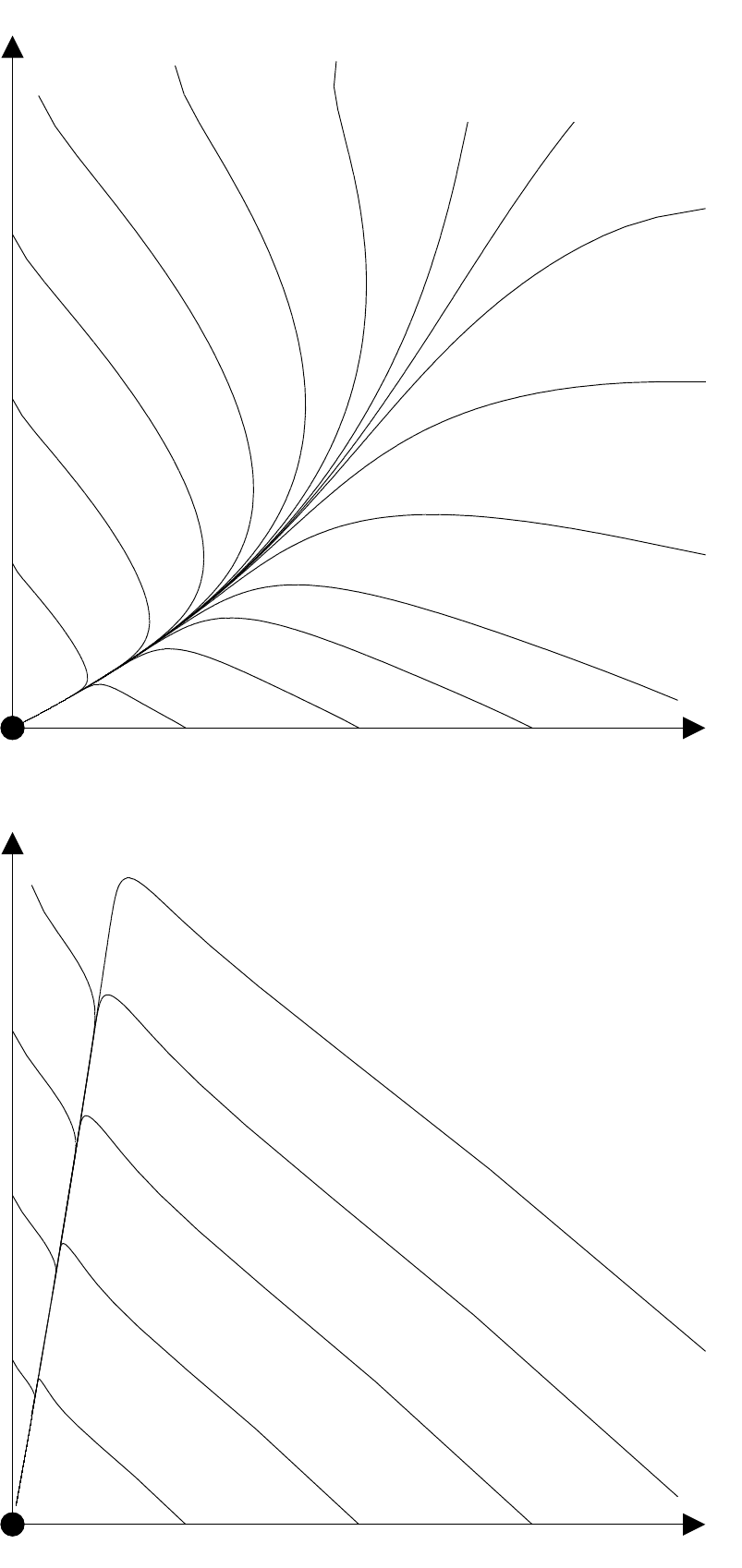}
\caption{\upshape{
 Projection in the~$(u_1, u_2)$ plane of the solution curves of~\eqref{eq:MF1}.
 The infection rates are the same in all three pictures but the value of the mutation rate~$\gamma$ varies.
 The pictures show that there is extinction for~$\gamma$ small and large~(right), but survival for intermediate values~(left), and so the lack of monotonicity with respect to~$\gamma$.}}
\label{fig:MF2}
\end{figure}
 According to~\eqref{eq:MF8} and~\eqref{eq:MF9},
 $$ \begin{array}{rcl}
    \Lambda_{11} < 1 < \Lambda_{22} & \Longrightarrow & \hbox{extinction} \ \,\forall \,\gamma > 0 \ \hbox{small but survival} \ \,\forall \,\gamma < \infty \ \hbox{large}, \vspace*{4pt} \\
    \Lambda_{11} > 1 > \Lambda_{22} & \Longrightarrow & \hbox{extinction} \ \,\forall \,\gamma < \infty \ \hbox{large but survival} \ \,\forall \,\gamma > 0 \ \hbox{small}, \end{array} $$
 showing that increasing~$\gamma$ promotes survival in some situations, but promotes extinction in other situations~(see Figure~\ref{fig:MF1}).
 Because~\eqref{eq:MF5} is a polynomial with degree two in~$\gamma$, there can be at most two phase transitions in the direction of~$\gamma$.
 In fact, the presence of two phase transitions is possible:
 taking~$\lambda_2 = \phi_1 = 0$,~$\lambda_1 = 3/4$ and~$\phi_2 = 11$, the left-hand side of~\eqref{eq:MF5} reduces to
 $$ \begin{array}{rcl}
      F (\gamma) & \n = \n & \lambda_1 ((2 + \phi_2) + \gamma (1 + \phi_2)) - \phi_2 - (1 + \gamma)(2 + \gamma) = - \frac{1}{4} (4 \gamma^2 - 24 \gamma + 13) \vspace*{4pt} \\
                 & \n = \n & - (\gamma - \gamma_-)(\gamma - \gamma_+) \quad \hbox{where} \quad \gamma_{\pm} = 3 \pm \sqrt{23}/2 > 0. \end{array} $$
 This shows extinction when~$\gamma < \gamma_-$, then survival when~$\gamma_- < \gamma < \gamma_+$, and again extinction when~$\gamma > \gamma_+$, and so a lack of monotonicity with respect to~$\gamma$~(see Figure~\ref{fig:MF2} for an illustration of the three cases).
 Note also that, in the previous numerical example,
 $$ \Lambda_{11} = \lambda_1 (1 + \phi_1) = \lambda_1 = 3/4 \quad \hbox{and} \quad \Lambda_{22} = \lambda_2 (1 + \phi_2) = 0, $$
 which shows that the conditions~$\Lambda_{11}, \Lambda_{22} < 1$ are not sufficient for extinction~(see the left picture in Figure~\ref{fig:MF2}), though there is always survival when the inequalities are flipped.


\section{Graphical representation and coupling}
\label{sec:GR}
 We now study the spatial stochastic model, starting from its construction.
 The process defined in the introduction keeps track of the number of asymptomatic and symptomatic individuals in each cluster, so the state at time~$t$ is a spatial configuration
  $$ \xi_t : \Z^d \longrightarrow \{(i, j) \in \Z_+ \times \Z_+ : i + j \leq N \}, $$
 where~$\xi_t (x) = (i, j)$ means that the social cluster~$x \in \Z^d$ contains~$i$ asymptomatic individuals and~$j$ symptomatic individuals.
 However, to study the process, it will be easier to think of each social cluster as a pile of~$N$ sites corresponding to the~$N$ individuals living in the cluster, and to assume that the process now keeps track of the state of each individual, rather than the state of each cluster.
 More precisely, we identify each cluster~$x \in \Z^d$ with
 $$ \{(x, 1), (x, 2), \ldots, (x, N) \} = \{(x, i) : i = 1, 2, \ldots, N \}, $$
 and think of site~$(x, i)$ as the location of the individual at level~$i$ of cluster~$x$, which can be of type~0 for healthy,~1 for asymptomatic, and~2 for symptomatic.
 Then, the state of the process at time~$t$, that we again denote by~$\xi_t$, becomes a spatial configuration
 $$ \xi_t : \Delta \longrightarrow \{0, 1, 2 \} \quad \hbox{where} \quad \Delta = \Z^d \times \{1, 2, \ldots, N \}, $$
 and where~$\xi_t (x, i)$ now gives the type of the individual at level~$i$ of cluster~$x$.
 Writing~$x \sim y$ to indicate that clusters~$x$ and~$y$ are nearest neighbors, i.e.,
 $$ \norm{x - y}_1 = |x_1 - y_1| + |x_2 - y_2| + \cdots + |x_d - y_d| = 1, $$
 and using an idea of Harris~\cite{harris_1972}, our epidemic model starting from any initial configuration can be constructed using the following collections of independent Poisson processes:
\begin{itemize}
\item {\bf External infections}.
 For~$j = 1, 2$, and all
 $$ (x, i) \in \Delta = \Z^d \times \{1, 2, \ldots, N \} \quad \hbox{and} \quad y \sim x, $$
 place an exponential clock with rate~$\lambda_j$ along the edge~$(x, i) \to (y, 1)$, and draw
 $$ \hbox{a type~$j$ arrow} \ (x, i, t) \arrow{$j$} (y, 1, t) \ \hbox{at the times~$t$ the clock rings} $$
 to indicate that if the tail of the arrow is occupied by a~$j$ and all the levels of~$y$ are occupied by a~0, then the head of the arrow becomes type~1. \vspace*{4pt}
\item {\bf Internal infections}.
 For~$j = 1, 2$, and all
 $$ (x, i) \in \Delta = \Z^d \times \{1, 2, \ldots, N \}, $$
 place an exponential clock with rate~$\phi_j$ at site~$(x, i)$, and put
 $$ \hbox{a type~$j$ bullet \bul{$j$} at} \ (x, i, t) \ \hbox{at the times~$t$ the clock rings} $$
 to indicate that if the bullet is occupied by a~$j$ and there is at least one~0 in cluster~$x$, then one of the~0s, say the lowest~0 to fix the ideas, becomes type~1. \vspace*{4pt}
\item {\bf Mutations}.
 Place an exponential clock with rate~$\gamma$ at each~$(x, i) \in \Delta$, and put
 $$ \hbox{a circle $\circ$ at} \ (x, i, t) \ \hbox{at the times~$t$ the clock rings} $$
 to indicate that if the circle is occupied by a~1, then it becomes type~2. \vspace*{4pt}
\item {\bf Recoveries}.
 Place an exponential clock with rate one at each~$(x, i) \in \Delta$, and put
 $$ \hbox{a cross $\times$ at} \ (x, i, t) \ \hbox{at the times~$t$ the clock rings} $$
 to indicate that if the cross is occupied by a~1 or a~2, then it becomes type~0.
\end{itemize}
 Because the transition rates at each site~$(x, i) \in \Delta$ depend on the number of~1s and~2s in the clusters but not on the specific location of these individuals in their cluster, the dynamics is invariant by permutation of the levels in each cluster.
 In particular, after each update, we may reorder the individuals in their cluster in such a way that
 $$ \hbox{the 2s are at the bottom}, \quad \hbox{the 1s are in the middle}, \quad \hbox{the 0s are at the top}. $$
 This vertical convention implies that, if an external infection arrow points at a healthy site~(necessarily at level~1 by construction), then the whole cluster must be healthy, which facilitates the couplings below.
 Using variations of the graphical representation above, we can now couple processes with different parameters and/or initial configurations as well as different processes to prove monotonicity/attractiveness and compare our process with other reference processes. \\
\indent
 To begin with, we let~$\xi_t, \xi_t'$ be the two individual recovery processes with different infection and mutation rates but the same cluster size:
 $$ \xi_t = \irpa (\lambda_1, \lambda_2, \phi_1, \phi_2, \gamma, N) \quad \hbox{and} \quad \xi_t' = \irpa (\lambda_1', \lambda_2', \phi_1', \phi_2', \gamma', N). $$
 Having two configurations~$\xi_t, \xi_t'$, we define the order
 $$ \xi_t \leq \xi_t' \quad \Longleftrightarrow \quad \xi_t (x, i) \leq \xi_t' (x, i) \ \hbox{for all} \ (x, i) \in \Delta, $$
 and assume that~$\xi_0 \leq \xi_0'$, which is equivalent to
 $$ (\xi_0, \xi_0') : \Delta \longrightarrow S \ \ \hbox{where} \ \ S = \{00, 01, 02, 11, 12, 22 \}. $$
 The next lemma shows monotonicity with respect to each of the five rates and attractiveness when the symptomatic individuals are more infectious than the asymptomatic individuals from the point of view of both external and internal infections.
\begin{lemma}
\label{lem:irpa}
 We have~$\xi_t \leq \xi_t'$ at all times~$t$ whenever
 $$ \xi_0 \leq \xi_0', \quad \lambda_1 \leq \lambda_1' \leq \lambda_2 \leq \lambda_2', \quad \phi_1 \leq \phi_1' \leq \phi_2 \leq \phi_2', \quad \gamma \leq \gamma'. $$
 In particular, if the process~$\xi_t$ survives then the process~$\xi_t'$ survives.
\end{lemma}
\begin{proof}
 The idea is to use a variation of the graphical representation introduced above and the superposition property for Poisson processes to couple the two processes in such a way that the domination holds at all times.
 The processes can be constructed jointly using
\begin{itemize}
 \item $\arrow{1}$ at rate~$\lambda_1$ with the same effect as~$\arrow{1}$ and~$\arrow{2}$ above on both processes, \vspace*{4pt}
 \item $\arrow{1'}$ at rate~$\lambda_1' - \lambda_1$ with the same effect as~$\arrow{1}$ above only on the second process, and the same effect as~$\arrow{2}$ above on both processes, \vspace*{4pt}
 \item $\arrow{2}$ at rate~$\lambda_2 - \lambda_1'$ with the same effect as~$\arrow{2}$ above on both processes, \vspace*{4pt}
 \item $\arrow{2'}$ at rate~$\lambda_2' - \lambda_2$ with the same effect as~$\arrow{2}$ above on the second process, \vspace*{4pt}
 \item \bul{1} at rate~$\phi_1$ with the same effect as~\bul{1} and~\bul{2} above on both processes, \vspace*{4pt}
 \item \bul{1'} at rate~$\phi_1' - \phi_1$ with the same effect as~\bul{1} above only on the second process, and the same effect as~\bul{2} above on both processes, \vspace*{4pt}
 \item \bul{2} at rate~$\phi_2 - \phi_1'$ with the same effect as~\bul{2} above on both processes, \vspace*{4pt}
 \item \bul{2'} at rate~$\phi_2' - \phi_2$ with the same effect as~\bul{2} above on the second process, \vspace*{4pt}
 \item $\circ$ at rate~$\gamma$ with the same effect as~$\circ$ above on both processes, \vspace*{4pt}
 \item $\circ'$ at rate~$\gamma' - \gamma$ with the same effect as~$\circ$ above on the second process, \vspace*{4pt}
 \item $\times$ at rate one with the same effect as~$\times$ above on both processes.
\end{itemize}
 The next step is to prove that the set~$S$ is closed under the dynamics of the coupling, which is modeled by the eleven collections of Poisson processes above.
 To deal with the external infections, because only healthy clusters can be infected from nearby clusters, to avoid trivialities, we may assume that the arrow points at a site~$(y, 1)$ at which at least one of the two processes is healthy, which gives only three possible states that are also in the set~$S$, namely,~00, 01, 02.
 Recalling that the whole cluster is healthy if and only if the first level is healthy by our vertical convention, we obtain Table~\ref{tab:lambda}, which shows all the possible transitions and confirms that the external infections do not create states that are not already in the set~$S$.
\begin{table}[h!]
\begin{tabular}{c|p{2em}p{0.4em}p{0.4em}p{1em}|p{2em}p{0.4em}p{0.4em}p{1em}|p{2em}p{0.4em}p{0.4em}p{1em}|p{2em}p{0.4em}p{0.4em}p{1em}}
        & $\arrow{1}$ & 00 & 01 & 02 & $\arrow{1'}$ & 00 & 01 & 02 & $\arrow{2}$ & 00 & 01 & 02 & $\arrow{2'}$ & 00 & 01 & 02 \\ \hline
    00 && 00 & 01 & 02 && 00 & 01 & 02 && 00 & 01 & 02 && 00 & 01 & 02 \\
    01 && 01 & 01 & 02 && 01 & 01 & 02 && 00 & 01 & 02 && 00 & 01 & 02 \\
    02 && 01 & 01 & 02 && 01 & 01 & 02 && 01 & 01 & 02 && 01 & 01 & 02 \\
    11 && 11 & 11 & 12 && 01 & 01 & 02 && 00 & 01 & 02 && 00 & 01 & 02 \\
    12 && 11 & 11 & 12 && 01 & 01 & 02 && 01 & 01 & 02 && 01 & 01 & 02 \\
    22 && 11 & 11 & 12 && 11 & 11 & 12 && 11 & 11 & 12 && 01 & 01 & 02
\end{tabular} \vspace*{10pt}
\caption{\upshape{
 Effects of the external infections on the coupling.
 The states in the first column, respectively, first row, represent the possible states at site~$(x, i)$, respectively, at site~$(y, 1)$, before a potential update, while the states in the table are the states resulting from the interactions.}}
\label{tab:lambda}
\end{table}
 To deal with the internal infections, assume that the lowest~0 for~$\xi_t$ is at level~$i$ and the lowest~0 for~$\xi_t'$ is at level~$i' > i$.
 Then, the state at site~$(x, i')$ before a potential internal infection must be~00, which either remains unchanged, or becomes~01, another state in~$S$.
 In particular, to check for stability, it is enough to look at the new state at site~$(x, i)$ after a potential internal infection.
 Table~\ref{tab:phi} shows all the possible transitions and confirms that the internal infections do not create states that are not already in the set~$S$.
\begin{table}[h!]
\begin{tabular}{c|p{2em}p{0.4em}p{0.4em}p{1em}|p{2em}p{0.4em}p{0.4em}p{1em}|p{2em}p{0.4em}p{0.4em}p{1em}|p{2em}p{0.4em}p{0.4em}p{1em}}
        & \bul{1} & 00 & 01 & 02 & \bul{1'} & 00 & 01 & 02 & \bul{2} & 00 & 01 & 02 & \bul{2'} & 00 & 01 & 02 \\ \hline
    00 && 00 & 01 & 02 && 00 & 01 & 02 && 00 & 01 & 02 && 00 & 01 & 02 \\
    01 && 01 & 01 & 02 && 01 & 01 & 02 && 00 & 01 & 02 && 00 & 01 & 02 \\
    02 && 01 & 01 & 02 && 01 & 01 & 02 && 01 & 01 & 02 && 01 & 01 & 02 \\
    11 && 11 & 11 & 12 && 01 & 01 & 02 && 00 & 01 & 02 && 00 & 01 & 02 \\
    12 && 11 & 11 & 12 && 01 & 01 & 02 && 01 & 01 & 02 && 01 & 01 & 02 \\
    22 && 11 & 11 & 12 && 11 & 11 & 12 && 11 & 11 & 12 && 01 & 01 & 02
\end{tabular} \vspace*{10pt}
\caption{\upshape{
 Effects of the internal infections on the coupling.
 The states in the first column, respectively, first row, represent the possible states at site~$(x, i)$, respectively, at site~$(x, i)$ where~$i$ is the level of the lowest~0 in the process~$\xi_t$, before a potential update, while the states in the table are the states resulting from the interactions.}}
\label{tab:phi}
\end{table}
 In the table, we use that, if~$(x, i)$ is in state~00 and all the states are in~$S$ before the infection, then~$i = i'$.
 Note that the states in Tables~\ref{tab:lambda} and~\ref{tab:phi} are exactly the same.
 Checking that the mutations and recoveries do not create states that are not already in the set~$S$ is straightforward.
 Finally, taking~$\xi_0 = \xi_0'$ gives monotonicity with respect to each of the five rates, while assuming that both processes have the same parameters gives attractiveness. of the process.
\end{proof} \\ \\
 We point out that the previous lemma does not hold if the asymptomatic individuals are more infectious than the symptomatic individuals from the point of view of either the external or the internal infections.
 Indeed, when~$\lambda_2 < \lambda_2' < \lambda_1 < \lambda_1'$, the two processes can be constructed jointly from the same collections of Poisson processes as before but exchanging the roles of the~1s and the~2s when defining the four types of arrows modeling the external infections.
\begin{table}[h!]
\begin{tabular}{c|p{7.5em}p{1em}||c|p{7.5em}p{1em}||c|p{7.5em}p{1em}}
        & $\arrow{2'}$ (rate~$\lambda_2' - \lambda_2$) & 00 &    & $\arrow{1}$ (rate~$\lambda_1 - \lambda_2'$) & 02 &    & $\arrow{1}$ (rate~$\lambda_1 - \lambda_2'$) & 00 \\ \hline
     22 &                                              & 02 & 11 &                                             & 12 & 12 &                                             & 10
\end{tabular} \vspace*{10pt}
\caption{\upshape{
 Counterexample showing that the natural coupling fails to prove monotonicity and attractiveness when the asymptomatic individuals are more infectious than the symptomatic individuals.}}
\label{tab:counterexample}
\end{table}
 For this natural coupling, even if the infection rates for the second process are larger than their counterparts for the first process, and even if both processes start from the same initial configuration, the dynamics can create the states~02, then~12, then~10, as shown in Table~\ref{tab:counterexample}, therefore the coupling fails to establish that the process is attractive or monotone.
 Assuming instead that~$\phi_2 < \phi_2' < \phi_1 < \phi_1'$ leads to a similar counterexample.
 This shows that the assumptions in the lemma are necessary. \\
\indent
 We now use additional coupling techniques to show that our process is sandwiched between two individual recovery processes.
 To prove this result, we let
 $$ \begin{array}{rcl}
    \xi_t^- & \n = \n & \irpa (\lambda_1 \wedge \lambda_2, \lambda_1 \wedge \lambda_2, \phi_1 \wedge \phi_2, \phi_1 \wedge \phi_2, \gamma, N), \vspace*{4pt} \\
    \xi_t^+ & \n = \n & \irpa (\lambda_1 \vee \lambda_2, \lambda_1 \vee \lambda_2, \phi_1 \vee \phi_2, \phi_1 \vee \phi_2, \gamma, N). \end{array} $$
 According to Lemma~\ref{lem:irpa}, we have the implication
 $$ \xi_0^- \leq \xi_0 \leq \xi_0^+, \ \ \lambda_1 \leq \lambda_2, \ \ \phi_1 \leq \phi_2 \quad \Longrightarrow \quad \xi_t^- \leq \xi_t \leq \xi_t^+ \ \hbox{at all times} \ t. $$
 This result is in fact true regardless of the ordering of the rates.
 The next lemma shows for instance the result when the asymptomatic individuals are more infectious than the symptomatic individuals from the point of view of both external and internal infections.
\begin{lemma}
\label{lem:irp}
 We have~$\xi_t^- \leq \xi_t \leq \xi_t^+$ at all times~$t$ whenever
 $$ \xi_0^- \leq \xi_0 \leq \xi_0^+, \quad \lambda_1 > \lambda_2, \quad \phi_1 > \phi_2. $$
\end{lemma}
\begin{proof}
 We construct the three processes jointly by using
\begin{itemize}
 \item $\arrow{2}$ at rate~$\lambda_2$ with the same effect as~$\arrow{1}$ and~$\arrow{2}$ above on all three processes, \vspace*{4pt}
 \item $\arrow{1}$ at rate~$\lambda_1 - \lambda_2$ with the same effect as~$\arrow{1}$ above only on the processes~$\xi_t^-$ and~$\xi_t$, and the same effect as~$\arrow{2}$ above only on the process~$\xi_t^+$, \vspace*{4pt}
 \item \bul{2} at rate~$\phi_2$ with the same effect as~\bul{1} and~\bul{2} above on all three processes, \vspace*{4pt}
 \item \bul{1} at rate~$\phi_1 - \phi_2$ with the same effect as~\bul{1} above only on the processes~$\xi_t$ and~$\xi_t^+$, and the same effect as~\bul{2} above on the process~$\xi_t^+$, \vspace*{4pt}
 \item $\circ$ at rate~$\gamma$ with the same effect as~$\circ$ above on all three processes, \vspace*{4pt}
 \item $\times$ at rate one with the same effect as~$\times$ above on all three processes.
\end{itemize}
 Table~\ref{tab:irp} is the analog of Table~\ref{tab:lambda} but for the couplings~$(\xi_t^-, \xi_t)$ and~$(\xi_t, \xi_t^+)$.
\begin{table}[h!]
\begin{tabular}{c|p{2em}p{0.4em}p{0.4em}p{1em}|p{2em}p{0.4em}p{0.4em}p{1em}||p{2em}p{0.4em}p{0.4em}p{1em}|p{2em}p{0.4em}p{0.4em}p{1em}}
        & $\arrow{2}$ & 00 & 01 & 02 & $\arrow{1}$ & 00 & 01 & 02 & $\arrow{2}$ & 00 & 01 & 02 & $\arrow{1}$ & 00 & 01 & 02 \\ \hline
    00 && 00 & 01 & 02 && 00 & 01 & 02 && 00 & 01 & 02 && 00 & 01 & 02 \\
    01 && 01 & 01 & 02 && 01 & 01 & 02 && 01 & 01 & 02 && 01 & 01 & 02 \\
    02 && 01 & 01 & 02 && 00 & 01 & 02 && 01 & 01 & 02 && 01 & 01 & 02 \\
    11 && 11 & 11 & 12 && 01 & 01 & 02 && 11 & 11 & 12 && 11 & 11 & 12 \\
    12 && 11 & 11 & 12 && 00 & 01 & 02 && 11 & 11 & 12 && 11 & 11 & 12 \\
    22 && 11 & 11 & 12 && 00 & 01 & 02 && 11 & 01 & 12 && 01 & 01 & 02
\end{tabular} \vspace*{10pt}
\caption{\upshape{
 Effects of the external infections on~$(\xi_t^-, \xi_t)$ on the left and on~$(\xi_t, \xi_t^+)$ on the right.
 The states in the first column, respectively, first row, represent the possible states at site~$(x, i)$, respectively, at site~$(y, 1)$, before a potential update, while the states in the table are the states resulting from the interactions.}}
\label{tab:irp}
\end{table}
 It shows all the possible transitions and confirms that the external infections do not create states that are not already in the set~$S$.
 As previously, the effects of the internal infections are exactly the same as the external infections, and it is straightforward to check that the mutations and recoveries do not create states that are not already in the set~$S$.
 This completes the proof.
\end{proof} \\ \\
 To compare our process with individual recovery processes, note that, after identifying asymptomatic and symptomatic individuals, we get
 $$ \begin{array}{rcl}
    \irpa (\lambda_1 \wedge \lambda_2, \lambda_1 \wedge \lambda_2, \phi_1 \wedge \phi_2, \phi_1 \wedge \phi_2, \gamma, N) & \n = \n &
     \irp (\lambda_1 \wedge \lambda_2, \phi_1 \wedge \phi_2, N), \vspace*{4pt} \\
    \irpa (\lambda_1 \vee \lambda_2, \lambda_1 \vee \lambda_2, \phi_1 \vee \phi_2, \phi_1 \vee \phi_2, \gamma, N) & \n = \n &
     \irp (\lambda_1 \vee \lambda_2, \phi_1 \vee \phi_2, N). \end{array} $$
 This, together with Lemma~\ref{lem:irp}, implies that
\begin{equation}
\label{eq:irp-irpa}
\irp (\lambda_1 \wedge \lambda_2, \phi_1 \wedge \phi_2, N) \leq
\irpa (\lambda_1, \lambda_2, \phi_1, \phi_2, \gamma, N) \leq
\irp (\lambda_1 \vee \lambda_2, \phi_1 \vee \phi_2, N)
\end{equation}
 at all times provided this is true at time~0, where~$\leq$ now means that the set of infected sites for the process on the left is included in the set of infected sites for the process on the right, i.e., we ignore the distinction between asymptomatic and symptomatic individuals. \\
\indent
 To conclude this section, we now prove that the process dominates the contact process with an asymptomatic state.
 To prove this result, we let
 $$ \xi_t = \irpa (\lambda_1, \lambda_2, \phi_1, \phi_2, \gamma, N) \quad \hbox{and} \quad \xi_t^0 = \irpa (\lambda_1, \lambda_2, \phi_1, \phi_2, \gamma, 1). $$
 Both processes have the same rates but the clusters reduce to a single individual in the second process.
 Coupling these two processes requires a special treatment because they are defined on different spatial structures.
 Defining the new order
 $$ \xi_t^0 \leq \xi_t \quad \Longleftrightarrow \quad \xi_t^0 (x, 1) \leq \xi_t (x, 1) \ \hbox{for all} \ x \in \Z^d, $$
 we have the following lemma.
\begin{lemma}
\label{lem:cpa}
 We have~$\xi_t^0 \leq \xi_t$ at all times~$t$ whenever~$\xi_0^0 \leq \xi_0$ and~$\lambda_1 \leq \lambda_2$.
\end{lemma}
\begin{proof}
 The idea is again to find a coupling that preserves the order.
 We construct~$\xi_t$ from the graphical representation introduced at the beginning of this section modified so that the external infections are coupled.
 More precisely, because~$\lambda_1 \leq \lambda_2$, we may assume that
\begin{itemize}
 \item $\arrow{1}$ now have the same effect as~$\arrow{1}$ and~$\arrow{2}$ above, \vspace*{4pt}
 \item $\arrow{2}$ now appear at rate~$\lambda_2 - \lambda_1$.
\end{itemize}
 To construct the process~$\xi_t^0$, notice that, because the internal infections only occur in clusters that have at least one infected and one healthy individual and because~$N = 1$, there cannot be any internal infections in~$\xi_t^0$ so we ignore the~\bul{1} and~\bul{2}.
 In addition, the arrows starting from and circles and crosses located at level~$i > 1$ are also unimportant because the clusters only have one level.
 In particular, the process~$\xi_t^0$ can be constructed using only
\begin{itemize}
 \item $\arrow{1}$ and~$\arrow{2}$ starting from level~1, and~$\circ$ and~$\times$ located at level~1,
\end{itemize}
 from the graphical representation of~$\xi_t$.
 Because the arrows starting from level~$i > 1$, the bullets at any location, and the circles and crosses located at level~$i > 1$ have no effect on the process~$\xi_t$, while they either turn the~0s into~1s and the~1s into~2s in the process~$\xi_t'$ or have no effect on the state at level~1, their action cannot break the domination to be proved.
 In addition, the effect of the arrows starting from level~1 of the coupling is the same as the effect of~$\arrow{1}$ and~$\arrow{2}$ in Table~\ref{tab:lambda} that preserve the order.
 As previously, checking that the circles and crosses at level~1 do not create states that are not already in the set~$S$ is straightforward.
\end{proof} \\ \\
 To compare our process with the contact process with an asymptomatic state, note that, regardless of the value of the internal infection rates,
\begin{equation}
\label{eq:cpa}
\irpa (\lambda_1, \lambda_2, \phi_1, \phi_2, \gamma, 1) = \cpa (\lambda_1, \lambda_2, \gamma).
\end{equation}
 This, together with Lemma~\ref{lem:cpa}, implies that, when~$\lambda_1 \leq \lambda_2$,
\begin{equation}
\label{eq:cpa-irpa}
\cpa (\lambda_1, \lambda_2, \gamma) \leq \irpa (\lambda_1, \lambda_2, \phi_1, \phi_2, \gamma, N)
\end{equation}
 at all times provided this is true at time~0.


\section{Comparisons with the IRP and the CPA}
\label{sec:irp-cpa}
 In this section, we use the results from~\cite{belhadji_lanchier_2005,belhadji_lanchier_mercer_2024} about the individual recovery process and the contact process with an asymptomatic state in combination with~\eqref{eq:irp-irpa}--\eqref{eq:cpa-irpa} to prove some of our main results.
 Throughout the section, the process~$\xi_t$ refers to our epidemic model.
 To begin with, we compare the process with the \irp \ to prove part a of all the theorems. \\ \\
\begin{proofof}{Theorems~\ref{th:lambdalarge}.a and~\ref{th:lambdasmall}.a}
 By Theorem~1 and Lemma~1.1 in~\cite{belhadji_lanchier_2005},
 $$ \begin{array}{llllll}
    \forall \,\phi \geq 0, \,N \geq 1, & \exists \,\lambda_+ = \lambda_+ (\phi, N) \leq \lambda_c & \n : \n &
    \irp (\lambda, \phi, N) & \hbox{survives} & \forall \,\lambda > \lambda_+, \vspace*{4pt} \\
    \forall \,\phi \geq 0, \,N \geq 1, & \exists \,\lambda_- = \lambda_- (\phi, N) > 0 & \n : \n &
    \irp (\lambda, \phi, N) & \hbox{dies out} & \forall \,\lambda < \lambda_-. \end{array} $$
 Using also~\eqref{eq:irp-irpa}, we deduce that
 $$ \begin{array}{rcll}
    \lambda_1, \lambda_2 > \lambda_+ & \Longrightarrow &
    \xi_t \geq \irp (\lambda_1 \wedge \lambda_2, \phi_1 \wedge \phi_2, N) & \hbox{both survive}, \vspace*{4pt} \\
    \lambda_1, \lambda_2 < \lambda_- & \Longrightarrow &
    \xi_t \leq \irp (\lambda_1 \vee \lambda_2, \phi_1 \vee \phi_2, N)     & \hbox{both die out}, \end{array} $$
 for all mutation rate~$\gamma$, which proves the theorems.
\end{proofof} \\ \\
\begin{proofof}{Theorems~\ref{th:philarge}.a and~\ref{th:phismall}.a}
 By Theorem~3 and Lemma~1.1 in~\cite{belhadji_lanchier_2005},
 $$ \begin{array}{llllll}
    \forall \,\lambda > 0, \,N > 1, & \exists \,\phi_+ = \phi_+ (\lambda, N) < \infty & \n : \n &
    \irp (\lambda, \phi, N) & \hbox{survives} & \forall \,\phi > \phi_+, \vspace*{4pt} \\
    \forall \,\lambda < \lambda_c, \,N \geq 1, & \exists \,\phi_- = \phi_- (\lambda, N) > 0 & \n : \n &
    \irp (\lambda, \phi, N) & \hbox{dies out} & \forall \,\phi < \phi_-. \end{array} $$
 Using also~\eqref{eq:irp-irpa}, we deduce that
 $$ \begin{array}{rcll}
    \lambda_1, \lambda_2 > 0, \,\phi_1, \phi_2 > \phi_+, \,N > 1 & \Longrightarrow &
    \xi_t \geq \irp (\lambda_1 \wedge \lambda_2, \phi_1 \wedge \phi_2, N) & \hbox{both survive}, \vspace*{4pt} \\
    \lambda_1, \lambda_2 < \lambda_c, \,\phi_1, \phi_2 < \phi_-, \,N \geq 1 & \Longrightarrow &
    \xi_t \leq \irp (\lambda_1 \vee \lambda_2, \phi_1 \vee \phi_2, N) & \hbox{both die out}, \end{array} $$
 for all mutation rate~$\gamma$, which proves the theorems.
\end{proofof} \\ \\
\begin{proofof}{Theorems~\ref{th:Nlarge}.a and~\ref{th:Nsmall}.a}
 By Theorem~5 and Lemma~1.1 in~\cite{belhadji_lanchier_2005},
 $$ \begin{array}{llllll}
    \forall \,\lambda > 0, \,\phi > 1, & \exists \,N_+ = N_+ (\lambda, \phi) < \infty & \n : \n &
    \irp (\lambda, \phi, N) & \hbox{survives} & \forall \,N > N_+, \vspace*{4pt} \\
    \forall \,\lambda < \lambda_c, \,\phi \geq 0, & \exists \,N_- = N_- (\lambda, \phi) > 1 & \n : \n &
    \irp (\lambda, \phi, N) & \hbox{dies out} & \forall \,N < N_-. \end{array} $$
 Using also~\eqref{eq:irp-irpa}, we deduce that
 $$ \begin{array}{rcll}
    \lambda_1, \lambda_2 > 0, \,\phi_1, \phi_2 > 1, \,N > N_+ & \Longrightarrow &
    \xi_t \geq \irp (\lambda_1 \wedge \lambda_2, \phi_1 \wedge \phi_2, N) & \hbox{both survive}, \vspace*{4pt} \\
    \lambda_1, \lambda_2 < \lambda_c, \,\phi_1, \phi_2 \geq 0, \,N < N_- & \Longrightarrow &
    \xi_t \leq \irp (\lambda_1 \vee \lambda_2, \phi_1 \vee \phi_2, N)     & \hbox{both die out}, \end{array} $$
 for all mutation rate~$\gamma$, which proves the theorems.
\end{proofof} \\ \\
 Comparing the process with the \cpa \ gives a few more results. \\ \\
\begin{proofof}{Theorems~\ref{th:Nsmall}.b and~\ref{th:Nsmall}.c}
 By~(5.a), (6.b) and~(9.b) in~\cite{belhadji_lanchier_mercer_2024},
 $$ \begin{array}{llllll}
    \forall \,\lambda_1 < \lambda_c, \,\lambda_2 \geq 0, & \exists \,\gamma_- = \gamma_- (\lambda_1, \lambda_2) > 0 & \n : \n &
    \cpa (\lambda_1, \lambda_2, \gamma) & \hbox{dies out} & \forall \,\gamma < \gamma_-, \vspace*{4pt} \\
    \forall \,\lambda_1 \geq 0, \,\lambda_2 < \lambda_c, & \exists \,\gamma_+ = \gamma_+ (\lambda_1, \lambda_2) < \infty & \n : \n &
    \cpa (\lambda_1, \lambda_2, \gamma) & \hbox{dies out} & \forall \,\gamma < \gamma_+. \end{array} $$
 Using also~\eqref{eq:cpa}, we deduce that
 $$ \begin{array}{rcll}
    \lambda_1 < \lambda_c, \,\lambda_2 \geq 0, \,\gamma < \gamma_-, \,N = 1 & \Longrightarrow & \xi_t = \cpa (\lambda_1, \lambda_2, \gamma) & \hbox{both die out}, \vspace*{4pt} \\
    \lambda_1 \geq 0, \,\lambda_2 < \lambda_c, \,\gamma > \gamma_+, \,N = 1 & \Longrightarrow & \xi_t = \cpa (\lambda_1, \lambda_2, \gamma) & \hbox{both die out}, \end{array} $$
 for all internal infection rates~$\phi_1, \phi_2$, which proves the theorems.
\end{proofof} \\ \\
\begin{proofof}{Theorem~\ref{th:lambdalarge}.c}
 By~(5.b) and~(6.a) in~\cite{belhadji_lanchier_mercer_2024},
 $$ \begin{array}{llllll}
    \forall \,\lambda_1 \geq 0, \,\lambda_2 > \lambda_c, & \exists \,\gamma_+ = \gamma_+ (\lambda_1, \lambda_2) < \infty & \n : \n &
    \cpa (\lambda_1, \lambda_2, \gamma) & \hbox{survives} & \forall \,\gamma > \gamma_+. \end{array} $$
 Using also~\eqref{eq:cpa-irpa}, we deduce that, when~$\lambda_1 \leq \lambda_2$,
 $$ \begin{array}{rcll}
    \lambda_2 > \lambda_c, \,\gamma > \gamma_+ & \Longrightarrow & \xi_t \geq \cpa (\lambda_1, \lambda_2, \gamma) & \hbox{both survive}. \end{array} $$
 This proves the result when~$\lambda_1 \leq \lambda_2$, while the case~$\lambda_1 > \lambda_2$ follows from Theorem~\ref{th:lambdalarge}.a.
\end{proofof}


\section{Percolation results}
\label{sec:perco}
 Most of our results show either survival or extinction when the parameter~$\gamma$ and possibly another parameter of the system are small but positive or large but finite.
 Though the technical details strongly depend on the value of the parameters, the common background strategy to prove these results is to first study the process in the limiting case where these parameters are equal to zero or infinity and then apply a perturbation argument.
 In order to apply a perturbation argument, the first step is to control the speed of convergence by coupling the process properly rescaled in space and time with oriented site percolation, a technique referred to as block construction, and the goal of this section is to collect a few percolation results we will use repeatedly later.
 More precisely, we consider two percolation processes, one to prove survival using the percolation of its open sites, and one to prove extinction using the lack of percolation of its closed sites. \\ \\
\noindent{\bf Percolation model 1 (survival).}
 To define the first percolation model, let
 $$ \lat_+ = \{(m, n) \in \Z \times \Z_+ : m + n \ \hbox{is even} \}, $$
 which we turn into a directed graph $\vec{\lat}_+$ by placing arrows
 $$ (m, n) \to (m', n') \ \iff \ |m - m'| = 1 \ \hbox{and} \ n' = n + 1, $$
 so each site has two outgoing arrows, one going up-left and one going up-right.
 Oriented site percolation with parameter~$p$ and range~$M$ on this directed graph is obtained by assuming that sites are open with probability~$p$, closed with probability~$1 - p$, and that the states of sites that are more than graph distance~$M$ apart are independent, meaning that having a finite collection of sites~$\pi_i = (m_i, n_i) \in \lat_+$ for~$i = 1, 2, \ldots, K$,
\begin{equation}
\label{eq:dependence}
  P (\pi_1, \pi_2, \ldots, \pi_K \ \hbox{are open}) = p^K \quad \hbox{whenever} \quad |\pi_i - \pi_j| > M \ \hbox{for all} \ i \neq j,
\end{equation}
 where~$|\cdot|$ refers to the graph distance.
 Writing~$\pi \overset{o}{\longrightarrow} \pi'$ to indicate the existence of a directed path of open sites from site~$\pi$ to site~$\pi'$, we define the open cluster starting at~$\pi$ as
 $$ O_{\pi} = \{\pi' \in \lat_+ : \pi \overset{o}{\longrightarrow} \pi' \}. $$
 In the next sections, open sites will correspond, via a coupling, to space-time blocks in which the infection successfully propagates in the interacting particle system, therefore sites that can be reached from level~$n = 0$ by an open path will correspond to infected space-time blocks.
 In particular, the key to proving survival is the following lemma, which states that the open sites percolate with positive probability if the density~$p$ of open sites is close enough to one.
\begin{lemma}
\label{lem:perco-1}
 Assume that~$p > 1 - 3^{- 4 (2M + 1)^2}$. Then,
 $$ P (\card (O_{\pi_0}) = \infty) > 0 \quad \hbox{for all} \quad \pi_0 \in \lat_+. $$
\end{lemma}
\begin{proof}
 The result is proved in~\cite[Section~10]{durrett_1984} when~$M = 0$, i.e., sites are independent, using a contour argument.
 The main ingredient is to show that the existence of a contour of length~$l$ induces a set of~$l/4$ sites that must be closed.
 Having a positive range~$M$, because
 $$ \card \{\pi' \in \lat_+ : |\pi' - \pi| \leq M \} \leq (2M + 1)^2, $$
 we can extract from the set of~$l/4$ closed sites a subset of~$l/(4 (2M + 1)^2)$ sites that are more than distance~$M$ apart, and therefore must be independent in addition to being closed.
 The rest of the proof follows the exact same arguments as in~\cite[Section~10]{durrett_1984}.
\end{proof} \\ \\
\noindent{\bf Percolation model 2 (extinction).}
 To define the second process, let~$\lat_- = \Z^d \times \Z_+$, which we turn into a directed graph~$\vec{\lat}_-$ by placing arrows
 $$ (m, n) \to (m', n') \ \iff \ |m_1 - m_1'| + \cdots + |m_d - m_d'| + |n - n'| = 1 \ \hbox{and} \ n' \leq n, $$
 so each site has $2d + 1$ outgoing arrows,~$2d$ horizontal arrows and one vertical upward arrow.
 As previously, we assume that sites are open with probability~$p$, closed with probability~$1 - p$, and that the range is~$M$, meaning that~\eqref{eq:dependence} holds for the graph distance on~$\vec{\lat}_-$.
 In addition, writing~$\pi \overset{c}{\longrightarrow} \pi'$ to indicate the existence of a directed path of closed sites from site~$\pi$ to site~$\pi'$, we now define the closed cluster starting at~$\pi$ as
 $$ C_{\pi} = \{\pi' \in \lat_- : \pi \overset{c}{\longrightarrow} \pi' \}. $$
 In the next sections, open sites will correspond, via a coupling, to space-time blocks in which the infection fails to propagate in the interacting particle system, therefore infected space-time blocks will correspond to sites that can be reached by a path of closed sites starting at level~$n = 0$.
 In particular, the key to proving extinction is the following lemma, which states that the closed sites do not percolate if the density~$p$ of open sites is close enough to one.
\begin{lemma}
\label{lem:perco-2}
 Assume that~$p > 1 - (2d + 1)^{- (2M + 1)^{d + 1}}$. Then,
 $$ P (\card (C_{\pi_0}) = \infty) = 0 \quad \hbox{for all} \quad \pi_0 \in \lat_-. $$
\end{lemma}
\begin{proof}
 Because there are~$2d + 1$ outgoing arrows from each site in~$\vec{\lat}_-$, the number of self-avoiding directed paths of length~$l$ starting from~$\pi$ satisfies
\begin{equation}
\label{eq:perco-2a}
\card \{\hbox{self-avoiding directed paths of length~$l$ from~$\pi$} \} \leq (2d + 1)^l.
\end{equation}
 Now, let~$\vec{\pi} = (\pi_0, \pi_1, \dots, \pi_{l - 1}) \subset \lat_-$ be a self-avoiding directed path of length~$l$.
 Because
 $$ \card \{\pi' \in \lat_- : |\pi' - \pi| \leq M \} \leq (2M + 1)^{d + 1}, $$
 we can extract from~$\vec \pi$ a subset of at least~$l / (2M + 1)^{d + 1}$ sites that are more than distance~$M$ apart which, in view of~\eqref{eq:dependence}, must be independent.
 This implies that
\begin{equation}
\label{eq:perco-2b}
 P (\vec \pi \ \hbox{is closed}) \leq (1 - p)^{l / (2M + 1)^{d + 1}}.
\end{equation}
 Using~\eqref{eq:perco-2a}--\eqref{eq:perco-2b}, monotone convergence, and a union bound, we deduce that
 $$ \begin{array}{rcl}
      P (\card(C_{\pi_0}) = \infty) & \n = \n & \lim_{l \to \infty} P (\exists \,\vec{\pi} = (\pi_0, \pi_1, \dots, \pi_{l - 1}) \subset \lat_- \ \hbox{closed}) \vspace*{4pt} \\
                                    & \n \leq \n & \lim_{l \to \infty} (2d + 1)^l (1 - p)^{l / (2M + 1)^{d + 1}} = 0, \end{array} $$
 where the last equation follows from the condition on~$p$.
\end{proof}


\section{Perturbation arguments}
\label{sec:perturb}
 Note that, when~$\gamma = 0$, mutations from asymptomatic to symptomatic individuals are suppressed, meaning that symptomatic individuals cannot be created.
 In particular, the type~2 individuals die out exponentially fast, so the process is expected to behave eventually as in the absence of 2s, in which case the parameters~$\lambda_2$ and~$\phi_2$ become irrelevant.
 Similarly, when~$\gamma = \infty$, mutations from asymptomatic to symptomatic individuals are instantaneous, so the process is expected to behave as in the absence of~1s, in which case the parameters~$\lambda_1$ and~$\phi_1$ become irrelevant.
 In case one can control the speed of convergence of the process in the absence of type~2, respectively, type~1, using a block construction, the same conclusion typically holds when~$\gamma > 0$ is sufficiently small, respectively, when~$\gamma < \infty$ is sufficiently large.
 The idea is to prove that the processes with~$\gamma = 0$ and~$\gamma > 0$ small, respectively, the processes with~$\gamma = \infty$ and~$\gamma < \infty$ large, exhibit essentially the same behavior with high probability in each given block.
 The goal of this section is to formalize this technique, called a perturbation argument. \\
\indent
 Consider two integers~$L$ and~$T$ which will be interpreted later as the space and time scales in a block construction, and another integer~$M$ which will be interpreted as the range of some underlying percolation process.
 Consider also the space-time block
 $$ \B = [- ML, ML]^d \times \{1, 2, \ldots, N \} \times [0, MT] \subset \Z^d \times \{1, 2, \ldots, N \} \times \R_+. $$
 The parameters being fixed and finite, assume that the process is constructed from the graphical representation introduced in Section~\ref{sec:GR}, and let
 $$ \begin{array}{rcl}
     \Lambda_i & \n = \n & \{(x, i, t) \in \B : \hbox{there is a~$\circ$ at~$(x, i, t)$} \}, \vspace*{4pt} \\
       \Lambda & \n = \n & \{(x, i, t) \in \B : \hbox{there is a~\bul{1} or a~\bul{2} or the head of an arrow at~$(x, i, t)$} \}, \end{array} $$
 be respectively the set of space-time points in~$\B$ at which there is a potential mutation at level~$i$, and the set of space-time points in~$\B$ at which there is a potential infection.
 Note that, by definition of our graphical representation, arrows in the set~$\Lambda$ must point at~$(x, 1, t)$.
 Define also the set of times at which there is a potential infection in the block:
 $$ T_{\Lambda} = \{t \in [0, MT] : (x, i, t) \in \Lambda \ \hbox{for some} \ (x, i) \}. $$
 The event of interest to study the process with parameter~$\gamma$ small is
 $$ G_1 = \{\Lambda_i = \varnothing \ \forall \,i = 1, 2, \ldots, N \}. $$
 Note that, if the bottom~(time~0) of the block~$\B$ does not contain any~2s then, on the event~$G_1$, no~2s will appear in the block~$\B$.
 The following lemma shows that this event occurs with probability arbitrarily close to one when the parameter~$\gamma$ is small.
\begin{lemma}
\label{lem:perturb-1}
 For all~$\ep > 0$, there exists~$\gamma_- > 0$ such that~$P (G_1) > 1  - \ep$ for all~$\gamma < \gamma_-$.
\end{lemma}
\begin{proof}
 The number of mutations in the space-time block~$\B$ is given by
 $$ X = \poisson (\gamma (2ML + 1)^d NMT), $$
 from which it follows that
 $$ P (G_1) = P (X = 0) = \exp (- \gamma (2ML + 1)^d NMT) > 1 - \ep $$
 for all~$\gamma > 0$ sufficiently small.
\end{proof} \\ \\
 The event of interest to study the process with parameter~$\gamma$ large is
 $$ \textstyle G_2 = \{\forall \,(x, i, t) \in \Lambda, \,\max_{j = 1, 2, \ldots, N} \inf \{s > t : (x, j, s) \in \Lambda_j \} < \inf (T_{\Lambda} \cap (t, MT]) \}. $$
 This is the event that, each time there is a new infection at a cluster in the block, all~$N$ levels of this cluster experience a mutation~(so if a~1 is created at any level, it becomes~2) before any other potential infection in the block.
 In particular, on the event~$G_2$, the process in the block behaves essentially as if there were no~1s.
 The following lemma shows that this event occurs with probability arbitrarily close to one when the parameter~$\gamma$ is large.
\begin{lemma}
\label{lem:perturb-2}
 For all~$\ep > 0$, there exists~$\gamma_+ <\infty$ such that~$P (G_2) > 1  - \ep$ for all~$\gamma > \gamma_+$.
\end{lemma}
\begin{proof}
 The first step is to control~$Y = \card (\Lambda)$, which corresponds to the random number of potential infections in the space-time block.
 Because potential internal infections originate from each level of each cluster at rate~$\phi_1 + \phi_2$, while external infections originate from each level of each cluster to each of the~$2d$ neighbors at rate~$\lambda_1 + \lambda_2$, we have
 $$ Y = \poisson (\lambda) \quad \hbox{where} \quad \lambda = (2ML + 1)^d \times (2dN (\lambda_1 + \lambda_2) + N (\phi_1 + \phi_2)) \times MT. $$
 In particular, the probability-generating function of~$Y$ is given by
\begin{equation}
\label{eq:PGF}
\phi_Y (s) = E (s^Y) = \sum_{k = 0}^{\infty} \,s^k P (Y = k) = \sum_{k = 0}^{\infty} \,\frac{s^k \lambda^k}{k!} \ e^{- \lambda} = e^{- \lambda (1 - s)}.
\end{equation}
 By Markov's inequality, there exists~$c = c (\ep, \lambda) < \infty$ such that
\begin{equation}
\label{eq:perturb-21}
  P (Y > c \lambda) = P (2^Y > 2^{c \lambda}) \leq 2^{- c \lambda} E (2^Y) = 2^{- c \lambda} \phi_Y (2) = (2^c / e)^{- \lambda} \leq \ep/2.
\end{equation}
 In addition, because potential infections in the block occur at rate~$\lambda$, while potential mutations at each level of each cluster occur at rate~$\gamma$, by the superposition property for independent Poisson processes and an application of the union bound over all~$N$ levels, we have
\begin{equation}
\label{eq:perturb-22}
\begin{array}{l}
  P (\max_{j = 1, 2, \ldots, N} \inf \{s > t : (x, j, s) \in \Lambda_j \} > \inf (T_{\Lambda} \cap (t, MT]) \}) \vspace*{4pt} \\ \hspace*{20pt} \leq
  N P (\inf \{s > t : (x, 1, s) \in \Lambda_1 \} > \inf (T_{\Lambda} \cap (t, MT]) \}) \vspace*{4pt} \\ \hspace*{20pt} =
  N \lambda / (\gamma + \lambda) \leq \ep / (2c \lambda) \end{array}
\end{equation}
 for all~$\gamma < \infty$ sufficiently large.
 Combining~\eqref{eq:perturb-21} and~\eqref{eq:perturb-22}, we conclude that
 $$ P (G_2^c) \leq P (G_2^c \,| \,Y \leq c \lambda) + P (Y > c \lambda) \leq c \lambda \times \ep / (2c \lambda) + \ep/2 = \ep $$
 for all~$\gamma < \infty$ sufficiently large.
\end{proof}


\section{Proof of Theorems~\ref{th:lambdalarge}.b}
\label{sec:lambdalarge}
 As previously mentioned, a major difficulty in proving Theorem~\ref{th:lambdalarge}.b, as well as the remaining theorems, is that, for arbitrary values of the infection rates, our epidemic model is not generally monotone with respect to the mutation rate~$\gamma$ or the cluster size~$N$.
 In particular, the couplings introduced in Section~\ref{sec:GR} fail to provide any useful ordering with other processes when~$\lambda_1 > \lambda_2$, thus preventing a direct comparison with either the individual recovery process or the contact process with an asymptomatic state.
 To prove the theorem, we use instead a block construction to show survival, but also control the rate of expansion of the process, in the limiting case~$\gamma = 0$, and then apply a perturbation argument to extend the result to~$\gamma > 0$ small. \\
\indent
 Setting~$\gamma = 0$, the~2s initial present die exponentially fast, so our process essentially reduces to the individual recovery process~$\irp (\lambda_1, \phi_1, N)$.
 By~\cite[Lemma~1.1]{belhadji_lanchier_2005}, this process is monotone with respect to~$N$, so it dominates the standard contact process~$\zeta_t$ with parameter~$\lambda_1$, in the sense that, constructing the contact process using the~1-arrows and the death marks at level~1 from the graphical representation of our process, in the absence of~2s,
 $$ \begin{array}{rcl}
    \zeta_0 (x) \leq \xi_0 (x, 1) \ \ \forall \,x \in \Z^d & \Longrightarrow & \zeta_t (x) \leq \xi_t (x, 1) \ \ \forall \,(x, t) \in \Z^d \times \R_+.
    \end{array} $$
 Because~$\lambda_1 > \lambda_c$, this underlying contact process is supercritical and survives.
 While the block construction of Bezuidenhout and Grimmett~\cite{bezuidenhout_grimmett_1990} is the standard tool to compare a supercritical contact process with oriented percolation, we will start instead from the more recent block construction developed by Durrett and Neuhauser~\cite{durrett_neuhauser_1997} for the multitype contact process that gives us more control on the~1s.
 More precisely, we first get rid of the~2s to give space for the~1s to spread, then use the construction from~\cite{durrett_neuhauser_1997} to move the~1s around in the space-time region void of~2s, and finally apply a perturbation argument to extend the result to~$\gamma > 0$ small. \\
\indent
 First, we let~$L$ be a large scale parameter to be fixed later.
 Recalling the first percolation model introduced in Section~\ref{sec:perco}, for all~$(m, n) \in \lat_+$, we consider the space-time blocks
 $$ \begin{array}{rcl}
	 \A_{m, n}  & \n = \n & (mLe_1 + [-L, L]^d) \times \{1, 2, \ldots, N \} \times \{2nL^2 \}, \vspace*{4pt} \\
	 \A_{m, n}' & \n = \n & [mLe_1 + [-L, L]^d) \times \{1, 2, \ldots, N \} \times \{(2n + 1) L^2 \}, \vspace*{4pt} \\
	 \B_{m, n}  & \n = \n & (mLe_1 + [- 2L, 2L]^d) \times \{1, 2, \ldots, N \} \times [2nL^2, (2n + 1) L^2], \vspace*{4pt} \\
     \B_{m, n}' & \n = \n & (mLe_1 + [- 3L, 3L]^d) \times \{1, 2, \ldots, N \} \times [(2n + 1) L^2, (2n + 2) L^2], \vspace*{4pt} \\
     \C_{m, n}  & \n = \n & (mLe_1 + [- 3L, 3L]^d) \times \{1, 2, \ldots, N \} \times [2nL^2, (2n + 2) L^2], \end{array} $$
 where~$e_1$ is the first unit vector in~$\Z^d$.
 These regions are shown in Figure~\ref{fig:spread1}.
\begin{figure}[t!]
\centering
\scalebox{0.40}{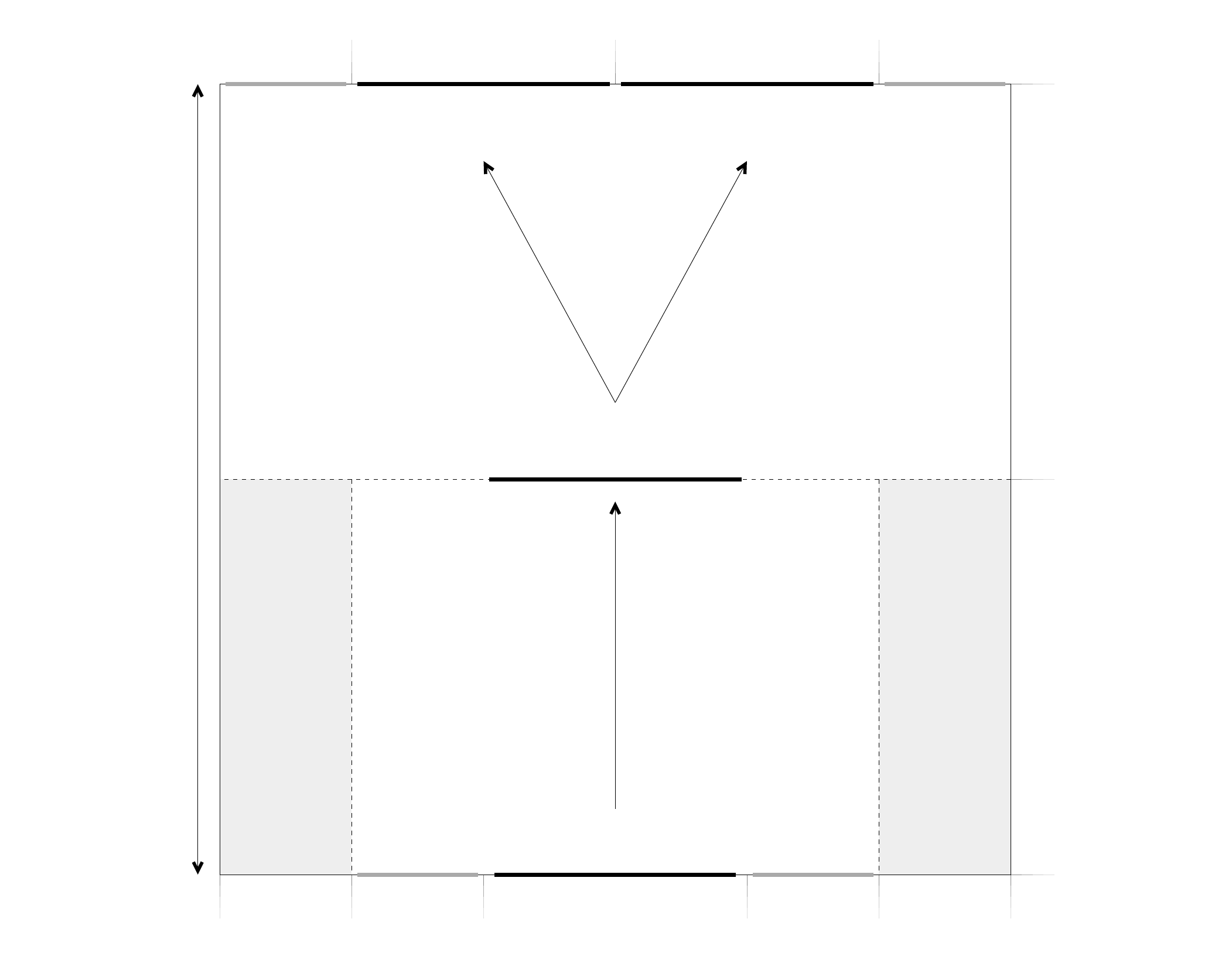}
\caption{\upshape{
  Picture of the blocks and events used in the proof of Theorem~\ref{th:lambdalarge}.b when~$d = 1$.
  The white region in the picture is the region void of~2s where the process dominates a supercritical contact process.}}
\label{fig:spread1}
\end{figure}
 To shorten the notations, in this section and the next ones, we also consider the projections
 $$ \pi_1 (x, i, t) = x, \quad \pi_{12} (x, i, t) = (x, i), \quad \pi_{13} (x, i, t) = (x, t), $$
 defined on~$\Z^d \times \{1, 2, \ldots, N \} \times \R_+$.
 Following Durrett and Neuhauser~\cite{durrett_neuhauser_1997}, we partition the flat region~$\A_{m, n}$ into small cubes of side length $L^{0.1}$, and define the good events
 $$ \begin{array}{rcl}
      E_{m, n} & \n = \n & \{\hbox{every small cube in~$\A_{m, n}$ contains at least one~1} \}, \vspace*{4pt} \\
      F_{m, n} & \n = \n & \{\hbox{there are no~2s in~$\pi_{12} (\B_{m, n})$ at time~$2nL^2$} \}. \end{array} $$
 In words, the first event indicates that there are many~1s in the flat region~$\A_{m, n}$ and the second event that there are no~2s at the bottom of the block~$\B_{m, n}$.
 Similarly, we partition the flat region~$\A_{m, n}'$ into small cubes of side length $L^{0.1}$, and define
 $$ \begin{array}{rcl}
      E_{m, n}' & \n = \n & \{\hbox{every small cube in~$\A_{m, n}'$ contains at least one~1} \}, \vspace*{4pt} \\
      F_{m, n}' & \n = \n & \{\hbox{there are no~2s in~$\pi_{12} (\B_{m, n}')$ at time~$(2n + 1) L^2$} \}. \end{array} $$
 We first prove that, with probability close to one when~$L$ is large, the potential~2s around the space-time block~$\B_{m, n}$ die by time~$(2n + 1) L^2$, thereby preventing them from obstructing the spatial expansion of the~1s.
 To do so, we introduce our first good event
 $$ \begin{array}{l}
      G_{m, n}^1 = \{\forall \,(x, i) \in \pi_{12} (\C_{m, n} \setminus \B_{m, n}), \vspace*{4pt} \\
                   \hspace*{60pt} \exists \,t \in [2nL^2, (2n + 1) L^2) : \hbox{there is a~$\times$ at~$(x, i, t)$} \}, \end{array} $$
 the event that there is at least one death mark at each site in the gray regions of Figure~\ref{fig:spread1}.
 The following lemma shows that this event occurs with probability close to one.
\begin{lemma}
\label{lem:lambdalarge-clean}
 For all~$\ep > 0$,~$P (G_{m, n}^1) \geq 1 - \ep / 4$ for all~$L$ large.
\end{lemma}
\begin{proof}
 Because the death marks occur at each site at rate one, and because
 $$ \begin{array}{rcl}
    \pi_{12} (\C_{m, n} \setminus \B_{m, n}) & \hbox{has} & ((6L + 1)^d - (4L + 1)^d) N \leq (6L + 1)^d N \ \hbox{sites}, \vspace*{4pt} \\
                       (2nL^2, (2n + 1) L^2) & \hbox{has} & L^2 \ \hbox{units of time}, \end{array} $$
 for all~$L$ large, the probability of the event~$G_{m, n}^1$ is given by
 $$ \begin{array}{rcl}
      P (G_{m, n}^1) & \n \geq \n & \Big(P (\poisson (L^2) \neq 0) \Big)^{(6L + 1)^d N} = \Big(1 - e^{-L^2} \Big)^{(6L + 1)^d N} \vspace*{4pt} \\
                     & \n = \n & \exp \Big((6L + 1)^d N \ln (1 - e^{-L^2}) \Big) \approx \exp \Big(- (6L + 1)^d N e^{-L^2} \Big). \end{array} $$
 The lemma follows noticing that this goes to one as~$L$ goes to infinity.
\end{proof} \\ \\
 On the event~$G_{m, n}^1$, all the~2s at the bottom of~$\C_{m, n} \setminus \B_{m, n}$ die by time~$(2n + 1) L^2$.
 Assuming in addition that~$\gamma = 0$, in which case no~2s can be created, the block~$\B_{m, n}'$, which includes the bottom of~$\B_{m \pm 1, n + 1}$, does not contain any~2s.
 In terms of events, this can be written as
\begin{equation}
\label{eq:lambdalarge-no2}
  F_{m, n} \cap G_{m, n}^1 \subset F_{m, n}' \subset F_{m - 1, n + 1} \cap F_{m + 1, n + 1} \quad \hbox{when} \quad \gamma = 0.
\end{equation}
 We can now prove that the~1s persist through the block~$\B_{m, n}$, which is done in Lemma~\ref{lem:persistence}, and then expand through the block~$\B_{m, n}'$, which is done in Lemma~\ref{lem:expansion}.
\begin{lemma}
\label{lem:persistence}
 Let~$\gamma = 0$ and~$\lambda_1 > \lambda_c$.
 Then, for all~$\ep > 0$, there exists an event~$G_{m, n}^2$ measurable with respect to the graphical representation in~$\B_{m, n}$ such that
 $$ E_{m, n} \cap F_{m, n} \cap G_{m, n}^2 \subset E_{m, n}' \quad \hbox{and} \quad P (G_{m, n}^2) \geq 1 - \ep / 4 \ \hbox{for all~$L$ large}. $$
\end{lemma}
\begin{proof}
 The key ingredient is the construction due to Durrett and Neuhauser~\cite{durrett_neuhauser_1997}, which gives the result for the supercritical contact process.
 More precisely, let~$\zeta_t$ be the standard contact process with parameter~$\lambda_1 > \lambda_c$ and configuration
 $$ \zeta_{2nL^2} (x) = \ind \{\xi_{2nL^2} (x, 1) \} \ \ \hbox{for all} \ \ x \in \Z^d \quad \hbox{at time} \ 2nL^2, $$
 so the process has at least one~1 in each small cube in the flat region~$\A_{m, n}$.
 Then, it follows from the proof of~\cite[Proposition~3.1]{durrett_neuhauser_1997} that each small cube in~$\A_{m, n}'$ is also occupied by at least one~1 with probability arbitrarily close to one when~$L$ is sufficiently large.
 Their proof relies on the existence of dual paths starting from~$\A_{m, n}'$ and landing on occupied sites in~$\A_{m, n}$ that they can redirect using a so-called repositioning algorithm.
 In addition, \cite[Lemma~3.7]{durrett_neuhauser_1997} shows that the position of these dual paths is controlled by large deviation estimates giving probabilities that are exponentially close to one, whereas the number of dual paths starting from~$\A_{m, n}'$ is only polynomial.
 In particular, with probability arbitrarily close to one when~$L$ is large, all the dual paths constructed from the repositioning algorithm are contained in~$\B_{m, n}$.
 This implies the existence of an event~$G_{m, n}^2$ measurable with respect to the graphical representation in~$\B_{m, n}$ such that
 $$ E_{m, n} \cap G_{m, n}^2 \subset E_{m, n}' \quad \hbox{and} \quad P (G_{m, n}^2) \geq 1 - \ep / 4 \ \hbox{for all~$L$ large} $$
 for the contact process~$\zeta_t$.
 To deduce the lemma, recall that, in the absence of~2s initially, and when~$\gamma = 0$, which prevents the creation of~2s later on, we have
 $$ \xi_t = \irpa (\lambda_1, \lambda_2, \phi_1, \phi_2, \gamma, N) = \irp (\lambda_1, \phi_1, N) \geq \cp (\lambda_1) = \zeta_t. $$
 Since in addition the event~$F_{m, n}$ and the assumption~$\gamma = 0$ force the absence of~2s in the block~$\B_{m, n}$, and the event~$G_{m, n}^2$ is measurable with respect to the graphical representation in this block~(the dual paths of interest are contained in this block), we deduce that
 $$ E_{m, n} \cap F_{m, n} \cap G_{m, n}^2 \subset E_{m, n}' \quad \hbox{when} \quad \gamma = 0 $$
 for the epidemic model~$\xi_t$.
 This completes the proof.
\end{proof} \\ \\
 Using the same argument, we can also prove expansion through~$\B_{m, n}'$.
\begin{lemma}
\label{lem:expansion}
 Let~$\gamma = 0$ and~$\lambda_1 > \lambda_c$.
 Then, for all~$\ep > 0$, there exists an event~$G_{m, n}^3$ measurable with respect to the graphical representation in~$\B_{m, n}'$ such that
 $$ E_{m, n}' \cap F_{m, n}' \cap G_{m, n}^3 \subset E_{m - 1, n + 1} \cap E_{m + 1, n + 1} \quad \hbox{and} \quad P (G_{m, n}^3) \geq 1 - \ep / 4 \ \hbox{for all~$L$ large}. $$
\end{lemma}
\begin{proof}
 As previously, the repositioning algorithm~\cite{durrett_neuhauser_1997} implies the existence of an event~$G_{m, n}^3$ measurable with respect to the graphical representation in~$\B_{m, n}'$ such that
 $$ E_{m, n}' \cap G_{m, n}^3 \subset E_{m - 1, n + 1} \cap E_{m - 1, n + 1} \quad \hbox{and} \quad P (G_{m, n}^3) \geq 1 - \ep / 4 \ \hbox{for all~$L$ large} $$
 for the contact process~$\zeta_t$ with parameter~$\lambda_1 > \lambda_c$.
 Because~$F_{m, n}'$ and the assumption~$\gamma = 0$ now force the absence of~2s in the block~$\B_{m, n}'$, we deduce that
 $$ E_{m, n}' \cap F_{m, n}' \cap G_{m, n}^3 \subset E_{m - 1, n + 1} \cap E_{m + 1, n + 1} \quad \hbox{when} \quad \gamma = 0 $$
 for the epidemic model~$\xi_t$.
\end{proof} \\ \\
 Using Lemma~\ref{lem:perco-1}~(percolation result), Lemma~\ref{lem:perturb-1}~(perturbation argument) and the previous three lemmas, we can now prove the theorem. \\ \\
\begin{proofof}{Theorem~\ref{th:lambdalarge}.b}
 We say that site~$(m, n) \in \lat_+$ is a good site if
 $$ E_{m, n} \cap F_{m, n} = \left\{\begin{array}{c}
             \hbox{every small cube in~$\A_{m, n}$ contains at least one~1} \vspace*{2pt} \\
             \hbox{and there are no~2s in~$\pi_{12} (\B_{m, n})$ at time~$2nL^2$} \end{array} \right\} $$
 occurs.
 Take~$\ep < 3^{- 4 (2M + 1)^2}$ with~$M = 6$ since
 $$ \C_{m, n} \cap \C_{m', n'} = \varnothing \quad \hbox{whenever} \quad |m - m'| > 6 \ \ \hbox{or} \ \ |n - n'| > 1. $$
 To study the process with~$\gamma > 0$ small, we also define the events
 $$ G_{m, n}^4 = \{\hbox{there are no~$\circ$ in~$\C_{m, n}$} \} \quad \hbox{and} \quad G_{m, n} = G_{m, n}^1 \cap G_{m, n}^2 \cap G_{m, n}^3 \cap G_{m, n}^4. $$
 Because no~2 can be created in~$\C_{m, n}$ on the event~$G_{m, n}^4$, it follows from~\eqref{eq:lambdalarge-no2} that
 $$ F_{m, n} \cap G_{m, n} \subset F_{m, n} \cap G_{m, n}^1 \cap G_{m, n}^4 \subset F_{m, n}' \cap G_{m, n}^4 \subset F_{m - 1, n + 1} \cap F_{m + 1, n + 1}. $$
 In fact, on the event~$F_{m, n} \cap G_{m, n}$, the white subset of the space-time block~$\C_{m, n}$ in Figure~\ref{fig:spread1} does not contain any~2s.
 Using also the first part of Lemmas~\ref{lem:persistence} and~\ref{lem:expansion}, we deduce that
\begin{equation}
\label{eq:lambdalarge-inclusion}
\begin{array}{rcl}
   E_{m, n} \cap F_{m, n} \cap G_{m, n} & \n \subset \n &
   E_{m, n}' \cap F_{m, n}' \cap G_{m, n} \vspace*{4pt} \\ & \n \subset \n &
  (E_{m - 1, n + 1} \cap F_{m - 1, n + 1}) \cap (E_{m + 1, n + 1} \cap F_{m + 1, n + 1}). \end{array}
\end{equation}
 In addition, by Lemma~\ref{lem:lambdalarge-clean} and the second part of Lemmas~\ref{lem:persistence} and~\ref{lem:expansion}, we can find
 $$ L = L (\ep) < \infty \ \ \hbox{such that} \ \ P (G_{m, n}^i) \geq 1 - \ep / 4 \ \ \hbox{for} \ \ i = 1, 2, 3. $$
 Once~$L$ is fixed, using Lemma~\ref{lem:perturb-1} with~$\C_{m, n}$ instead of~$\B$, we can also find
 $$ \gamma_- = \gamma_- (\ep, L) > 0 \ \ \hbox{such that} \ \ P (G_{m, n}^4) \geq 1 - \ep / 4 \ \ \hbox{for all} \ \ \gamma < \gamma_-. $$ 
 Using the union bound, we deduce that
\begin{equation}
\label{eq:lambdalarge-proba}
\textstyle P (G_{m, n}) \geq 1 - \sum_{i = 1, 2, 3, 4} P ((G_{m, n}^i)^c) \geq 1 - \ep.
\end{equation}
 In view of~\eqref{eq:lambdalarge-inclusion} and~\eqref{eq:lambdalarge-proba}, we may apply the comparison result~\cite[Theorem~4.3]{durrett_1995} to deduce that the set of good sites dominates stochastically the set of open sites in a~6-dependent oriented site percolation process on the directed graph $\vec{\lat}_+$ with density parameter $1 - \epsilon$.
 Lemma~\ref{lem:perco-1} and our choice of~$\epsilon$ imply that the open sites percolate, so the good sites percolate as well, which proves survival and completes the proof of the theorem.
\end{proofof}


\section{Proof of Theorems~\ref{th:philarge}.b and~\ref{th:philarge}.c}
\label{sec:philarge}
 As previously, the proof relies on the combination of a block construction and a perturbation argument.
 In the limiting case~$\gamma = 0$, the process essentially behaves as if there were no~2s so, once infected, a cluster will remain infected for an arbitrarily large time when~$\phi_1$ is large, and assuming in addition that~$\lambda_1 > 0$ even small will result in survival.
 In the limiting case~$\gamma = \infty$, because all the~1s turn instantaneously into~2s and so the process behaves as if there were no~1s, the same conclusion holds after reversing the roles of the 1s and the 2s.
 To turn our intuition into a rigorous proof, the idea is to use a block construction to control the rate of growth of the infection for extreme values of the mutation rate~$\gamma$, and then apply a perturbation argument to extend the result to the parameter regions~$\gamma > 0$ small and~$\gamma < \infty$ large. \\
\indent
 To begin with, let~$L$ be a large scale parameter to be fixed later and, for all~$(m, n) \in \lat_+$, we now consider the space-time blocks
 $$ \C_{m, n} = \{(m - 1) e_1, me_1, (m + 1) e_1 \} \times \{1, 2, \ldots, N \} \times [2nL, (2n + 2) L], $$
 where~$e_1$ again refers to the ﬁrst unit vector in $\Z^d$.
 This region is shown in  Figure~\ref{fig:spread2}.
\begin{figure}[t!]
\centering
\scalebox{0.40}{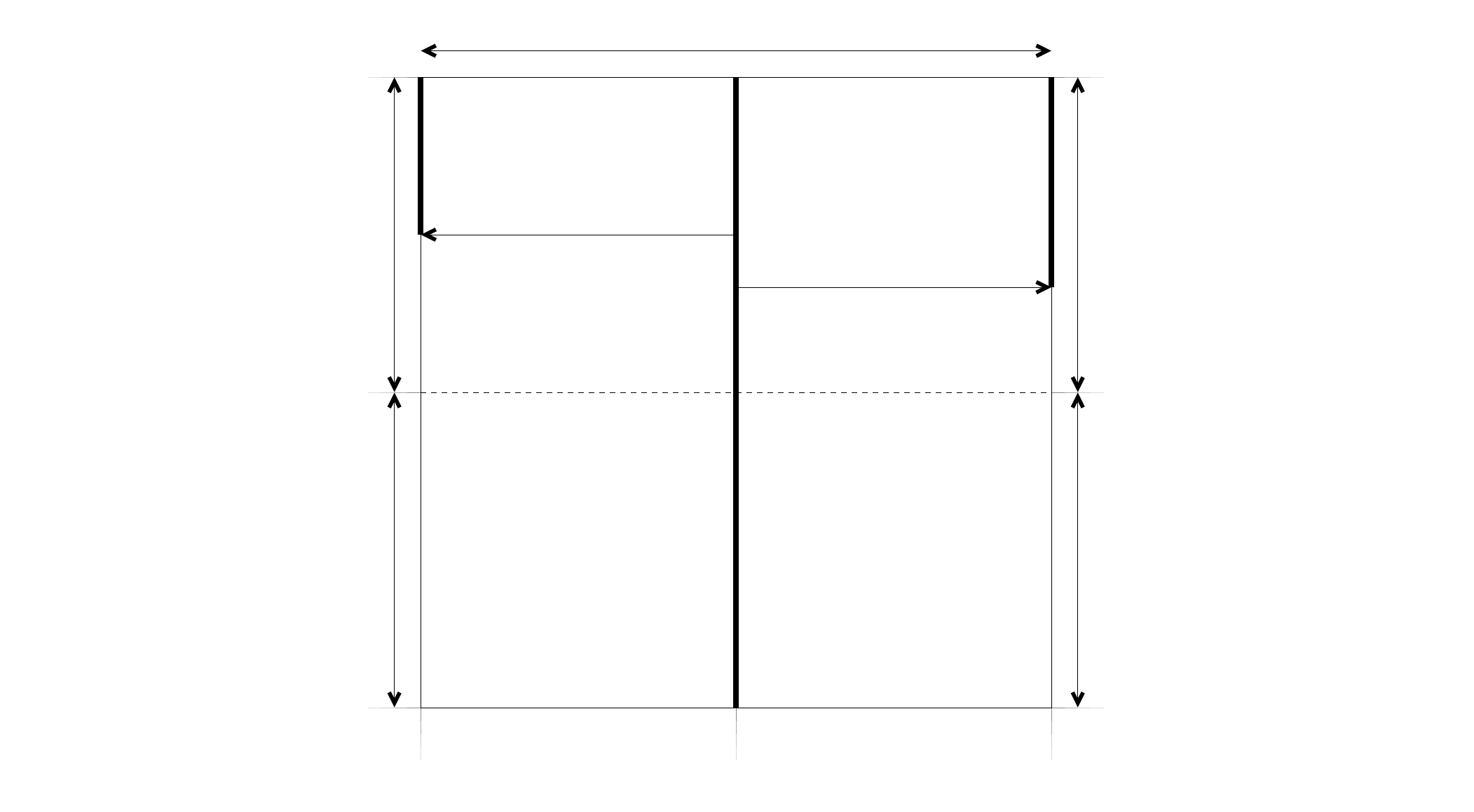}
\caption{\upshape{
  Picture of the blocks and events used in the proof of Theorem~\ref{th:philarge}.b when~$d = 1$.
  The bold lines indicate the space-time regions occupied by at least one~1 and no~2s.}}
\label{fig:spread2}
\end{figure}
 We first deal with the no mutation~$\gamma = 0$ case in which the~2s cannot be created.
 To proceed with the block construction, we consider the collection of good events
 $$ \begin{array}{rcl}
      E_{m, n} & \n = \n & \{\hbox{cluster~$m e_1$ contains at least one~1, but no~2s, at time~$2nL$} \} \vspace*{4pt} \\
               & \n = \n & \{\xi_{2nL} (m e_1, 1) = 1 \}. \end{array} $$
 We first define three good events that altogether guarantee that, whenever~$E_{m, n}$ occurs, the two events~$E_{m \pm 1, n + 1}$ occur as well. 
 As previously, the first step is to give space for the~1s to spread by getting rid of the~2s on both sides of~$m e_1$.
 More precisely, we prove that the event
 $$ \begin{array}{l}
      G_{m, n}^1 = \{\forall \,(x, i) \in (m \pm 1) e_1 \times \{1, 2, \ldots, N \}, \vspace*{4pt} \\
                   \hspace*{60pt} \exists \,t \in [2nL, (2n + 1) L) : \hbox{there is a~$\times$ at~$(x, i, t)$} \} \end{array} $$
 that there is at least one death mark at each site of the clusters~$(m \pm 1) e_1$ between time~$2nL$ and time~$(2n + 1) L$ occurs with probability close to one when~$L$ is large.
\begin{lemma}
\label{lem:philarge-clean}
 For all~$\ep > 0$,~$P (G_{m, n}^1) \geq 1 - \ep / 4$ for all~$L$ large.
\end{lemma}
\begin{proof}
 Using the same approach as in the proof of Lemma~\ref{lem:lambdalarge-clean}, we have
 $$ P (G_{m, n}^1) = \Big(P (\poisson (L) \neq 0) \Big)^{2N} = \Big(1 - e^{-L} \Big)^{2N}, $$
 which goes to one as~$L$ goes to infinity.
\end{proof} \\ \\
 Our second good event imposes at least two~1-arrows from the first level of cluster~$m e_1$ pointing at the two clusters~$(m \pm 1) e_1$ during the next~$L$ units of time:
 $$ G_{m, n}^2 = \{\forall \,x = (m \pm 1) e_1, \exists \,t \in [(2n + 1) L, (2n + 2) L) : (m e_1, 1, t) \arrow{1} (x, 1, t) \}. $$
 This event again occurs with probability close to one when~$L$ is large.
\begin{lemma}
\label{lem:philarge-arrow}
 For all~$\lambda_1 > 0$ and~$\ep > 0$,~$P (G_{m, n}^2) \geq 1 - \ep / 4$ for all~$L$ large.
\end{lemma}
\begin{proof}
 Since~1-arrows~$(m e_1, 1) \arrow{1} ((m \pm 1) e_1, 1)$ occur at rate~$\lambda_1 > 0$,
 $$ P (G_{m, n}^2) = \Big(P (\poisson (\lambda_1 L) \neq 0) \Big)^2 = \Big(1 - e^{-\lambda_1 L} \Big)^2, $$
 which goes to one as~$L$ goes to infinity.
\end{proof} \\ \\
 Our third good event~$G_{m, n}^3$ is more complicated and consists of three components.
\begin{itemize}
\item
 There is at least one~\bul{1} at site~$(m e_1, 1)$ between time~$2nL$ and the first~$\times$ at cluster~$m e_1$, and between any two death marks~$\times$ at cluster~$m e_1$. \vspace*{4pt}
\item
 From time~$(2n + 1) L$ to time~$(2n + 2) L$, there is at least one~\bul{1} at site~$((m \pm 1) e_1, 1)$ between each~$(m e_1, 1) \arrow{1} ((m \pm 1) e_1, 1)$ and the next~$\times$ at~$(m \pm 1) e_1$. \vspace*{4pt}
\item
 From time~$(2n + 1) L$ to time~$(2n + 2) L$, there is at least one~\bul{1} at site~$((m \pm 1) e_1, 1)$ between any two death marks~$\times$ at cluster~$(m \pm 1) e_1$. \vspace*{4pt}
\end{itemize}
 This event guarantees that, starting with at least one~1 and no~2s at~$m e_1$ at time~$2nL$, this property remains true until time~$(2n + 2) L$.
 In addition, if there are no~2s at~$(m \pm 1) e_1$ at time~$(2n + 1) L$ and there are at least two~1-arrows~$(m e_1, 1) \arrow{1} ((m \pm 1) e_1, 1)$ between times~$(2n + 1) L$ and~$(2n + 2) L$, then there will be at least one~1 and no~2s at~$(m \pm 1) e_1$ from the times of these arrows until at least time~$(2n + 2) L$.
 In terms of events, this can be written as
\begin{equation}
\label{eq:philarge}
  E_{m, n} \cap G_{m, n}^1 \cap G_{m, n}^2 \cap G_{m, n}^3 \subset E_{m - 1, n + 1} \cap E_{m + 1, n + 1} \quad \hbox{when} \quad \gamma = 0.
\end{equation}
 The next lemma also shows that, the scale parameter~$L$ being fixed, the event~$G_{m, n}^3$ occurs with probability arbitrarily close to one when~$\phi_1$ is large.
\begin{lemma}
\label{lem:perturbphi}
 For all~$\ep > 0$, there exists~$\phi_+ < \infty$ such that
 $$ P(G_{m,n}^3) \geq 1 - \ep/4 \quad \hbox{for all} \quad \phi_1 > \phi_+. $$
\end{lemma}
\begin{proof}
 This is similar to the proof of Lemma~\ref{lem:perturb-2} so we briefly explain the difference between the two proofs.
 We consider three clusters for a time interval of length~$2L$.
 In addition, death marks occur at each site at rate one, while arrows~$(m e_1, 1) \arrow{1} ((m \pm 1) e_1, 1)$ occur at rate~$\lambda_1$.
 In particular, we replace the random variable~$Y$ in the proof of Lemma~\ref{lem:perturb-2} by
 $$ Y = \poisson (\lambda) \quad \hbox{where} \quad \lambda = 3 \times (\lambda_1 + N) \times 2L. $$
 Following the proof of Lemma~\ref{lem:perturb-2} using the cutoff~$\ep / 4$ instead of~$\ep$, we deduce that there exists a constant~$c = c (\ep, \lambda) < \infty$ such that
 $$   P (Y > c \lambda) \leq (2^c / e)^{- \lambda} \leq \ep/8, $$
 and that there exists~$\phi_+ < \infty$ such that
 $$ P ((G_{m, n}^3)^c) \leq P ((G_{m, n}^3)^c \,| \,Y \leq c \lambda) + P (Y > c \lambda) \leq c \lambda \times \ep / (8c \lambda) + \ep/8 = \ep/4 $$
 for all~$\phi_1 > \phi_+$.
 This completes the proof.
\end{proof} \\ \\
 We are now ready to prove the theorem. \\ \\
\begin{proofof}{Theorem~\ref{th:philarge}.b}
 We say that site~$(m, n) \in \lat_+$ is a good site if
 $$ E_{m, n} = \{\hbox{cluster~$m e_1$ contains at least one~1, but no~2s, at time~$2nL$} \} $$
 occurs.
 Take~$\ep < 3^{- 4 (2M + 1)^2}$ with~$M = 2$ since
 $$ \C_{m, n} \cap \C_{m', n'} = \varnothing \quad \hbox{whenever} \quad |m - m'| > 2 \ \ \hbox{or} \ \ |n - n'| > 1. $$
 As previously, to study the process with~$\gamma > 0$ small, we define
 $$ G_{m, n}^4 = \{\hbox{there are no~$\circ$ in~$\C_{m, n}$} \} \quad \hbox{and} \quad G_{m, n} = G_{m, n}^1 \cap G_{m, n}^2 \cap G_{m, n}^3 \cap G_{m, n}^4. $$
 It follows from~\eqref{eq:philarge} and the definition of~$G_{m, n}^4$ that
\begin{equation}
\label{eq:philarge-inclusion}
  E_{m, n} \cap G_{m, n} \subset E_{m - 1, n + 1} \cap E_{m + 1, n + 1}.
\end{equation}
 Now, according to Lemmas~\ref{lem:philarge-clean} and~\ref{lem:philarge-arrow}, there exists
 $$ L = L (\ep) < \infty \ \ \hbox{such that} \ \ P (G_{m, n}^i) \geq 1 - \ep / 4 \ \ \hbox{for} \ \ i = 1, 2. $$
 Once the scale parameter~$L$ is fixed so that the previous two inequalities hold, using Lemma~\ref{lem:perturbphi} as well as Lemma~\ref{lem:perturb-1} with~$\C_{m, n}$ instead of~$\B$, we can also find
 $$ \begin{array}{rclcl}
    \phi_+ = \phi_+ (\ep, L) < \infty & \hbox{such that} & P (G_{m, n}^3) \geq 1 - \ep / 4 & \hbox{for all} & \phi > \phi_+, \vspace*{4pt} \\
    \gamma_- = \gamma_- (\ep, L) > 0  & \hbox{such that} & P (G_{m, n}^4) \geq 1 - \ep / 4 & \hbox{for all} & \gamma < \gamma_-. \end{array} $$
 Using the union bound, we deduce that
\begin{equation}
\label{eq:philarge-proba}
\textstyle P (G_{m, n}) \geq 1 - \sum_{i = 1, 2, 3, 4} P ((G_{m, n}^i)^c) \geq 1 - \ep.
\end{equation}
 As previously, survival follows from~\cite[Theorem~4.3]{durrett_1995}, Lemma~\ref{lem:perco-1} and~\eqref{eq:philarge-inclusion}--\eqref{eq:philarge-proba}.
\end{proofof} \\ \\
\begin{proofof}{Theorem~\ref{th:philarge}.c} This follows the proof of Theorem~\ref{th:philarge}.b with the following exceptions.
\begin{enumerate}
\item
 Lemma~\ref{lem:lambdalarge-clean} still holds under the assumptions of the theorem but is no longer needed because the~1s turn instantaneously into~2s in the limiting case~$\gamma = \infty$. \vspace*{4pt}
\item
 The event~$G_{m, n}^2$ is defined similarly but uses the~$\arrow{2}$ instead of the~$\arrow{1}$. \vspace*{4pt}
\item
 The event~$G_{m,n}^3$ is defined similarly but uses the~\bul{2} instead of the~\bul{1}. \vspace*{4pt}
\item
 Because we need to do a perturbation at~$\gamma = \infty$ instead of~$\gamma = 0$, the event~$G_{m, n}^4$ is now the event that any potential internal or external infection in the block~$\C_{m ,n}$ is followed by a mutation before the next potential infection in the block.
 In particular, the probability of this event is now controlled by Lemma~\ref{lem:perturb-2} instead of Lemma~\ref{lem:perturb-1}.
\end{enumerate}
 This completes the proof of the theorem.
\end{proofof}


\section{Proof of Theorems~\ref{th:Nlarge}.b and~\ref{th:Nlarge}.c}
\label{sec:Nlarge}
 As previously, in the limiting cases~$\gamma = 0$ and~$\gamma = \infty$, the process behaves eventually as if there were only one type of infected individuals.
 In particular, the number of infected individuals in an infected cluster evolves essentially like an asymmetric random walk whose transition rates are controlled by the recovery rate and the relevant internal infection rate.
 Also, the proof relies on random walk estimates, in addition to a block construction and a perturbation argument.
 More precisely, when~$\gamma = 0$, we will prove that, if a cluster has at least~$K = \lfloor N/2 \rfloor$ asymptomatic individuals long enough, then the nearby clusters will reach~$N$ asymptomatic individuals and remain above~$K$ asymptomatic individuals for a long time.
 The same conclusion holds when~$\gamma = \infty$ looking instead at the symptomatic individuals.
 Finally, we will apply a perturbation argument to extend the result to the parameter regions~$\gamma > 0$ small and~$\gamma < \infty$ large. \\
\indent
 To set up the block construction, we again let~$L$ be a large scale parameter to be fixed later, and, for all sites~$(m, n) \in \lat_+$, we consider the space-time blocks
 $$ \C_{m, n} = \{(m - 1) e_1, me_1, (m + 1) e_1 \} \times \{1, 2, \ldots, N \} \times [2nL, (2n + 4) L], $$
 where~$e_1$ again refers to the ﬁrst unit vector in~$\Z^d$.
 This region is shown in Figure~\ref{fig:spread3}.
\begin{figure}[t!]
\centering
\scalebox{0.40}{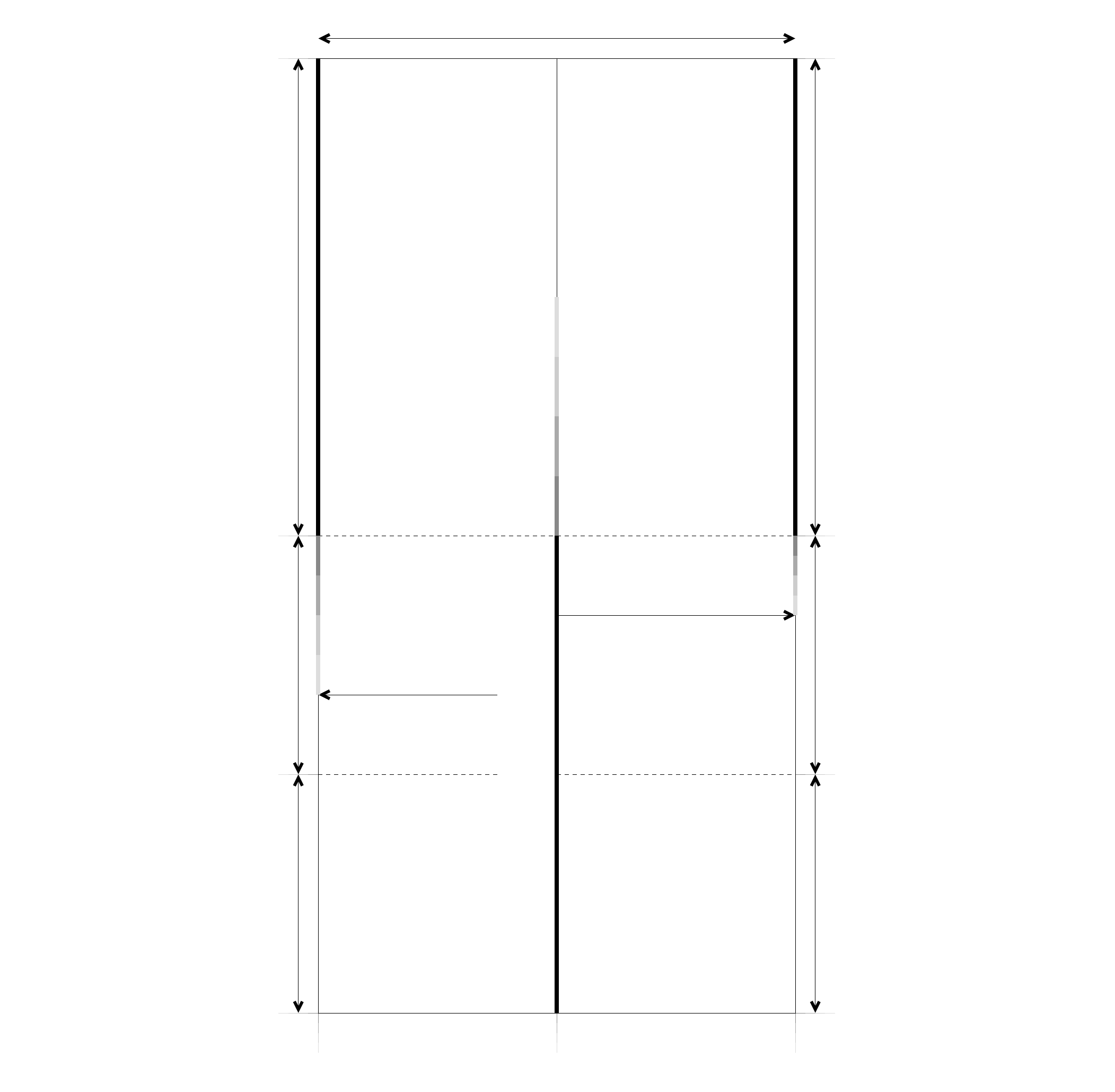}
\caption{\upshape{
 Picture of the blocks and events used in the proof of Theorem~\ref{th:Nlarge}.b when~$d = 1$.
 The bold black lines indicate the space-time regions occupied by at least~$K = \lfloor N/2 \rfloor$~1s and no~2s.}}
\label{fig:spread3}
\end{figure}
 As previously, we first deal with the no mutation~$\gamma = 0$ case in which the~2s cannot be created.
 To proceed with the block construction, we now consider the collection of good events
 $$ \begin{array}{rcl}
      E_{m, n} & \n = \n & \{\hbox{cluster~$m e_1$ contains at least $K$~1s, but no~2s, for all $t\in [2nL, (2n+2)L]$} \} \vspace*{4pt} \\
	           & \n = \n & \{\forall \,t \in [2nL, (2n + 2) L], \ \sum_i \ind \{\xi_{t} (m e_1, i) = 1 \} \geq K \vspace*{4pt} \\
	           &         & \hspace*{94pt} \hbox{and} \ \sum_i \ind \{\xi_{t} (m e_1, i) = 2 \} = 0 \}, \end{array} $$
 where~$K = \lfloor N/2 \rfloor$.
 We now construct four good events that occur with probability close to one, and that, altogether, guarantee that, when~$E_{m, n}$ occurs, the two events~$E_{m \pm 1, n + 1}$ occur as well.
 As in the proof of Theorem~\ref{th:philarge}.b, the first step is to create space for the~1s to spread by killing the~2s on both sides of $m e_1$.
 To this end, the first good event remains unchanged and is defined as
 $$ \begin{array}{l}
      G_{m, n}^1 = \{\forall \,(x, i) \in (m \pm 1) e_1 \times \{1, 2, \ldots, N \}, \vspace*{4pt} \\
                   \hspace*{60pt} \exists \,t \in [2nL, (2n + 1) L) : \hbox{there is a~$\times$ at~$(x, i, t)$} \}. \end{array} $$
 Recall from Lemma~\ref{lem:philarge-clean} that this event occurs with high probability when~$L$ is large.
 In addition, this event ensures that all the~2s initially present at~$(m \pm 1) e_1$ are killed by time~$(2n + 1) L$, so there are no~2s in these two clusters from this time onward when~$\gamma = 0$.
 Our next two good events are again measurable with respect to the graphical representation in the block~$\C_{m, n}$ but are defined implicitly from~$E_{m, n}$ and~$E_{m \pm 1, n + 1}$ and the following event describing the configuration:
 $$ F_{m, n} = \{\hbox{$(m \pm 1) e_1$ contain~$N$ 1s at some time~$t \in [(2n + 1) L, (2n + 2) L]$} \} $$
 To estimate later the conditional probability of these events, note that, in the absence of~2s and when a cluster contains~$i > 0$ asymptomatic individuals, the number of infections in the cluster increases at rate~$i \phi_1$ and decreases at rate~$i$.
 This motivates the introduction and the study of the discrete-time random walk~$X_n$ on the integers with transition probabilities
 $$ \begin{array}{rcl}
      P (X_{n + 1} = i + 1 \,| \,X_n = i) & \n = \n & p = \phi_1 / (\phi_1 + 1), \vspace*{4pt} \\
      P (X_{n + 1} = i - 1 \,| \,X_n = i) & \n = \n & q = 1 / (\phi_1 + 1). \end{array} $$
 For all~$i \in \Z$, let~$P_i$ be the law of the random walk starting from $X_0 = i$ and
 $$ \tau_i^X = \inf \{n \geq 0 : X_n = i \} $$
 denote the first hitting time.
 The random-walk estimates needed for the block construction are collected in the following three lemmas.
 The first lemma provides a large-deviation bound showing that, when starting from the origin, the probability of the hitting time~$\tau_N$ exceeding twice its expected value~$N / (p - q)$ decays exponentially fast with $N$.
\begin{lemma}
\label{lem:RW1}	
 For all~$p > q$ and all~$N$, we have
 $$ P_0 \bigg(\tau_N^X > \frac{2N}{p - q} \bigg) \leq e^{- \frac{(p - q)^3}{6} N} = e^{- \frac{1}{6} \big(\frac{\phi_1 - 1}{\phi_1 + 1} \big)^3 N}. $$
\end{lemma}
\begin{proof}
 For all~$s < 1$ and~$n > 0$, by Markov's inequality,
 $$ \begin{array}{rcl}
      P_0 (\tau_N^X > n) & \n \leq \n & P_0 (X_n < N) = P_0 (s^{X_n} > s^N) \leq s^{-N} E_0 (s^{X_n}) \vspace*{4pt} \\
                         & \n = \n & s^{-N} E (s^{X_1}) \,E (s^{X_2 - X_1}) \,\cdots \,E (s^{X_n - X_{n - 1}}) = s^{-N} (ps + qs^{-1})^n. \end{array} $$
 Taking~$n = 2N / (p - q) > 0$, this becomes
 $$ P_0 \bigg(\tau_N^X > \frac{2N}{p - q} \bigg) \leq \bigg(\frac{1}{s} \bigg)^N \bigg(ps + \frac{q}{s} \bigg)^{\frac{2N}{p - q}} =
        \bigg(\frac{1}{s} \bigg(ps + \frac{q}{s} \bigg)^{\frac{2}{p - q}} \bigg)^N = (a (s))^N, $$
 then taking~$s = \sqrt{q/p} < 1$, we get
 $$ P_0 \bigg(\tau_N^X > \frac{2N}{p - q} \bigg) \leq \bigg(\sqrt{\frac{p}{q}} \,\bigg(p \sqrt{\frac{q}{p}} + q \sqrt{\frac{p}{q}} \bigg)^{\frac{2}{p - q}} \bigg)^N =
        \bigg(\sqrt{\frac{p}{q}} \,\Big(2 \sqrt{pq} \Big)^{\frac{2}{p - q}} \bigg)^N = a^N. $$
 Letting~$x = p - q > 0$ and taking the logarithm, we obtain
 $$ \ln (a) = \ln \bigg(\sqrt{\frac{1 + x}{1 - x}} \,\big(1 - x^2 \big)^{\frac{1}{x}} \bigg) = \frac{\ln (1 + x)}{2} - \frac{\ln (1 - x)}{2} + \frac{\ln (1 - x^2)}{x}. $$
 Finally, using Taylor expansions,
 $$ \begin{array}{rcl}
    \ln (a) & \n = \n &
    \displaystyle - \frac{1}{2} \,\sum_{n = 1}^{\infty} \frac{(-x)^n}{n}
                  + \frac{1}{2} \,\sum_{n = 1}^{\infty} \frac{x^n}{n}
                  - \frac{1}{x} \,\sum_{n = 1}^{\infty} \frac{x^{2n}}{n} \vspace*{8pt} \\ & \n = \n &
    \displaystyle \sum_{n = 1}^{\infty} \frac{x^{2n - 1}}{2n - 1} - \sum_{n = 1}^{\infty} \frac{x^{2n - 1}}{n} =
	\sum_{n = 2}^{\infty} \bigg(\frac{1}{2n - 1} - \frac{1}{n} \bigg) x^{2n - 1} \leq - \frac{x^3}{6}, \end{array} $$
 which proves that~$a \leq e^{- \frac{(p - q)^3}{6}} < 1$.
\end{proof} \\ \\
 Our second lemma controls the probability of a large deviation against the drift.
 Since~$p > q$, the random walk has a positive drift.
 Starting at~$N - 1$, it is exponentially more likely to reach the target~$N$ quickly than to backtrack all the way down to~$K = \lfloor N/2 \rfloor$.
\begin{lemma}
\label{lem:RW2}
 Let~$K = \lfloor N/2 \rfloor$. Then,
 $$ P_{N - 1} (\tau_K^X < \tau_N^N) = \frac{\phi_1^{- (N - 1)} - \phi_1^{-N}}{\phi_1^{-K} - \phi_1^{-N}} \leq \phi_1^{- (N - K - 1)}. $$
\end{lemma}
\begin{proof}
 Note first that~$Y_n = \phi_1^{- X_n}$ is a martingale. Indeed,
 $$ \begin{array}{rcl}
      E (Y_{n + 1} \,| \,Y_n) & \n = \n & E (\phi_1^{- X_{n + 1}} \,| \,X_n) = p \phi_1^{- (X_n + 1)} + q \phi_1^{- (X_n - 1)} \vspace*{4pt} \\
                              & \n = \n & (p \phi_1^{-1} + q \phi_1) \phi_1^{- X_n} = (p (q/p) + q (p/q)) \phi_1^{- X_n} = \phi_1^{- X_n} = Y_n. \end{array} $$
 Starting at~$X_0 = N - 1$, the stopping time~$T = \tau_K \wedge \tau_N$ is almost surely finite, and the process stopped at time~$T$ is bounded, so the optional stopping theorem implies that
 $$ \begin{array}{rcl}
    \phi_1^{- (N - 1)} & \n = \n & E_{N - 1} (\phi_1^{- X_0}) = E_{N - 1} (\phi_1^{- X_T}) \vspace*{4pt} \\
                       & \n = \n & \phi_1^{-K} P_{N - 1} (X_T = K) + \phi_1^{-N} P_{N - 1} (X_T = N) \vspace*{4pt} \\
                       & \n = \n & \phi_1^{-K} P_{N - 1} (\tau_K^X < \tau_N^X) + \phi_1^{-N} P_{N - 1} (\tau_N^X < \tau_K^X) \vspace*{4pt} \\
                       & \n = \n & \phi_1^{-N} + (\phi_1^{-K} - \phi_1^{-N}) P_{N - 1} (\tau_K^X < \tau_N^X). \end{array} $$
 Rearranging the terms gives the desired equation.
\end{proof} \\ \\
 The next lemma is concerned with the continuous-time process~$Z_t$ that describes the number of asymptomatic individuals at cluster~$(m + 1) e_1$, defined as
 $$ \textstyle Z_t = \sum_i \ind \{\xi_t ((m + 1) e_1, i) = 1 \}. $$
 In the absence of symptomatic individuals, and when there are at least~$K$ asymptomatic individuals at~$m e_1$, the process~$Z_t$ evolves according to the transition rates
 $$ \begin{array}{rclclrclcll}
      0 & \n \to \n & 1     & \hbox{at rate} & \hbox{at least} \ K \lambda_1, & \quad
      i & \n \to \n & i + 1 & \hbox{at rate} & i \phi_1                       & \hbox{for} \ i < N, \vspace*{4pt} \\ &&&&&
      i & \n \to \n & i - 1 & \hbox{at rate} & i                              & \hbox{for} \ i \leq N. \end{array} $$
 Our third lemma gives exponential upper and lower bounds for the time~$T (n)$ it takes~$Z_t$ to perform~$n$ jumps, which is used after to compare the process with~$X_n$.
\begin{lemma}
\label{lem:RW3}
 Assume that~$\phi_1, K \lambda_1 > 1$ and let~$\Lambda = \phi_1 \vee \lambda_1 / 2$. Then,
 $$ P (T (n) \geq 2n) \leq (2/e)^n \quad \hbox{and} \quad P (T (n) \leq n / (2N \Lambda)) \leq (\sqrt{e} / 2)^n. $$
\end{lemma}
\begin{proof}
 The moment-generating function of the sum~$S (n, \mu) = S_1 + \cdots + S_n$ of~$n$ independent exponential random variables with parameter~$\mu$ is given by
 $$ \begin{array}{rcl}
      M_{n, \mu} (t) & \n = \n & \displaystyle E \Big(e^{t S (n, \mu)} \Big) = \Big(E \big(e^{t S_i} \big) \Big)^n =
                                 \bigg(\int_0^{\infty} e^{st} (\mu e^{- \mu s}) \,ds \bigg)^n \vspace*{8pt} \\ & \n = \n &
                                 \displaystyle \bigg(\mu \bigg[- \frac{e^{- (\mu - t) s}}{\mu - t} \bigg]_0^{\infty} \bigg)^n =
                                 \bigg(\frac{\mu}{\mu - t} \bigg)^n \quad \hbox{for all} \quad t < \mu. \end{array} $$
 Now, because~$Z_t$ jumps at rate at least one, time~$T (n)$ is stochastically smaller than~$S (n, 1)$.
 In particular, taking~$t = 1/2$, and using Markov's inequality, we get
 $$ \begin{array}{rcl}
      P (T (n) \geq 2n) & \n \leq \n & P (S (n, 1) \geq 2n) = P (e^{S (n, 1) / 2} \geq e^n) \vspace*{4pt} \\
                        & \n \leq \n & e^{-n} E (e^{S (n, 1) / 2}) = e^{-n} M_{n, 1} (1/2) = (2/e)^n. \end{array} $$
 Similarly, because~$Z_t$ jumps at rate at most~$N \Lambda$, time~$T (n)$ is stochastically larger than~$S (n, N \Lambda)$, therefore, taking~$t = - N \Lambda$ and using Markov's inequality, we get
 $$ \begin{array}{l}
      P (T (n) \leq n / (2N \Lambda)) \leq P (S (n, N \Lambda) \leq n / (2N \Lambda)) \vspace*{6pt} \\ \hspace*{40pt} =
      P (e^{- N \Lambda S (n, N \Lambda)} \geq e^{- n/2}) \leq e^{n/2} E (e^{- N \Lambda S (n, N \Lambda)}) \vspace*{4pt} \\ \hspace*{40pt} =
      e^{n/2} M_{n, N \Lambda} (- N \Lambda) = (\sqrt{e} / 2)^n. \end{array} $$
 This completes the proof.
\end{proof} \\ \\
 Using the random-walk estimates collected in the previous three lemmas, we can now define implicitly our next two good events from~$E_{m, n}$ and~$F_{m, n}$.
 From now on, we set
\begin{equation}
\label{eq:LN}
  L = 4N / (p - q) \quad \Longleftrightarrow \quad N = (p - q) L/4,
\end{equation}
 so taking~$L$ large is equivalent to taking~$N$ large.
 We first prove that having~$K$ 1s in cluster~$me_1$~(event~$E_{m, n}$) results in~$N$ 1s on both sides of~$me_1$ in less than~$L$ units of time~(event~$F_{m, n}$) with probability close to one when~$L$ is large.
\begin{lemma}
\label{lem:Nlarge-expand}
 Let~$\gamma = 0$ and~$\phi_1 > 1$ and~\eqref{eq:LN}.
 Then, for all~$\ep > 0$, there is an event~$G_{m, n}^2$ measurable with respect to the graphical representation in~$\C_{m, n}$ such that
 $$ E_{m, n} \cap G_{m, n}^1 \cap G_{m, n}^2 \subset F_{m, n} \quad \hbox{and} \quad P (G_{m, n}^2) \geq 1 - \ep / 4 \ \hbox{for all~$L$ large}. $$
\end{lemma}
\begin{proof}
 Letting~$\tau_N^Z = \inf \{t \geq 0 : Z_t = N \}$, and using symmetry, we get
\begin{equation}
\label{eq:Nlarge-expand-1}
  P (F_{m, n} \,| \,E_{m, n} \cap G_{m, n}^1) \geq (P (\tau_N^Z \leq L \,| \,Z_0 = 0))^2 = (P_0 (\tau_N^Z \leq L))^2
\end{equation}
 regardless of the graphical representation outside~$\C_{m, n}$.
 Now, because~$Z_t$ has reflecting boundaries at the origin and~$N$, its embedded Markov chain dominates stochastically the discrete-time random walk~$X_n$ until time~$\tau_N^X$, so it follows from Lemmas~\ref{lem:RW1} and~\ref{lem:RW3} that
\begin{equation}
\label{eq:Nlarge-expand-2}
\begin{array}{rcl}
  P_0 (\tau_N^Z > L) & \n = \n & P_0 (\tau_N^Z > L, \tau_N^X > L/2) + P_0 (\tau_N^Z > L, \tau_N^X \leq L/2) \vspace*{4pt} \\
                     & \n \leq \n & P_0 (\tau_N^X > L/2) + P_0 (T (L/2) > L) \vspace*{4pt} \\
                     & \n \leq \n & \exp (- (p - q)^3 N / 6) + (2/e)^{L/2} \vspace*{4pt} \\
                     & \n = \n & \exp (- (p - q)^4 L / 24) + (2/e)^{L/2} \leq \ep/8 \end{array}
\end{equation}
 for all~$L$ sufficiently large.
 The first inequality above follows from the fact that, if~$Z_t$ takes more than~$L$ units of time but less than~$L/2$ jumps to reach the target~$N$, then the time to perform~$L/2$ jumps must be more than~$L$.
 Combining~\eqref{eq:Nlarge-expand-1}--\eqref{eq:Nlarge-expand-2}, we conclude that
 $$ P (F_{m, n} \,| \,E_{m, n} \cap G_{m, n}^1) \geq (1 - P_0 (\tau_N^Z > L))^2 \geq (1 - \ep/8)^2 \geq 1 - \ep/4, $$
 which shows the existence of~$G_{m, n}^2$ that satisfies the lemma.
 Note that, because the estimates above are true regardless of the graphical representation outside~$\C_{m, n}$, the event~$G_{m, n}^2$ can indeed be made measurable with respect to the graphical representation in the block.
\end{proof} \\ \\
 Next, we prove that having~$N$ 1s on both sides of~$me_1$~(event~$F_{m, n}$) results in at least~$K$ 1s during the next~$3L$ units of time~(events~$E_{m \pm 1, n + 1}$) with high probability for~$L$ large.
\begin{lemma}
\label{lem:Nlarge-persist}
 Let~$\gamma = 0$ and~$\phi_1 > 1$ and~\eqref{eq:LN}.
 Then, for all~$\ep > 0$, there is an event~$G_{m, n}^3$ measurable with respect to the graphical representation in~$\C_{m, n}$ such that
 $$ F_{m, n} \cap G_{m, n}^1 \cap G_{m, n}^3 \subset E_{m - 1, n + 1} \cap E_{m + 1, n + 1} \quad \hbox{and} \quad P (G_{m, n}^3) \geq 1 - \ep / 4 \ \hbox{for all~$L$ large}. $$
\end{lemma}
\begin{proof}
 Letting~$\tau_K^Z = \inf \{t \geq 0 : Z_t = K \}$, we get
\begin{equation}
\label{eq:Nlarge-persist-1}
\begin{array}{l}
  P (E_{m - 1, n + 1} \cap E_{m + 1, n + 1} \,| \,F_{m, n} \cap G_{m, n}^1) \vspace*{4pt} \\ \hspace*{40pt} \geq
 (P (\tau_K^Z > 3L \,| \,Z_0 = N))^2 = (P_N (\tau_K^Z > 3L))^2 \end{array}
\end{equation}
 regardless of the graphical representation outside~$\C_{m, n}$.
 Using that the embedded Markov chain associated with~$Z_t$ is described by the random walk~$X_n$ as long as the process is between the origin and~$N$, it follows from Lemmas~\ref{lem:RW2} and~\ref{lem:RW3} and the union bound that
\begin{equation}
\label{eq:Nlarge-persist-2}
\begin{array}{rcl}
  P_N (\tau_K^Z \leq 3L) & \n = \n & P_N (\tau_K^Z \leq 3L, T (6 \Lambda NL) \leq 3L) + P_N (\tau_K^Z \leq 3L, T (6 \Lambda NL) > 3L) \vspace*{4pt} \\
                         & \n \leq \n & P_N (T (6 \Lambda NL) \leq 3L) + 6 \Lambda NL P_{N - 1} (\tau_K^X < \tau_N^X) \vspace*{4pt} \\
                         & \n \leq \n & (\sqrt{e} / 2)^{6 \Lambda NL} + 6 \Lambda NL \phi_1^{- (N - K - 1)} \vspace*{4pt} \\
                         & \n \leq \n & (\sqrt{e} / 2)^{(p - q) \Lambda L^2} + 2 (p - q) \Lambda L^2 \phi_1^{- (p - q) L/8 + 1} \leq \ep/8 \end{array}
\end{equation}
 for all~$L$ sufficiently large.
 The first inequality above follows from the fact that, if~$Z_t$ reaches~$K$ but jumps less than~$6 \Lambda NL$ times in~$3L$ units of time, then it takes less than~$6 \Lambda NL$ attempts to reach~$K$ before~$N$ starting from~$N - 1$.
 Combining~\eqref{eq:Nlarge-persist-1}--\eqref{eq:Nlarge-persist-2}, we conclude that
 $$ P (E_{m - 1, n + 1} \cap E_{m + 1, n + 1} \,| \,F_{m, n} \cap G_{m, n}^1) \geq (1 - \ep/8)^2 \geq 1 - \ep/4. $$
 Because, as previously, the estimates above are true regardless of the graphical representation outside the space-time block~$\C_{m, n}$, this proves the lemma.
\end{proof} \\ \\
 Using Lemmas~\ref{lem:philarge-clean}, \ref{lem:Nlarge-expand} and~\ref{lem:Nlarge-persist}, we can now prove the theorem following the same approach as in the previous two sections that focus on the external and internal infection rates. \\ \\
\begin{proofof}{Theorem~\ref{th:Nlarge}.b}
 We say that site~$(m, n) \in \lat_+$ is a good site if
 $$ E_{m, n} = \{\hbox{cluster~$m e_1$ contains at least $K$~1s, but no~2s, for all $t\in [2nL, (2n+2)L]$} \} $$
 occurs.
 Take~$\ep < 3^{- 4 (2M + 1)^2}$ with~$M = 2$ since
 $$ \C_{m, n} \cap \C_{m', n'} = \varnothing \quad \hbox{whenever} \quad |m - m'| > 2 \ \ \hbox{or} \ \ |n - n'| > 2. $$
 As previously, to study the process with~$\gamma > 0$ small, we define
 $$ G_{m, n}^4 = \{\hbox{there are no~$\circ$ in~$\C_{m, n}$} \} \quad \hbox{and} \quad G_{m, n} = G_{m, n}^1 \cap G_{m, n}^2 \cap G_{m, n}^3 \cap G_{m, n}^4. $$
 It follows from the first part of Lemmas~\ref{lem:Nlarge-expand} and~\ref{lem:Nlarge-persist}, and the definition of~$G_{m, n}^4$ that
\begin{equation}
\label{eq:Nlarge-inclusion}
\begin{array}{rcl}
  E_{m, n} \cap G_{m, n} & \n = \n & (E_{m, n} \cap G_{m, n}^1 \cap G_{m, n}^2 \cap G_{m, n}^4) \cap G_{m, n}^3 \vspace*{4pt} \\
                         & \n \subset \n & F_{m, n} \cap G_{m, n}^1 \cap G_{m, n}^3 \cap G_{m, n}^4 \subset E_{m - 1, n + 1} \cap E_{m + 1, n + 1}. \end{array}
\end{equation}
 In addition, letting~$L = 4N / (p - q)$ or~$N = (p - q) L/4$ as in~\eqref{eq:LN}, it follows from Lemma~\ref{lem:philarge-clean} and the second part of Lemmas~\ref{lem:Nlarge-expand} and~\ref{lem:Nlarge-persist} that we can find
 $$ L = L (\ep) < \infty \ \ \hbox{such that} \ \ P (G_{m, n}^i) \geq 1 - \ep / 4 \ \ \hbox{for} \ \ i = 1, 2, 3. $$
 Once~$L$ and~$N$ are fixed, using Lemma~\ref{lem:perturb-1} with~$\C_{m, n}$ instead of~$\B$, we can also find
 $$ \gamma_- = \gamma_- (\ep, L) > 0 \ \ \hbox{such that} \ \ P (G_{m, n}^4) \geq 1 - \ep / 4 \ \ \hbox{for all} \ \ \gamma < \gamma_-. $$
 Using the union bound, we deduce that
\begin{equation}
\label{eq:Nlarge-proba}
\textstyle P (G_{m, n}) \geq 1 - \sum_{i = 1, 2, 3, 4} P ((G_{m, n}^i)^c) \geq 1 - \ep.
\end{equation}
 As previously, survival follows from~\cite[Theorem~4.3]{durrett_1995}, Lemma~\ref{lem:perco-1} and~\eqref{eq:Nlarge-inclusion}--\eqref{eq:Nlarge-proba}.
\end{proofof} \\ \\
\begin{proofof}{Theorem~\ref{th:Nlarge}.c}
 This follows the proof of Theorem~\ref{th:Nlarge}.b with the following exceptions.
\begin{enumerate}
\item
 Lemma~\ref{lem:philarge-clean} still holds under the assumptions of the theorem but is no longer needed because the~1s turn instantaneously into~2s in the limiting case~$\gamma = \infty$. \vspace*{4pt}
\item
 The events~$G_{m, n}^2$ and~$G_{m, n}^3$ are defined similarly looking at the~2s rather than the~1s. \vspace*{4pt}
\item
 As before, because we need to do a perturbation at~$\gamma = \infty$ instead of~$\gamma = 0$, the probability of the event~$G_{m, n}^4$ is now controlled by Lemma~\ref{lem:perturb-2} instead of Lemma~\ref{lem:perturb-1}.
\end{enumerate}
 This completes the proof of the theorem.
\end{proofof}


\begin{figure}[t!]
\centering
\scalebox{0.40}{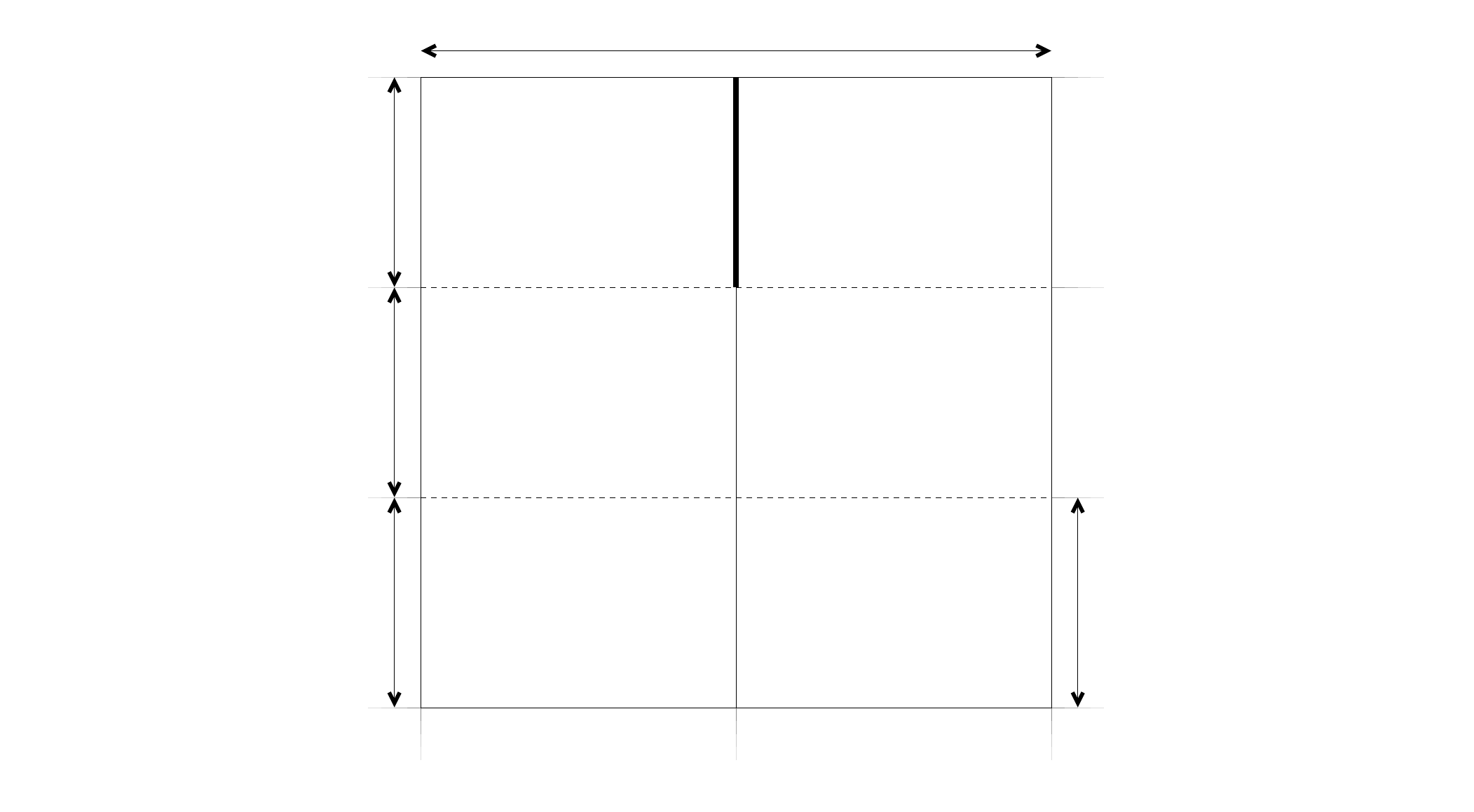}
\caption{\upshape{
 Picture of the blocks and events used in the proof of Theorem~\ref{th:lambdasmall}.b.
 The bold line represents the core~$\A_{m, n}$ and the big block represents the pillar~$\C_{m, n}$.
 All the~2s in the pillar then the~1s at cluster~$m$ die out consecutively in less than~$L$ units of time, and cluster~$m$ remains healthy for at least~$L$ more units of time.}}
\label{fig:decay2}
\end{figure}

\section{Proof of Theorems \ref{th:lambdasmall}.b and~\ref{th:lambdasmall}.c}
\label{sec:lambdasmall}
 To prove Theorem~\ref{th:lambdasmall}.b, the key is to observe that, when~$\gamma = 0$, the process essentially behaves as if there were no~2s, so if in addition~$\lambda_1 = 0$, the process reduces to a collection of independent random walks trapped at zero.
 To prove Theorem~\ref{th:lambdasmall}.c, observe that, when~$\gamma = \infty$, the same conclusion holds after reversing the roles of the~1s and the~2s.
 To turn our arguments into a proof, we first use a block construction to control the rate of decay, and then use two consecutive perturbation arguments, one for the relevant external infection rate and one for the mutation rate.
 To begin with, let~$L$ be a large scale parameter to be fixed later.
 Recalling the second percolation model introduced in Section~\ref{sec:perco}, for all~$(m, n) \in \lat_-$, we consider the space-time blocks
 $$ \begin{array}{rcl}
	\A_{m, n} & \n = \n & \{m \} \times \{1, 2, \ldots, N \} \times  [(n + 2) L, (n + 3) L], \vspace*{4pt} \\
    \C_{m, n} & \n = \n & [m - 1, m + 1]^d \times \{1, 2, \ldots, N \} \times [nL, (n + 3) L], \end{array} $$
 shown in Figure~\ref{fig:decay2}.
 We refer to~$\A_{m, n}$ as the core, and to~$\C_{m, n}$ as the pillar.
 The goal is to prove that, regardless of the configuration and graphical representation outside the pillar, the core is healthy with probability close to one when~$L$ is large.
 We first prove Theorem~\ref{th:lambdasmall}.b and briefly explain how to also deduce Theorem~\ref{th:lambdasmall}.c.
 To study the process in the~$\gamma = 0$ case, the first step is to get rid of the potential~2s at the bottom of the pillar.
 More precisely, letting
 $$ G_{m, n}^1 = \{\forall \,(x, i) \in \pi_{12} (\C_{m, n}), \,\exists \,t \in [nL, (n + 1) L) : \hbox{there is a~$\times$ at~$(x, i, t)$} \}, $$
 we have the following lemma.
\begin{lemma}
\label{lem:lambdasmall-clean}
 For all~$\ep > 0$,~$P (G_{m, n}^1) \geq 1 - \ep / 4$ for all~$L$ large.
\end{lemma}
\begin{proof}
 Using the same approach as in the proof of Lemma~\ref{lem:lambdalarge-clean}, we have
 $$ P (G_{m, n}^1) = \Big(P (\poisson (L) \neq 0) \Big)^{3^d N} = \Big(1 - e^{-L} \Big)^{3^d N}, $$
 which goes to one as~$L$ goes to infinity.
\end{proof} \\ \\
 Note that, on the event~$G_{m, n}^1$ and when the mutation rate~$\gamma = 0$, there are no more symptomatic individuals in the pillar starting from time~$(n + 1) L$ onward.
 Our second good event~$G_{m, n}^2$ is defined implicitly from the configuration event
 $$ F_{m, n} = \{\forall \,i \in \{1, 2, \ldots, N \}, \,\xi_{(n + 2) L} (m, i) = 0 \} $$
 that the asymptomatic individuals at cluster~$m$ also die out in less than~$L$ units of time.
\begin{lemma}
\label{lem:RW}
 Let~$\lambda_1 = \gamma = 0$.
 Then, for all~$\ep > 0$, there exists an event~$G_{m, n}^2$ measurable with respect to the graphical representation in~$\C_{m, n}$ such that
 $$ G_{m, n}^1 \cap G_{m, n}^2 \subset F_{m, n} \quad \hbox{and} \quad P (G_{m, n}^2) \geq 1 - \ep / 4 \ \hbox{for all~$L$ large}. $$
\end{lemma}
\begin{proof}
 Because there are no more~2s in the pillar at time~$(n + 1) L$ and~$\lambda_1 = 0$, the core can no longer be subject to external infections, so the number of~1s in the core evolves according to the continuous-time random walk~$X_t$ with transition rates
 $$ \begin{array}{rclcl}
      j \to j + 1 & \hbox{at rate} & j \phi_1 & \hbox{for} & 0 \leq j < N, \vspace*{4pt} \\
      j \to j - 1 & \hbox{at rate} & j        & \hbox{for} & 0 < j \leq N. \end{array} $$
 This process has two communication classes:
 $C_0 = \{0 \}$ and~$C_+ = \{1, 2, \ldots, N \}$.
 The class~$C_0$ is closed and finite so recurrent, whereas~$C_+$ is open so transient.
 Because the class~$C_+$ is also finite, we deduce that the time to extinction of the process,
 $$ T_0 = \inf \{t \geq (n + 1) L : X_t = 0 \}, $$
 is almost surely finite.
 This and the monotone convergence theorem imply that, for all~$\ep > 0$ and regardless of~$X_{(n + 1) L}$, there exists a large enough~$L$ such that
 $$ \begin{array}{rcl}
      P (F_{m, n} \,| \,G_{m, n}^1) & \n = \n & P (\sum_{i = 1, 2, \ldots, N} \xi_{(n + 2) L} (m, i) = 0 \,| \,G_{m, n}^1) \vspace*{4pt} \\
                                    & \n = \n & P (X_{(n + 2) L} = 0) = P (T_0 - (n + 1) L \leq L) \geq 1 - \ep / 4, \end{array} $$
 which proves the lemma.
 Indeed, the event~$G_{m, n}^2$ can be defined implicitly in terms of Poisson processes via a specific sequence of death marks and internal infections \bul{1} at cluster~$m$ so it only depends on the graphical representation in the block $\C_{m, n}$.
\end{proof} \\ \\
 Using Lemmas~\ref{lem:lambdasmall-clean}--\ref{lem:RW}, we can now prove Theorem~\ref{th:lambdasmall}.b.
 For a visualization of the different steps of the proof, we refer the reader to Figure~\ref{fig:decay2}. \\ \\
\begin{proofof}{Theorem~\ref{th:lambdasmall}.b}
 We say that~$(m, n) \in \lat_-$ is a good site if
 $$ E_{m, n} = \{\A_{m, n} \ \hbox{is healthy} \} = \{\xi_t (x, i) = 0 \ \hbox{for all} \ (x, i, t) \in \A_{m, n} \} $$
 occurs.
 Take~$\ep < (2d + 1)^{- (2M + 1)^{d + 1}}$ with~$M = 3$ since
 $$ \C_{m, n} \cap \C_{m', n'} = \varnothing \quad \hbox{whenever} \quad |m - m'| > 2 \ \ \hbox{or} \ \ |n - n'| > 3. $$
 Note that, when~$\lambda_1 = \gamma = 0$, and on the event~$G_{m, n}^1 \cap G_{m, n}^2 \subset F_{m, n}$, all the clusters in the pillar from time~$(n + 1) L$ onward have no~2s and the core is healthy~(event~$E_{m, n}$).
 To extend this result to~$\lambda_1, \gamma > 0$ small, we first fix~$L$ large so that Lemmas~\ref{lem:lambdasmall-clean}--\ref{lem:RW} hold, and define
 $$ G_{m, n}^3 = \{\hbox{there are no~$\arrow{1}$ in~$\C_{m, n}$} \} \quad \hbox{and} \quad G_{m, n}^4 = \{\hbox{there are no~$\circ$ in~$\C_{m, n}$} \}. $$
 Then, the proof of Lemma~\ref{lem:perturb-1}~(replacing the rate~$\gamma$ with the infection rate~$2d \lambda_1$) and Lemma~\ref{lem:perturb-1} itself with~$M = 3$ and~$T = L$ imply that there exist~$\lambda_-, \gamma_- > 0$ such that
 $$ \lambda_1 < \lambda_- \ \hbox{and} \ \gamma < \gamma_- \quad \Longrightarrow \quad P (G_{m, n}^3) \geq 1 - \ep / 4 \ \ \hbox{and} \ \ P (G_{m, n}^4) \geq 1 - \ep / 4. $$
 In addition, on these events, the process evolves as if~$\lambda_1 = \gamma = 0$ in the pillar, so the core is again fully healthy.
 In summary, calling~$G_{m, n}$ the intersection of the four events~$G_{m, n}^i$,
 $$ \textstyle P (G_{m, n}) \geq 1 - \sum_{i = 1, 2, 3, 4} P ((G_{m, n}^i)^c) \geq 1 - \ep \quad \hbox{and} \quad G_{m, n} \subset E_{m, n}, $$
 where the event~$G_{m ,n}$ only depends on the graphical representation in the block~$\C_{m ,n}$.
 This shows that the set of good sites dominates the set of open sites in a~3-dependent oriented site percolation process on the directed graph~$\vec{\lat}_-$ with parameter~$1 - \ep$.
 Lemma~\ref{lem:perco-2} and our choice of~$\ep$ imply that the closed sites do not percolate, so the sites that are not good, corresponding to cores that might be infected, do not percolate either.
 This shows extinction.
\end{proofof} \\ \\
\begin{proofof}{Theorem~\ref{th:lambdasmall}.c}
 This follows the proof of Theorem~\ref{th:lambdasmall}.b with a few exceptions:
\begin{enumerate}
\item
 Lemma~\ref{lem:lambdasmall-clean} still holds under the assumptions of the theorem but is no longer needed because the~1s turn instantaneously into~2s in the limiting case~$\gamma = \infty$. \vspace*{4pt}
\item
 Lemma~\ref{lem:RW} still holds~(and the proof is identical) under the new assumption~$\lambda_2 = 0$ and~$\gamma = \infty$ after reversing the roles of the~1s and the~2s. \vspace*{4pt}
\item
 The event~$G_{m, n}^3$ is defined similarly but looks at the~$\arrow{2}$ instead of the~$\arrow{1}$. \vspace*{4pt}
\item
 Because we need to do a perturbation at~$\gamma = \infty$ instead of~$\gamma = 0$, the event~$G_{m, n}^4$ is now the event that any potential internal or external infection in the block~$\C_{m ,n}$ is followed by a mutation before the next potential infection in the block.
 In particular, the probability of this event is now controlled by Lemma~\ref{lem:perturb-2} instead of Lemma~\ref{lem:perturb-1}.
\end{enumerate}
 This completes the proof of the theorem.
\end{proofof}


\section{Proof of Theorems~\ref{th:phismall}.b and~\ref{th:phismall}.c}
\label{sec:phismall}
 As for the proof of Theorems~\ref{th:lambdasmall}.b and~\ref{th:lambdasmall}.c, the proof of Theorems~\ref{th:phismall}.b and~\ref{th:phismall}.c relies on a block construction that controls the rate of decay for extreme values of the relevant internal infection rate and the mutation rate, followed by two consecutive perturbation arguments, one for each of those two rates.
 The technical details, however, are quite different.
 For Theorem~\ref{th:phismall}.b, the key is to observe that, when~$\gamma = 0$, the process essentially behaves as if there were no~2s, so if in addition~$\phi_1 = 0$, the process reduces to a subcritical contact process with parameter~$\lambda_1$.
 For Theorem~\ref{th:phismall}.c, in which case we start by imposing~$\gamma = \infty$, the same conclusion holds after reversing the roles of the~1s and the~2s.
 As in the previous section, we let~$L$ be a large scale parameter to be fixed later, but for every site~$(m, n) \in \lat_-$, we now consider the space-time blocks
 $$ \begin{array}{rcl}
    \A_{m, n} & \n = \n & [(2m - 1) L, (2m + 1) L]^d \times \{1, 2, \ldots, N \} \times [(n + 3) L, (n + 4) L], \vspace*{4pt} \\
    \B_{m, n} & \n = \n & [(2m - 2) L, (2m + 2) L]^d \times \{1, 2, \ldots, N \} \times [(n + 2) L, (n + 4) L], \vspace*{4pt} \\
    \C_{m, n} & \n = \n & [(2m - 2) L, (2m + 2) L]^d \times \{1, 2, \ldots, N \} \times [nL, (n + 4) L], \end{array} $$
 shown in Figure~\ref{fig:decay1}.
\begin{figure}[t!]
\centering
\scalebox{0.40}{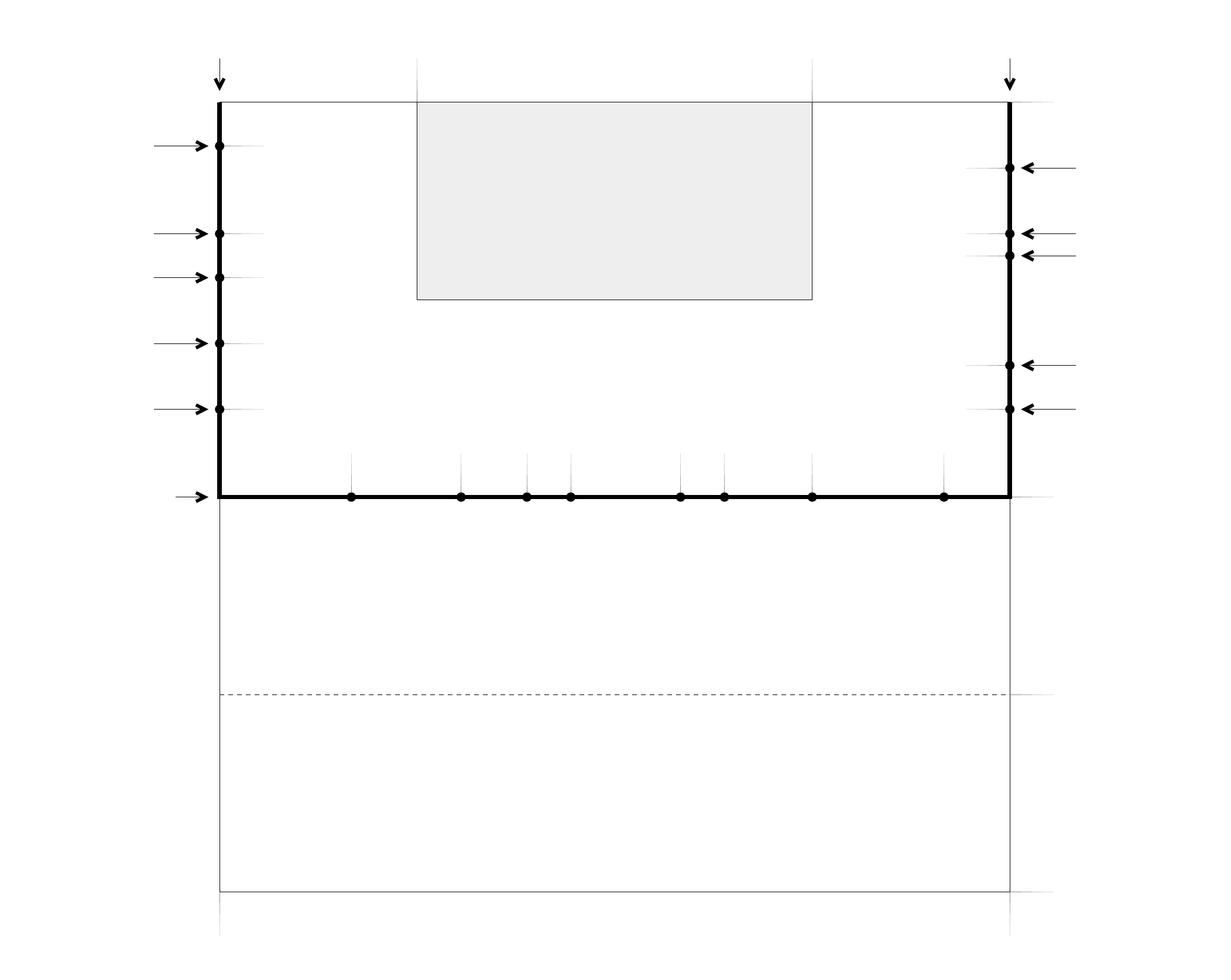}
\caption{\upshape{
 Picture of the blocks and events used in the proof of Theorem~\ref{th:phismall}.b.
 The bold lines represent the bottom and lateral boundaries of~$\B_{m, n}$, and the dots along these lines the entry points of the infection.}}
\label{fig:decay1}
\end{figure}
 The goal is to prove that, regardless of the configuration and the graphical representation outside the larger block~$\C_{m, n}$, with probability close to one when~$L$ is large, the smaller block~$\A_{m, n}$ is healthy.
 First, we show that, with high probability, at each site in~$\C_{m, n}$, there is at least one death mark in the first~$L$ units of time and another death mark in the next~$L$ units of time.
 This is done in the next lemma, looking at the events
 $$ \begin{array}{rcl}
      G_{m, n}^1 & \n = \n & \{\forall \,(x, i) \in \pi_{12} (\C_{m, n}), \,\exists \,t \in [nL, (n + 1) L) : \hbox{there is a~$\times$ at~$(x, i, t)$} \}, \vspace*{4pt} \\
      G_{m, n}^2 & \n = \n & \{\forall \,(x, i) \in \pi_{12} (\C_{m, n}), \,\exists \,t \in [(n + 1) L, (n + 2) L) : \hbox{there is a~$\times$ at~$(x, i, t)$} \}. \end{array} $$
\begin{lemma}
\label{lem:phismall-clean}
 For all~$\ep > 0$,~$P (G_{m, n}^1) = P (G_{m, n}^2) \geq 1 - \ep / 5$ for all~$L$ large.
\end{lemma}
\begin{proof}
 Using the same approach as in the proof of Lemma~\ref{lem:lambdalarge-clean}, we have
 $$ P (G_{m, n}^1) = P (G_{m, n}^2) = \Big(P (\poisson (L) \neq 0) \Big)^{(4L + 1)^d N} = \Big(1 - e^{-L} \Big)^{(4L + 1)^d N}, $$
 which goes to one as~$L$ goes to infinity.
\end{proof} \\ \\
 It is important to note for later that the events~$G_{m, n}^1$ and~$G_{m, n}^2$ act as two consecutive sweeps, one to get rid of one type of infected individuals, and one to get rid of possible multiple infections in each cluster.
 More precisely, we have the following.
\begin{itemize}
\item
 When~$\gamma = 0$, the~2s can't be created, so the first event guarantees that there are no~2s in the block at time~$(n + 1) L$.
 If in addition~$\phi_1 = 0$, the~1s are not infectious in their cluster, so the second event guarantees that the clusters contain at most one~1 at time~$(n + 2) L$, and from this time onward, the process reduces to a contact process with parameter~$\lambda_1$. \vspace*{4pt}
\item
 When~$\gamma = \infty$, the~1s turn instantaneously into~2s.
 If in addition~$\phi_2 = 0$, the~2s are not infectious in their cluster, so the first event guarantees that the clusters contain at most one~2 at time~$(n + 1) L$, and from this time onward, the process reduces to a contact process with parameter~$\lambda_2$.
 In this context, the second event becomes unimportant.
\end{itemize}
 Next, we prove that none of the infection paths originating from the bottom and lateral boundaries of the block~$\B_{m, n}$ reaches the block~$\A_{m, n}$.
 We first control the number of such paths.
 We call entry points the points along the boundary the infection paths start from.
 The number of entry points at the bottom boundary of~$\B_{m, n}$ is bounded by
\begin{equation}
\label{eq:bottom}
\card (\pi_1 (\B_{m, n})) = (4L + 1)^d.
\end{equation}
 Now, a point~$(x, t)$ on the lateral boundary is a potential entry point only if there is an arrow from outside the block~$\B_{m, n}$ pointing at~$(x, 1, t)$.
 The next lemma controls the number~$Y$ of entry points on the lateral boundary, corresponding to the number of infection paths.
\begin{lemma}
\label{lem:lateral}
 Let~$\lambda = 2LN ((4L + 1)^d - (4L - 1)^d)(\lambda_1 + \lambda_2)$. Then,
\begin{equation}
\label{eq:lateral}
  P (Y > 2 \lambda) \leq (4/e)^{- \lambda} \quad \hbox{for all} \quad L \ \hbox{large}.
\end{equation}
\end{lemma}
\begin{proof}
 The lateral boundary of the block~$\B_{m, n}$ consists of
 $$ N ((4L + 1)^d - (4L - 1)^d) \ \hbox{sites} \times 2L \ \hbox{units of time}. $$
 In addition, arrows originate from each site at rate~$\lambda_1 + \lambda_2$.
 This shows that the number of entry points is dominated by the Poisson random variable with parameter~$\lambda$.
 Recalling the probability-generating function~\eqref{eq:PGF}, and using Markov's inequality, we deduce that
 $$ P (Y > 2 \lambda) = P (2^Y > 4^{\lambda}) \leq 4^{- \lambda} E (2^Y) = 4^{- \lambda} \phi_Y (2) = (4 / e)^{- \lambda}. $$
 This completes the proof of the lemma.
\end{proof} \\ \\
 Now that we control the number of entry points/infection paths, the next step is to also control the length of the infection paths individually.
 We focus on the case~$\phi_1 = \gamma = 0$ and~$\lambda_1 < \lambda_c$ to prove Theorem~\ref{th:phismall}.b, but our argument also applies to the context of Theorem~\ref{th:phismall}.c by exchanging the roles of the~1s and the~2s and taking~$\gamma = \infty$.
 For each entry point~$(x, t)$, we let~$\xi^{(x, t)}$ be the contact process with parameter~$\lambda_1$ starting with a single~1 at site~$x$ at time~$t$ constructed using the same~1-arrows and death marks as~$\xi$ but modified so that births outside~$\B_{m, n}$ are suppressed so the process only depends on the graphical representation inside~$\B_{m, n}$.
\begin{lemma}
\label{lem:decay}
 Let~$\lambda_1 < \lambda_c$.
 Then, there exists~$\delta > 0$ such that
\begin{equation}
\label{eq:decay}
  P (\xi^{(x, t)} \cap \A_{m, n} \neq \varnothing) \leq e^{- \delta L} \quad \hbox{for all} \quad L \ \hbox{large}.
\end{equation}
\end{lemma}
\begin{proof}
 Letting~$\eta$ be the contact process with parameter~$\lambda_1 < \lambda_c$ starting with a single~1 at site~0 at time~0, it follows from the exponential decay~\cite{bezuidenhout_grimmett_1991} in space and time of the subcritical contact process that there are constants~$\alpha, \beta > 0$ such that
 $$ P (\eta \not \subset [-r, r]^d \times \R_+) \leq e^{- \alpha r} \quad \hbox{and} \quad P (\eta \not \subset \Z^d \times [0, s]) \leq e^{- \beta s} \quad \hbox{for all} \quad r, s > 0. $$
 Using also that~$\xi^{(x, t)}$ is dominated stochastically by~$(x, t) + \eta$, and that the entry point~$(x, t)$ is at space-time distance at least~$L$ from~$\A_{m, n}$, we deduce that
 $$ \begin{array}{rcl}
      P (\xi^{(x, t)} \cap \A_{m, n} \neq \varnothing) & \n \leq \n &
      P (\eta \not \subset [-L, L]^d \times [0, L]) \vspace*{4pt} \\ & \n \leq \n &
      P (\eta \not \subset [-L, L]^d \times \R_+) + P (\eta \not \subset \Z^d \times [0, L]) \vspace*{4pt} \\ & \n \leq \n &
      e^{- \alpha L} + e^{- \beta L} \leq e^{- \delta L}, \end{array} $$
 for~$\delta = (\alpha \wedge \beta) / 2 > 0$ and all~$L$ sufficiently large.
\end{proof} \\ \\
 We can now deduce that our third good event,
 $$ G_{m, n}^3 = \{\xi^{(x, t)} \cap \A_{m, n} = \varnothing \ \hbox{for every entry point} \ (x, t) \}, $$
 occurs with probability close to one for~$L$ large.
\begin{lemma}
\label{lem:healthy}
 For all~$\ep > 0$,~$P (G_{m, n}^3) \geq 1 - \ep / 5$ for all~$L$ large.
\end{lemma}
\begin{proof}
 Conditioning on whether the number~$Y$ of entry points on the lateral boundary is larger or smaller than~$2 \lambda$, and using~\eqref{eq:bottom}--\eqref{eq:decay}, we get
 $$ \begin{array}{rcl}
      P ((G_{m, n}^3)^c) & \n = \n &
      P ((G_{m, n}^3)^c \,| \,Y > 2 \lambda) P (Y > 2 \lambda) + P ((G_{m, n}^3)^c \,| \,Y \leq 2 \lambda) P (Y \leq 2 \lambda) \vspace*{4pt} \\ & \n \leq \n &
      P (Y > 2 \lambda) + P ((G_{m, n}^3)^c \,| \,Y \leq 2 \lambda) \vspace*{4pt} \\ & \n = \n &
      P (Y > 2 \lambda) + P (\xi^{(x, t)} \cap \A_{m, n} \neq \varnothing \ \hbox{for some} \ (x, t) \,| \,Y \leq 2 \lambda) \vspace*{4pt} \\ & \n \leq \n &
        (4/e)^{- \lambda} + ((4L + 1)^d + 2 \lambda) e^{- \delta L}. \end{array} $$
 Because this goes to zero as~$L$ goes to infinity, the lemma follows.
\end{proof} \\ \\
 Using Lemmas~\ref{lem:phismall-clean} and~\ref{lem:healthy}, we are now ready to prove Theorem~\ref{th:phismall}.b.
 For a visualization of the different steps of the proof and key results collected so far, we refer the reader to Figure~\ref{fig:decay1}. \\ \\
\begin{proofof}{Theorem~\ref{th:phismall}.b}
 We again say that~$(m, n) \in \lat_-$ is a good site if
 $$ E_{m, n} = \{\A_{m, n} \ \hbox{is healthy} \} = \{\xi_t (x, i) = 0 \ \hbox{for all} \ (x, i, t) \in \A_{m, n} \} $$
 occurs and take~$\ep < (2d + 1)^{- (2M + 1)^{d + 1}}$ with~$M = 4$ since
 $$ \C_{m, n} \cap \C_{m', n'} = \varnothing \quad \hbox{whenever} \quad |m - m'| > 2 \ \ \hbox{or} \ \ |n - n'| > 4. $$
 Assume first that~$\phi_1 = \gamma = 0$.
 Then, on the event~$G_{m, n}^1 \cap G_{m, n}^2$, all the social clusters in the space-time block~$\B_{m, n}$ have at most one~1 and no~2s so the process~$\xi_t$ reduces to the basic contact process with parameter~$\lambda_1$ in this block.
 In particular, the process is dominated by the superposition of all the non-interacting contact processes~$\xi^{(x, t)}$ that allows multiple infections per site:
\begin{equation}
\label{eq:domination}
\textstyle \sum_i \,\xi_s (y, i) = \xi_s (y, 1) \leq \sum_{(x, t)} \,\xi_s^{(x, t)} (y) \quad \hbox{for all} \quad (y, s) \in \pi_{13} (\B_{m, n}),
\end{equation}
 where the second sum is over all the entry points~$(x, t)$ on the boundary of~$\B_{m, n}$.
 In particular, if the event~$G_{m, n}^3$ also occurs, then the block~$\A_{m, n}$ is fully healthy.
 To extend this result to~$\phi_1, \gamma > 0$ small, we first fix~$L$ large so that Lemmas~\ref{lem:phismall-clean} and~\ref{lem:healthy} hold, and define
 $$ G_{m, n}^4 = \{\hbox{there are no~\bul{1} in~$\C_{m, n}$} \} \quad \hbox{and} \quad G_{m, n}^5 = \{\hbox{there are no~$\circ$ in~$\C_{m, n}$} \}. $$
 Then, the proof of Lemma~\ref{lem:perturb-1}~(replacing the rate~$\gamma$ with the infection rate~$\phi_1$) and Lemma~\ref{lem:perturb-1} itself with~$M = 4$ and~$T = L$ imply that there exist~$\phi_-, \gamma_- > 0$ such that
 $$ \phi_1 < \phi_- \ \hbox{and} \ \gamma < \gamma_- \quad \Longrightarrow \quad P (G_{m, n}^4) \geq 1 - \ep / 5 \ \hbox{and} \ P (G_{m, n}^5) \geq 1 - \ep / 5. $$
 In addition, on these events, the process evolves as if~$\phi_1 = \gamma = 0$ in the block~$\C_{m ,n}$, so the block~$\A_{m ,n}$ is again fully healthy.
 In summary, calling~$G_{m, n}$ the intersection of the five events~$G_{m, n}^i$,
 $$ \textstyle P (G_{m, n}) \geq 1 - \sum_{i = 1, 2, 3, 4, 5} P ((G_{m, n}^i)^c) \geq 1 - \ep \quad \hbox{and} \quad G_{m, n} \subset E_{m, n}, $$
 where the event~$G_{m ,n}$ only depends on the graphical representation in the block~$\C_{m ,n}$.
 This shows that the set of good sites dominates the set of open sites in a~4-dependent oriented site percolation process on the directed graph~$\vec{\lat}_-$ with parameter~$1 - \ep$.
 Lemma~\ref{lem:perco-2} and our choice of~$\ep$ imply that the closed sites do not percolate, so the sites that are not good, corresponding to~$\A$-blocks that might be infected, do not percolate either.
 This shows extinction.
\end{proofof} \\ \\
\begin{proofof}{Theorem~\ref{th:phismall}.c}
 This follows the proof of Theorem~\ref{th:phismall}.b with a few exceptions.
\begin{enumerate}
\item
 Lemma~\ref{lem:phismall-clean} still holds under the assumptions of the theorem but, as previously explained, the second good event~$G_{m, n}^2$ is no longer needed. \vspace*{4pt}
\item
 The domination~\eqref{eq:domination} still holds when~$\phi_2 = 0$ and~$\gamma = \infty$, but the process~$\xi^{(x, t)}$ is now defined as the contact process with parameter~$\lambda_2$, constructed using the same~2-arrows and death marks as~$\xi$, that keeps track of the~2s rather than the~1s. \vspace*{4pt}
\item
 The event~$G_{m, n}^4$ is defined similarly but looks at the~\bul{2} instead of the~\bul{1}. \vspace*{4pt}
\item
 The event~$G_{m, n}^5$ should now be defined as the event~$G_{m, n}^4$ in the proof of Theorem~\ref{th:phismall}.c, and, as explained earlier, because the perturbation is now performed at~$\gamma = \infty$ instead of~$\gamma = 0$, the probability of this event is now controlled by Lemma~\ref{lem:perturb-2} instead of Lemma~\ref{lem:perturb-1}.
\end{enumerate}
 This completes the proof of the theorem.
\end{proofof}


\end{document}